\documentclass[a4paper,12pt,oneside,headsepline]{scrartcl}

\usepackage{ifdraft}

\usepackage[utf8]{inputenc}
\usepackage[T1]{fontenc}

\usepackage{lmodern}
\usepackage{fourier}
\usepackage{amssymb, amsfonts}

\usepackage{mathrsfs} %
\usepackage{dsfont}
\usepackage[usenames,dvipsnames]{xcolor}%
\usepackage{lettrine}
\usepackage{url}

\usepackage{stmaryrd}

\usepackage[english]{babel}

\usepackage[draft=false]{scrlayer-scrpage}
\usepackage{wrapfig}
\usepackage{footmisc}
\usepackage{needspace}

\usepackage{amsmath}
\usepackage[thmmarks, amsmath, amsthm]{ntheorem}

\usepackage[style=alphabetic,giveninits,url=false,sortcites=true,maxbibnames=99]{biblatex}
\bibliography{bibliography.bib}

\ifdraft{{}}
  {\usepackage[pdftex,
               pdfborder={0 0 0}, colorlinks=true,
               linkcolor=BrickRed, citecolor=ForestGreen, urlcolor=RoyalBlue]{hyperref}}

\ifdraft{\usepackage[colorinlistoftodos]{todonotes}}{}

\usepackage{booktabs}
\usepackage[inline]{enumitem}
\usepackage{float}
\usepackage{graphicx}
\usepackage{multicol}
\usepackage{tikz}
\usetikzlibrary{matrix,arrows,positioning}

\ifdraft{\date{\today}}{\date{}}

\KOMAoptions{DIV=13}

\clearscrheadings
\automark[section]{subsection}
\renewcommand{\sectionmark}[1]{\markright{Section~\thesection}}
\renewcommand{\sectionmark}[1]{\markright{\thesection\;#1}}

\lohead{{\small \headertitle}}
\rohead{{\small \headerauthors}}

\newenvironment{plainfootnotes}{%
  \deffootnote[0em]{0em}{0em}{}
}{%
  \deffootnote[1em]{1.5em}{1em}{\textsuperscript{\thefootnotemark}}
}

\RedeclareSectionCommand[%
  font=\Large\sffamily\bfseries,%
  beforeskip=1\baselineskip,%
  afterskip=0.5\baselineskip,%
  indent=0em%
  ]{section}

\RedeclareSectionCommands[%
  font=\normalfont\bfseries,%
  afterskip=-1em%
  ]{subsection,subsubsection}

\RedeclareSectionCommands[%
  font=\normalfont\itshape,%
  afterskip=-1em,%
  indent=0pt,%
  ]{paragraph}

\AtEveryBibitem{\clearfield{isbn}}
\AtEveryBibitem{\clearlist{language}}
\AtEveryBibitem{\clearfield{pages}}

\newenvironment{enumeratearabic}{%
\begin{enumerate}[label=(\arabic*), leftmargin=0pt,labelindent=2em,itemindent=!]
}{%
\end{enumerate}
}

\newenvironment{enumerateroman}{%
\begin{enumerate}[label=(\roman*), leftmargin=0pt,labelindent=2em,itemindent=!]
}{%
\end{enumerate}
}

\newenvironment{enumeratearabic*}{%
\begin{enumerate*}[label=(\arabic*)] %
}{%
\end{enumerate*}
}

\newenvironment{enumerateroman*}{%
\begin{enumerate*}[label=(\roman*)] %
}{%
\end{enumerate*}
}

\numberwithin{equation}{section}

\newtheorem{theoremcounter}{theoremcounter}[section]

\theoremstyle{plain}

\newtheorem{lemma}[theoremcounter]{Lemma}
\newtheorem{proposition}[theoremcounter]{Proposition}

\theoremnumbering{Roman}
\newtheorem{maintheoremcounter}{maintheoremcounter}

\newtheorem{maintheorem}[maintheoremcounter]{Theorem}

\theoremstyle{definition}

\newtheorem{definition}[theoremcounter]{Definition}

\theoremstyle{remark}

\newtheorem{remark}[theoremcounter]{Remark}
\newtheorem{remarks}[theoremcounter]{Remarks}

\newtheorem*{mainremark}{Remark}
\newtheorem*{mainremarks}{Remarks}

\newtheorem*{remarkcomputation}{Computation}

\let\cal\undefined

\ifdraft{%
\newcommand{\texpdf}[2]{#1}
}{%
\newcommand{\texpdf}[2]{\texorpdfstring{#1}{#2}}
}

\newcommand{\tx}{\ensuremath{\text}}

\newcommand{\tbf}{\bfseries}

\newcommand{\thdash}{\nbd th}

\newcommand{\nbd}{\nobreakdash-\hspace{0pt}}

\newcommand{\bboard}{\ensuremath{\mathbb}}
\newcommand{\cal}{\ensuremath{\mathcal}}
\renewcommand{\frak}{\ensuremath{\mathfrak}}

\newcommand{\bbJ}{\ensuremath{\bboard J}}

\newcommand{\bbL}{\ensuremath{\bboard L}}

\newcommand{\cF}{\ensuremath{\cal{F}}}

\newcommand{\cJ}{\ensuremath{\cal{J}}}

\newcommand{\cM}{\ensuremath{\cal{M}}}

\newcommand{\cO}{\ensuremath{\cal{O}}}

\newcommand{\frake}{\ensuremath{\frak{e}}}
\newcommand{\frakf}{\ensuremath{\frak{f}}}

\newcommand{\rmf}{\ensuremath{\mathrm{f}}}

\newcommand{\rms}{\ensuremath{\mathrm{s}}}
\newcommand{\rmt}{\ensuremath{\mathrm{t}}}

\newcommand{\rmA}{\ensuremath{\mathrm{A}}}

\newcommand{\rmC}{\ensuremath{\mathrm{C}}}

\newcommand{\rmG}{\ensuremath{\mathrm{G}}}

\newcommand{\rmJ}{\ensuremath{\mathrm{J}}}

\newcommand{\rmL}{\ensuremath{\mathrm{L}}}
\newcommand{\rmM}{\ensuremath{\mathrm{M}}}

\newcommand{\rmR}{\ensuremath{\mathrm{R}}}

\newcommand{\td}{\tilde}
\newcommand{\wtd}{\widetilde}
\newcommand{\ov}{\overline}

\newcommand{\wht}{\widehat}

\newcommand{\ra}{\ensuremath{\rightarrow}}

\newcommand{\lra}{\ensuremath{\longrightarrow}}

\newcommand{\mto}{\ensuremath{\mapsto}}
\newcommand{\lmto}{\ensuremath{\longmapsto}}

\newcommand{\amid}{\ensuremath{\mathop{\mid}}}

\newcommand{\ZZ}{\ensuremath{\mathbb{Z}}}
\newcommand{\QQ}{\ensuremath{\mathbb{Q}}}
\newcommand{\RR}{\ensuremath{\mathbb{R}}}
\newcommand{\CC}{\ensuremath{\mathbb{C}}}

\renewcommand{\Re}{\ensuremath{\mathrm{Re}}}
\renewcommand{\Im}{\ensuremath{\mathrm{Im}}}

\newcommand{\isdiv}{\amid}

\renewcommand{\pmod}[1]{\ensuremath{\;(\mathrm{mod}\, #1)}}

\newcommand{\sgn}{\ensuremath{\mathrm{sgn}}}

\newenvironment{psmatrix}{\left(\begin{smallmatrix}}{\end{smallmatrix}\right)}

\newcommand{\Mat}[1]{\ensuremath{\mathrm{Mat}_{#1}}}
\newcommand{\MatT}[1]{\ensuremath{\mathrm{Mat}^\rmt_{#1}}}
\newcommand{\GL}[1]{\ensuremath{\mathrm{GL}_{#1}}}

\newcommand{\SL}[1]{\ensuremath{\mathrm{SL}_{#1}}}

\newcommand{\Sp}[1]{\ensuremath{\mathrm{Sp}_{#1}}}

\newcommand{\SO}[1]{\ensuremath{\mathrm{SO}_{#1}}}
\newcommand{\U}[1]{\ensuremath{\mathrm{U}_{#1}}}

\newcommand{\T}{\ensuremath{\rmt}}
\newcommand{\rT}{\ensuremath{\,{}^\T\!}}

\newcommand{\tr}{\ensuremath{\mathrm{tr}}}
\newcommand{\diag}{\ensuremath{\mathrm{diag}}}

\renewcommand{\det}{\ensuremath{\mathrm{det}}}

\newcommand{\lspan}{\ensuremath{\mathop{\mathrm{span}}}}
\newcommand{\HS}{\mathbb{H}}

\renewcommand{\ss}{ss}

\newcommand{\sk}{\ensuremath{{\mathrm{skew}}}}

\newcommand{\skmd}{\ensuremath{{\sk,\mathrm{md}}}}

\newcommand{\cJsk}{\ensuremath{{\cJ\sk}}}
\newcommand{\Jsk}{\ensuremath{{\rmJ\sk}}}
\newcommand{\rmEll}{\ensuremath{\mathrm{Ell}}}

\newcommand{\trace}{\ensuremath{{\mathrm{trace}}}}

\newcommand{\HJ}{\HS^\rmJ}

\newcommand{\GSp}[1]{\ensuremath{\mathrm{GSp}_{#1}}}
\newcommand{\Ga}{\ensuremath{\Gamma}}
\newcommand{\ga}{\ensuremath{\gamma}}

\newcommand{\GaJ}{\ensuremath{\Gamma^\rmJ}}
\newcommand{\GJ}{\ensuremath{\rmG^\rmJ}}
\newcommand{\gaJ}{\ensuremath{\gamma^\rmJ}}

\newcommand{\MAsk}{\ensuremath{\mathrm{Maass}^\sk}}

\newcommand{\rot}{\ensuremath{\mathrm{rot}}}
\newcommand{\trans}{\ensuremath{\mathrm{trans}}}
\newcommand{\up}{\ensuremath{\mathrm{up}}}
\newcommand{\down}{\ensuremath{\mathrm{down}}}

\newcommand{\hol}{\ensuremath{\mathrm{hol}}}
\newcommand{\indef}{\ensuremath{\mathrm{indef}}}

\newcommand{\Klim}{\ensuremath{\mathrm{Klim}}}

\newcommand{\xb}{\ensuremath{\breve{x}}}
\newcommand{\yb}{\ensuremath{\breve{y}}}
\newcommand{\taub}{\ensuremath{\breve{\tau}}}
\newcommand{\nb}{\ensuremath{\breve{n}}}

\newcommand{\Li}[1]{\ensuremath{\mathrm{Li}_{#1}}}

\newcommand{\headertitle}{{\normalfont%
The skew-Maa\ss\ lift
}}
\newcommand{\headerauthors}{%
M.~Raum, O.~K.~Richter
}

\begin{document}

\begin{plainfootnotes}
\begin{flushleft}
{\fontfamily{lms}\sffamily
\hspace{20pt}{\huge%
The skew-Maa\ss\ lift~I%
}%
}
\\{\fontfamily{lms}\sffamily
  \hspace{20pt}%
  The case of harmonic Maa\ss-Jacobi forms
}
\\[.6em]\hspace{20pt}{\large%
Martin Raum%
\footnote{The first author was partially supported by Vetenskapsr\aa det Grant~2015-04139.} and
Olav K.\ Richter%
\footnote{The second author was partially supported by Simons Foundation Grant \#412655.}
}
\\[1.2em]
\end{flushleft}
\end{plainfootnotes}

\thispagestyle{scrplain}

{\small
\noindent
{\tbf Abstract:}
The classical Maa\ss\ lift is a map from holomorphic Jacobi forms to holomorphic scalar-valued Siegel modular forms.  Automorphic representation theory predicts a non-holomorphic and vector-valued analogue for Hecke eigenforms. This paper is the first part of a series of papers.  In this series of papers, we provide an explicit construction of the non-holomorphic Maa\ss\ lift that is linear and also applies to non-eigenforms.  In this first part, we develop new techniques to study Fourier series expansions of Siegel modular forms, which allow us to construct a Maa\ss\ lift from harmonic Maa\ss-Jacobi forms to scalar-valued Maa\ss-Siegel forms.
\\[.35em]
\textsf{\textbf{%
Maass lift of harmonic Maa\ss-Jacobi forms%
\hspace{0.3em}{\tiny$\blacksquare$}\hspace{0.3em}%
Saito-Kurokawa lift%
\hspace{0.3em}{\tiny$\blacksquare$}\hspace{0.3em}%
Kohnen limit process%
\hspace{0.3em}{\tiny$\blacksquare$}\hspace{0.3em}%
real-analytic Siegel modular forms%
\hspace{0.3em}{\tiny$\blacksquare$}\hspace{0.3em}%
Maa\ss-Siegel forms
}}
\\[0.15em]
\noindent
\textsf{\textbf{%
MSC Primary:
11F46%
\hspace{0.3em}{\tiny$\blacksquare$}\hspace{0.3em}%
MSC Secondary:
11F30, 11F50
}}
}

\vspace{-2.5em}
\renewcommand{\contentsname}{}
\setcounter{tocdepth}{2}
\tableofcontents
\vspace{.5em}

\addcontentsline{toc}{section}{Introduction}
\markright{Introduction}
\lettrine[lines=2,nindent=.2em]{\tbf S}{aito} and Kurokawa~\cite{kurokawa-1978} independently conjectured the existence of a mapping from elliptic modular forms of weight~$2k-2$ to degree~$2$ Siegel modular forms of weight~$k$. This conjecture was resolved in a series of papers by Maa\ss~\cite{maass-1979a,maass-1979b,maass-1979c}, Andrianov~\cite{andrianov-1979}, and Zagier~\cite{zagier-1981}.

Piatetski-Shapiro~\cite{piatetski-shapiro-1983} reinterpreted the Maa\ss\ lift and provided a representation theoretic construction. His work (see also~\cite{miyazaki-2004,schmidt-2005} and Example~5.13 of~\cite{bruinier-2002a}) shows that holomorphic elliptic modular forms give rise to two Saito-Kurokawa lifts: One of them is scalar-valued and holomorphic. The other one, which is vector-valued and not holomorphic, is the focus of this paper and its sequel of papers.

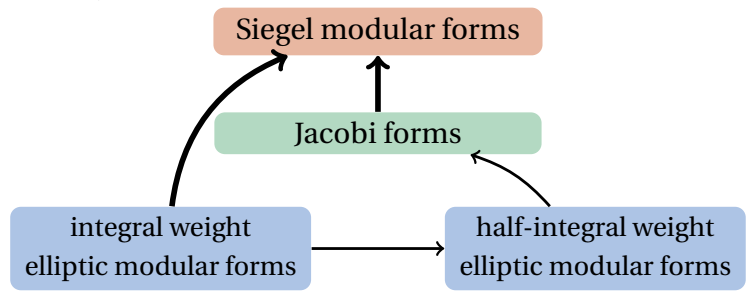
\begin{wrapfigure}{r}{0.6\textwidth}
\caption{The Saito-Kurokawa lift and Maa\ss\ lift}
\label{fig:saito-kurokawa-maass-lift}
\begin{tikzpicture}

\node[text centered, fill=ForestGreen!30, rectangle, rounded corners, inner sep=0.3em]
(jacobi-forms)
{\begin{minipage}{11em}\begin{center}
Jacobi forms
\end{center}\end{minipage}};

\node[below=3em of jacobi-forms] (dummybelow) {};

\node[left=2em of dummybelow, fill=RoyalBlue!30, rectangle, rounded corners]
(elliptic-integral)
{\begin{minipage}{10em}\begin{center}
\small%
integral weight\\
elliptic modular forms
\end{center}\end{minipage}};

\node[right=2em of dummybelow, fill=RoyalBlue!30, rectangle, rounded corners]
(elliptic-half-integral)
{\begin{minipage}{10em}\begin{center}
\small%
half-integral weight\\
elliptic modular forms
\end{center}\end{minipage}};

\node[above=2em of jacobi-forms, fill=BrickRed!30, rectangle, rounded corners]
(siegel-modular)
{\begin{minipage}{11em}\begin{center}
Siegel modular forms
\end{center}\end{minipage}};

\path[line width=0.10em]
(elliptic-half-integral) edge[->, bend right=15] (jacobi-forms)
(elliptic-integral) edge[->] (elliptic-half-integral);

\path[line width=0.2em]
(jacobi-forms) edge[->] (siegel-modular)
(elliptic-integral) edge[->, bend left=30] (siegel-modular);
\end{tikzpicture}
\end{wrapfigure}
Recall that the Saito-Kurokawa conjecture was proved by combining two isomorphisms and the Maa\ss\ lift, which we highlight in Figure~\ref{fig:saito-kurokawa-maass-lift}. Integral weight and half-integral weight elliptic modular forms are connected via the Shimura-Shintani lifts. The link between half-integral weight elliptic modular forms and Jacobi forms is given by the theta-decomposition of Jacobi forms. Finally, the Maa\ss\ lift sends Jacobi forms to Siegel modular forms. This yields the Saito-Kurokawa lift from elliptic modular forms of integral weight to Siegel modular forms.

Let us introduce some notation to recall the holomorphic Maa\ss\ lift in more detail. The Siegel upper half space~$\HS^{(2)}$ of degree~$2$ consists of complex matrices~$Z = \begin{psmatrix} \tau & z \\ z & \tau' \end{psmatrix}$ with positive definite imaginary part. The holomorphic Maa\ss\ lift maps a Jacobi form~$\phi$ of weight~$k$ and  index~$1$ to the Siegel modular form
\begin{gather}
\label{eq:introduction:maass-lift}
\sum_{m = 0}^\infty
\big( \phi \big|_{k,1}\, V_m \big) \,
\exp(2 \pi i\, m \tau')
\tx{,}
\end{gather}
where $V_m$ is a Jacobi-Hecke operator that changes the Jacobi index from~$1$ to~$m$, if $m \ge 1$, and $V_0$ is the Hecke-like operator on page~43 of~\cite{eichler-zagier-1985}. The series in~\eqref{eq:introduction:maass-lift} is a priori a formal series. Its convergence follows from a formula for the Fourier series coefficients of $\phi | V_m$ that is polynomial in the Fourier indices. Modularity of~\eqref{eq:introduction:maass-lift} is a consequence of two observations: Firstly, the product $\big( \phi | V_m \big) \exp(2 \pi i\, m \tau')$ transforms like a Siegel modular form under the embedded Jacobi group. Secondly, a formula for the Fourier series coefficients of $\phi | V_m$ implies that the series in~\eqref{eq:introduction:maass-lift} is invariant when interchanging $\tau$ and $\tau'$. The corresponding transformation in the Siegel modular group together with the embedded Jacobi group generate the full Siegel modular group.
\vspace{.5\baselineskip}

The goal of this paper and its sequence of papers is to extend the construction of the holomorphic Maa\ss\ lift to the non-holomorphic setting. We call this new lift the \emph{skew-Maa\ss\ lift}. In this first part, we construct the skew-Maa\ss\ lift from harmonic Maa\ss-Jacobi forms to skew-harmonic Siegel modular forms, where throughout this paper, $k > 3$ is an odd integer.

In order to explain how we arrive at the analogue of~\eqref{eq:introduction:maass-lift}, we contrast Fourier-Jacobi expansions of holomorphic Siegel modular forms and general real-analytic Siegel modular forms:
\begin{gather}
\label{eq:introduction:fourier-jacobi}
\sum_{m = 0}^\infty
\phi_m(\tau, z)\, \exp(2 \pi i\, m \tau')
\tx{,}
\qquad
\sum_{m \in \ZZ} \underbrace{\wht\phi_m(\tau, z, y') \exp(2 \pi i\, m x')}_{=:\, \wtd\phi_m(\tau, z, \tau')}
\tx{,}
\end{gather}
where $\tau' = x' + iy'$. We call $\wht\phi_m$ the $m$\thdash\ Fourier-Jacobi coefficient and $\wtd\phi_m$ the $m$\thdash\ Fourier-Jacobi term. In the case of holomorphic Siegel modular forms the Fourier-Jacobi coefficients $\wht\phi_m(\tau, z, y')$ factor as products $\phi_m(\tau, z) \exp(-2 \pi m y')$, and it is equally common to refer to $\phi_m$ as a (holomorphic) Fourier-Jacobi coefficient. The holomorphic Fourier-Jacobi coefficients $\phi_m$ are Jacobi forms of index~$m$. In particular, if $m < 0$, then $\phi_m = 0$. In comparison, Fourier-Jacobi coefficients~$\wtd\phi_m$ of real-analytic Siegel modular forms are in general not zero for $m < 0$.

Skew Maa\ss\ lifts of harmonic Maa\ss-Jacobi forms belong to the space~$\rmM^\sk_k$ (see Definition~\ref{def:skew-harmonic-siegel-maass} and~\cite{bringmann-raum-richter-2011}), which is a space of real-analytic, scalar-valued Siegel modular forms. It is difficult to directly inspect Fourier-Jacobi expansions of forms in~$\rmM^\sk_k$.  While the differential equations that they satisfy have already been examined by Maa\ss~\cite{maass-1953}, only integral representations~\cite{shimura-1982} and explicit solutions in a few special cases~\cite{bringmann-raum-richter-2011} are available. Our main tool is the Kohnen limit process. If $m>0$ (and $k>3$), then \cite{bringmann-raum-richter-2011} established that Kohnen's limit process~$\Klim(\wtd\phi_m)$ for Fourier-Jacobi coefficients $\wtd\phi_m$ of skew-harmonic Maa\ss-Siegel forms belongs to $\rmJ^\sk_{k,m}$, the space of skew-holomorphic Jacobi forms of weight~$k$ and index~$m$. In this paper, we extend this Kohnen limit process to~$m < 0$ (see Section~\ref{ssec:kohnen-limit-negative-m}), which provides a map into a space~$\cJ^\sk_{2-k,|m|}$ of certain real-analytic Jacobi forms (see Section~\ref{ssec:harmonic-jacobi-forms}). This space is isomorphically related by differential operators to the space of harmonic Maa\ss-Jacobi forms~$\bbJ_{3-k,|m|}$, which can then be mapped to $\rmJ^\sk_{k,|m|}$ via the Jacobi $\xi$-operator (see~\cite{bringmann-richter-2010} for dual weights in the theory of Jacobi forms). Furthermore, we introduce a family of Kohnen limit processes for~$m=0$, which are maps to certain scalar-valued and vector-valued elliptic modular forms (see Section~\ref{sec:kohnen-limit-zero-m}).

The discussion so far in the cases $m \ne 0$ is summarized in Figure~\ref{fig:spaces-of-siegel-modular-forms}. It displays the interplay among various spaces of Siegel modular forms and Jacobi forms that we use in our construction. The top is the space of skew-harmonic Maa\ss-Siegel forms. Via the Fourier-Jacobi expansion in~\eqref{eq:introduction:fourier-jacobi} it maps to the middle, which is the space $\rmA\rmJ^\sk_{k,m}$ of abstract Fourier-Jacobi coefficients (see Definition~\ref{def:abstract-fourier-jacobi-term-skew-harmonic}). The Kohnen limit process provides an injective map from $\rmA\rmJ^\sk_{k,m}$ to $\cJ^\sk_{2-k,|m|} \cong \bbJ_{3-k,|m|}$ (if $m < 0$) and $\rmJ^\sk_{k,|m|} = \rmJ^\sk_{k,m}$ (if $m > 0$) at the bottom of Figure~\ref{fig:spaces-of-siegel-modular-forms} (see Propositions~\ref{prop:kohnen-limit-positive-m-injective} and~\ref{prop:kohnen-limit-negative-m-injective}).
\begin{figure}[ht]
\caption{Spaces of Siegel modular forms and Jacobi forms for $m \ne 0$}
\label{fig:spaces-of-siegel-modular-forms}
\begin{center}
\begin{tikzpicture}
\matrix(m)[matrix of math nodes,
column sep=12em, row sep=4em,
text height=1.5em, text depth=1.25ex]{%
& \rmM^{\sk}_k & \\
& \rmA\rmJ^\sk_{k,m} & \\
\cJ^{\sk}_{2-k,|m|} & \bbJ_{3-k,|m|} & \rmJ^{\sk}_{k,|m|} \\
};

\path
(m-3-2) edge[left hook->>] node[below] {$\rmR^\bbJ_{3-k}$} (m-3-1)
(m-3-2) edge[right hook->] node[below] {$\xi^\rmJ_{3-k}$} (m-3-3)

(m-1-2) edge[->] (m-2-2)

(m-2-2) edge[left hook->] node[right=2.4em] {$\underset{(m < 0)}{\Klim}$} (m-3-1)
(m-2-2) edge[right hook->] node[left=2.4em] {$\underset{(m > 0)}{\Klim}$} (m-3-3)

(m-3-1) edge[right hook->,bend left = 30] node[above left=.2em] {$\underset{(m < 0)}{\Klim^{-1}}$} (m-2-2)
(m-3-3) edge[left hook->,bend right = 30] node[above right=.2em] {$\underset{(m > 0)}{\Klim^{-1}}$} (m-2-2)
;
\end{tikzpicture}
\end{center}
\end{figure}
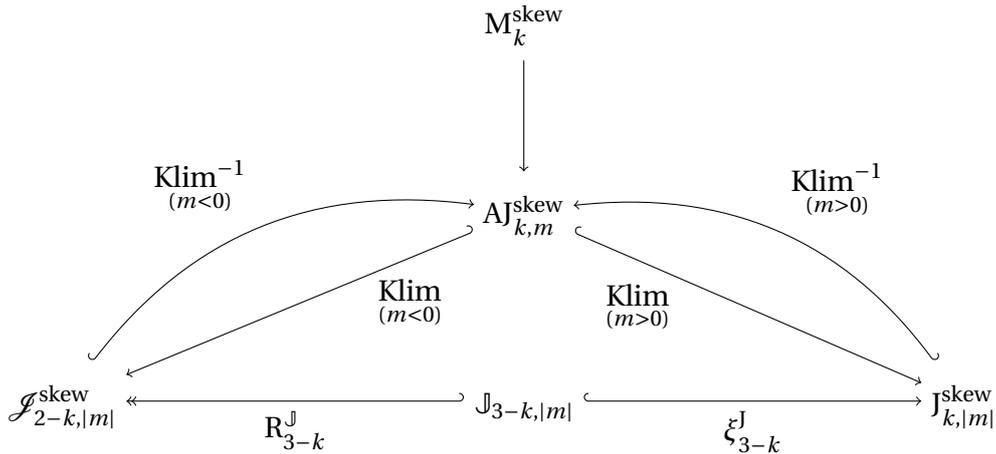

Note that the space $\bbJ_{3-k, |m|}$ of Maa\ss-Jacobi forms consists only of Eisenstein series (see Remark~\ref{rem:almost-skew-harmonic-jacobi-forms}).

Now we address three of the four key points of our work. Figure~\ref{fig:spaces-of-siegel-modular-forms} displays the Fourier-Jacobi expansion and the associated Kohnen limit processes for $\rmM^\sk_k$ if~$m \ne 0$. Specifically, we extend the Kohnen limit processes to the space of abstract Fourier-Jacobi coefficients~$\rmA\rmJ^\sk_{k,m}$. The first key point is to recognize that the Kohnen limit process for $m > 0$ and $m < 0$ maps into two different spaces in a natural way (see Section~\ref{sec:kohnen-limit-nonzero-m}). If $m>0$, then one recovers skew-holomorphic Jacobi forms as in~\cite{bringmann-raum-richter-2011}, while if $m<0$, then one discovers the space of real-analytic Maa\ss-Jacobi forms $\cJ^{\sk}_{2-k,|m|}$. That space of Jacobi forms is isomorphic to the space of harmonic Maa\ss-Jacobi forms $\bbJ_{3-k,|m|}$, which is linked to the space~$\rmJ^{\sk}_{k,|m|}$ via the Jacobi $\xi$-operator. This is displayed in the bottom row of Figure~\ref{fig:spaces-of-siegel-modular-forms}. As a consequence, the skew-Maa\ss\ lift to skew-harmonic Maa\ss-Siegel forms originates from Maa\ss-Jacobi forms as opposed to skew-holomorphic Jacobi forms. In particular, we do not expect to encounter lifts of cusp forms. This is confirmed by the second part of Theorem~\ref{mainthm:skew-maass-lift}, which asserts that~$\rmM^\sk_k$ consists only of Eisenstein series. Hence $\bbJ^\sk_{3-k, |m|}$ and $\rmM^\sk_k$ cannot serve as the preimage and target space to the skew-Maa\ss\ lift of cusp forms, respectively.

The second key point is to invert the Kohnen limit process for $m \ne 0$. The spaces of abstract Fourier-Jacobi coefficients~$\rmA\rmJ^\sk_{k,m}$ are isomorphic to the space of skew-holomorphic Jacobi forms~$\rmJ^\sk_{k,m}$, if $m > 0$, and the space of Maa\ss-Jacobi forms~$\bbJ_{3-k,|m|}$, if $m < 0$. In particular, we establish the bent arrows at the left and right of Figure~\ref{fig:spaces-of-siegel-modular-forms}, which allows us to proceed with the proof of the skew-Maa\ss\ lift as in the holomorphic case. This is a delicate task, which we accomplish by using the covariance of the Kohnen limit process and multiplicity-one for Fourier series coefficients.

We next use Hecke operators to map $\rmJ^{\sk}_{k,1}$ to $\rmJ^{\sk}_{k,m}$ for $m > 0$, and $\bbJ_{3-k,1}$ to $\bbJ_{3-k,|m|}$ for~$m < 0$. These correspond to maps from $\rmA\rmJ^\sk_{k,\pm 1}$ to $\rmA\rmJ^\sk_{k,m}$ for $m \ne 0$. The usual Jacobi-Hecke operators~$V_m$ for~$m > 0$ easily extend to skew-holomorphic Jacobi forms and are readily available for Maa\ss-Jacobi forms. We use the associated Fourier series coefficient formula to define $V_m$ for the abstract Fourier-Jacobi coefficients of index~$m \ne 0$. Observe that the same formulas could be obtained from the view on Jacobi-Hecke operators in~\cite{gritsenko-1995}. It allows us to define two partial Fourier-Jacobi expansions:
\begin{gather}
\label{eq:introduction:skew-maass-lift-nonzero-m}
  \sum_{m=1}^\infty
  \Big( \Klim_{k,-1}^{-1\,\sk}(\phi) \Big) \big|^\sk_k\, V_m
\qquad\tx{and}\qquad
  - 4 (k-2)!
  \sum_{m=1}^\infty
  \Big( \Klim_{k,1}^{-1\,\sk}(\xi^\rmJ_{3-k}\, \phi) \Big) \big|^\sk_k\, V_m
\tx{,}
\end{gather}
where $\xi^\rmJ_{3-k}$ is the Jacobi~$\xi$-operator in~\eqref{eq:def:jacobi-xi-operator}. By construction, both expansions in~\eqref{eq:introduction:skew-maass-lift-nonzero-m} are invariant under the embedded Jacobi group. It remains to patch them in such a way that the resulting Fourier series expansion transforms like a Siegel modular form under the map that interchanges $\tau$ and $ \tau'$. Covariance properties of the Kohnen limit process reduce the corresponding compatibility of Fourier series expansions to a normalizing factor~$-4 (k-2)!$ relating $\Klim^{-1}_{+1}$ and $\Klim^{-1}_{-1}$, and the definition of a $0$\thdash\ Fourier-Jacobi coefficient. We discover that the inverse Kohnen limit processes are compatible with Fourier series expansions, and therefore allow us to control Fourier series coefficients of skew-harmonic Maa\ss-Siegel forms sufficiently well to define the $0$\thdash\ Fourier-Jacobi coefficient of the skew-Maa\ss\ lift. This leads us to the next key point.

The third key point is the Hecke-like operator~$V_0$, which we employ to obtain the $0$\thdash\ Fourier-Jacobi coefficient of the skew-Maa\ss\ lift. In this paper, it consists of two parts. The first part accounts for the contribution of Eisenstein series and parallels~$V_0$ in the holomorphic case~\cite{eichler-zagier-1985}. We compute the Fourier series expansion of the skew-harmonic Eisenstein series and thus obtain an explicit expression for~$V_0$ analogous to~\cite{eichler-zagier-1985}. The second part of $V_0$ is novel in the sense that it does not appear in the holomorphic case. It arises from Fourier series coefficients whose index is indefinite and isotropic as a quadratic form. Observe that in the holomorphic case, Fourier series coefficients of indefinite index vanish automatically by the Koecher principle. The second part of~$V_0$ will also appear in the sequel to this paper.  It will be responsible for the phenomenon that the skew-Maa\ss\ lift of a cusp form is not necessarily a cusp form. This agrees with the fact from automorphic representation theory that the nonholomorphic Saito-Kurokawa lift of holomorphic elliptic cusp forms under certain conditions is residual. We also remark here already that in the case of cusp forms, there will be a third part of~$V_0$, which reflects the theory of constant terms in automorphic representation theory~\cite{moeglin-waldspurger-1994}.

\begin{maintheorem}[The skew-Maa\ss\ lift]
\label{mainthm:skew-maass-lift}
Let $\phi \in \bbJ_{3-k,1}$ be a harmonic Maa\ss-Jacobi form (of moderate growth) of weight\/~$3-k$ and index~$1$, $\xi^\rmJ_{3-k}$ be the Jacobi~$\xi$-operator in~\eqref{eq:def:jacobi-xi-operator}, $\Klim^{-1}_{k,\pm 1}$ be the inverse Kohnen limit processes in~\eqref{eq:def:inverse-kohnen-limit-positive-m} and~\eqref{eq:def:inverse-kohnen-limit-negative-m-maass-jacobi}, and $V_m$ be the Hecke operator in~\eqref{eq:def:V-hecke-operator-positive-l-abstract-fourier-jacobi-coefficient}. Then
\begin{multline}
\label{eq:def:skew-maass-lift-maass-jacobi}
  \MAsk_k(\phi)
\;:=\;
  \sum_{m = 0}^\infty
  \Klim^{-1}_{k,-1}(\phi) \big|^\sk_k\, V_m
  \,-\,
  4 (k-2)!
  \sum_{m = 1}^\infty
  \Klim^{-1}_{k,1}(\xi^\rmJ_{3-k}\,\phi) \big|^\sk_k\, V_m
\\
\in\;
  \rmM^{\sk}_k
\tx{.}
\end{multline}
The skew-Maa\ss\ lift is an isomorphism of spaces of Eisenstein series:
\begin{gather*}
  \MAsk_k :\,
  \bbJ_{3-k,1}
\lra
  \rmM^\sk_k
\tx{,}\quad
  E^{\bbJ}_{3-k,1}
\lmto
  \frac{(-1)^{\frac{1-k}{2}} \Gamma(k-\frac{1}{2})\, \zeta(k)}{(2 \pi)^k}\,
  E^\sk_k
\tx{.}
\end{gather*}
In particular, $\rmM^\sk_k$ is one-dimensional and spanned by~$E^\sk_k$.
\end{maintheorem}

\begin{mainremarks}
\begin{enumeratearabic}
\item
Let $\phi \in \bbJ_{3-k,1}$ be a harmonic Maa\ss-Jacobi form, $\sum_m \wtd\phi_m(Z)$ be the Fourier-Jacobi expansion of its skew-Maa\ss\ lift in Theorem~\ref{mainthm:skew-maass-lift}, $\Klim^\sk_{k,-1}$ be the Kohnen limit process defined in~\eqref{eq:def:kohnen-limit-negative-m}, and $\rmR^\bbJ_{3-k}$ be the isomorphism in Proposition~\ref{prop:harmonic-and-almost-skew-harmonic-jacobi-forms}. Then $\rmR^\bbJ_{3-k} \phi = \Klim^\sk_{k,-1}(\wtd\phi_{-1})$.

\item
Kohnen asked on page~128 of~\cite{kohnen-1993} which real-analytic Siegel modular forms correspond to skew-holomorphic Jacobi forms.  Theorem~\ref{mainthm:skew-maass-lift} improves the initial answer that was given in~\cite{bringmann-raum-richter-2011}. The skew-Maa\ss\ lift of skew-holomorphic Jacobi forms in the last paper of this series will provide a complete answer to Kohnen's question.

\item
The mere definition of the skew-Maa\ss\ lift depends on several ingredients from Sections~\ref{sec:kohnen-limit-nonzero-m} and~\ref{sec:kohnen-limit-zero-m}. In the last part of this series of papers, we will discover that it naturally extends to the definition of skew-Maa\ss\ lifts of cuspidal skew-holomorphic Jacobi forms. 

\item
Our proof of Theorem~\ref{mainthm:skew-maass-lift} relies heavily on the computation of the Fourier series coefficients of~$E^\sk_k$. Our method can be applied to all Siegel Eisenstein series and it does provide a novel access to the Fourier series coefficients of Siegel Eisenstein series. These are directly related to Siegel series~(see p.~306 of~\cite{maass-1971}) and have recently been studied by $p$-adic methods~\cite{ikeda-katsurada-2016-preprint}.
\end{enumeratearabic}
\end{mainremarks}

We now discuss the case $m = 0$ of the Kohnen limit process, displayed in Figure~\ref{fig:asymptotics-of-siegel-modular-forms}. Its top row coincides with the top row of Figure~\ref{fig:spaces-of-siegel-modular-forms}. We define Kohnen limit processes, $\Klim[0,0]$, $\Klim[0,1]$, and $\Klim[N]$ for $N \in \ZZ_{> 0}$, which exhibit the first three terms of the asymptotic expansion of a Siegel modular form with respect to~$y'$. They map $\rmM^\sk_k$ to $\cM_{1-k}(k-\frac{1}{2})$, $\rmM_{k-1}$, and $\rmM_{k-1}(\rho_N)$, respectively. These are spaces of elliptic Maa\ss\ forms with spectral parameter~$k-\frac{1}{2}$ (i.e., Laplace eigenvalue~$\frac{1}{2}(k - \frac{1}{2})$) and elliptic modular forms, respectively. None of the maps $\Klim[N]$ are injective.
\begin{figure}[ht]
\caption{Spaces of Siegel modular forms and elliptic modular forms for $m = 0$}
\label{fig:asymptotics-of-siegel-modular-forms}
\begin{center}
\begin{tikzpicture}
\matrix(m)[matrix of math nodes,
column sep=6.5em, row sep=4em,
text height=1.5em, text depth=1.25ex]{%
& \rmM^{\sk}_k & \\
& \rmA\rmJ^{\sk}_{k,0} & \\
\cM_{1-k}(k-\frac{1}{2}) & \rmM_{k-1} & \rmM_{k-1}(\rho_N) \\
};

\path
(m-1-2) edge[->] (m-2-2)
(m-2-2) edge[->] node[above left] {$\underset{(m=0)}{\Klim[0,0]}$} (m-3-1)
(m-2-2) edge[->] node[below right] {$\underset{(m=0)}{\Klim[0,1]}$} (m-3-2)
(m-2-2) edge[->] node[above right] {$\underset{(m=0, N \ge 1)}{\Klim[N]}$} (m-3-3)
;
\end{tikzpicture}
\end{center}
\end{figure}
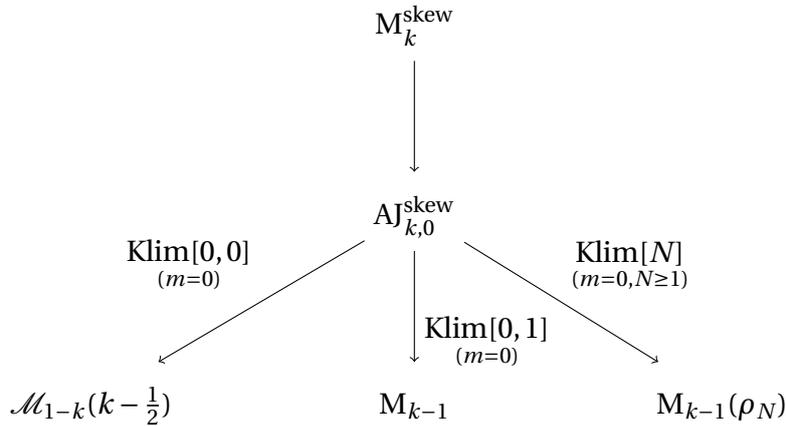
\vspace{.5\baselineskip}

The fourth key point of this paper is our definition of the Kohnen limit processes for Jacobi index~$m = 0$. They differ significantly from the case of nonzero~$m$, in that we define three of them, and $\Klim[N]$, the most intricate one, relies on the limit $y' \ra 0$. While it seems possible to define such a Kohnen limit process for $m=0$ via $y' \ra \infty$, the resulting objects are highly discontinuous. They bare a certain similarity to the local harmonic weak Maa\ss\ forms in~\cite{bringmann-kane-kohnen-2015}, but it is unclear to us in how far this similarity is conceptual or coincidental. The Kohnen limit process of index~$m = 0$ will be studied in more detail in the second paper of this series of papers.

The paper is organized as follows. In Section~\ref{sec:elliptic-modular-forms}, we review some aspects of elliptic modular forms. In Section~\ref{sec:harmonic-siegel-maass-forms}, we revisit skew-harmonic Maa\ss-Siegel forms and their Fourier series expansions. Section~\ref{sec:jacobi-forms} contains all material on Jacobi forms and Fourier-Jacobi coefficients. Sections~\ref{sec:kohnen-limit-nonzero-m} and~\ref{sec:kohnen-limit-zero-m} are the heart of this paper and they are dedicated to the Kohnen limit process. In Section~\ref{sec:kohnen-limit-nonzero-m}, we define the Kohnen limit processes for $m\not=0$, and we give the corresponding inverse Kohnen limit processes. In Section~\ref{sec:kohnen-limit-zero-m}, we investigate the Kohnen limit process for $m = 0$, which is the most subtle point of this paper.  Finally, in Section~\ref{sec:maass-lift}, we define the Hecke-like operator~$V_0$, and we prove Theorem~\ref{mainthm:skew-maass-lift}.

\paragraph{Acknowledgments}

We thank Wee Teck Gan, Werner Hoffmann, \"Ozlem \.Imamo\u glu, and Rainer Weissauer for useful discussions and helpful comments. We are grateful to the referees for a careful reading and for making detailed suggestions that improved this paper.

\section{Elliptic modular forms}
\label{sec:elliptic-modular-forms}

We start by revisiting some facts about elliptic modular forms. Let $\tau = x + iy \in \HS$ be a variable in the Poincar\'e upper half space, and write $e(a) := \exp(2 \pi i a)$ for any~$a \in \CC$. Throughout the paper, $\zeta$ denotes the Riemann $\zeta$\nbd function, $\Gamma$ the $\Gamma$-function, and~$W_{\kappa,\mu}$ the $W$-Whittaker function.

\subsection{Vector-valued modular forms}
\label{ssec:elliptic-vector-valued-modular-forms}

A finite dimensional, complex right representation of $\SL{2}(\ZZ)$ is called an arithmetic type\footnote{Usually, the notion of arithmetic types requires left representations, but for our purpose it is more convenient to consider right representations.}. Any arithmetic type~$\rho$ gives rise to slash actions
\begin{gather*}
  \big( f \big|_{\kappa,\rho}\, \gamma \big) (\tau)
:=
  (c \tau + d)^{-\kappa}\,
  f\big( \frac{a \tau + b}{c \tau + d} \big) \rho(\gamma)
\tx{,}\qquad
  \kappa \in \ZZ
\tx{,}\;
  \gamma
=
  \begin{psmatrix} a & b \\ c & d \end{psmatrix}
\end{gather*}
on functions~$f :\, \HS \ra V(\rho)$, where $V(\rho)$ is the representation space of~$\rho$. In the classical setting, when $\rho$ is the trivial representation, we suppress it from our notation and simply write $( f \big|_\kappa\, \gamma) (\tau)$.
\begin{definition}
\label{def:vector-valued-elliptic-modular-form}
A holomorphic function $f :\, \HS \ra V(\rho)$ is a holomorphic modular form of weight~$\kappa$ and type~$\rho$ if it satisfies
\begin{enumerateroman}
\item
$f \big|_{\kappa,\rho}\, \gamma = f$ for all $\gamma \in \SL{2}(\ZZ)$.

\item
$f(\tau) = \cO(1)$ as $y \ra \infty$.
\end{enumerateroman}
\end{definition}
Real-analytic modular forms of weight~$\kappa$ and type~$\rho$ are defined analogously.

\subsection{Eisenstein series of level \texpdf{$1$}{1}}

For $\kappa \in 2 \ZZ$ and $s \in \CC$ with $\kappa + \Re(s) > 2$, the level~$1$ real-analytic Eisenstein series is defined by
\begin{gather}
\label{eq:def:elliptic-eisenstein-series}
  E_\kappa(s, \tau)
:=
  \sum_{\Gamma_\infty \backslash \SL{2}(\ZZ)}
  y^s \big|_\kappa\,\ga
\tx{,}\quad\tx{where}\quad
  \Gamma_\infty
:=
  \big\{ \begin{psmatrix} \pm 1 & \ast \\ 0 & \pm 1 \end{psmatrix} \in \SL{2}(\ZZ) \big\}
\tx{.}
\end{gather}
For general $s \in \CC$ we define $E_\kappa(s,\tau)$ by analytic continuation. If $s=0$ and~$\kappa > 2$, then the Eisenstein series is holomorphic, and we write $E_\kappa(\tau) := E_\kappa(0, \tau)$. We give the known Fourier series expansion of $E_\kappa(s, \tau)$ (see, for example, pages~181 and~210 of~\cite{maass-1964}) in the following Lemma:
\begin{lemma}
\label{la:elliptic-eisenstein-series-fourier-expansion}
Let $\kappa \in 2\ZZ$ and $s \in \CC$.  We have the expansion
\begin{multline*}
  E_\kappa(s, \tau)
\,=\,
  y^s
  +
  \frac{(-1)^{\frac{\kappa}{2}} 2^{2-\kappa-2s} \pi \Gamma(\kappa + 2s - 1)}
       {\Gamma(\kappa + s) \Gamma(s)}
  \frac{\zeta(\kappa + 2s - 1)}{\zeta(\kappa + 2s)}\,
  y^{1 - \kappa - s}
\\
  +\,
  \frac{(-1)^{\frac{\kappa}{2}} 2^\kappa \pi^{\kappa+s}}{\zeta(\kappa + 2s)}
  \sum_{m \in \ZZ \setminus \{0\}}
  \frac{|m|^{-s}\, \sigma_{\kappa + 2s - 1}(|m|)}
       {\Gamma(\frac{1+\sgn(m)}{2}\kappa + s)}\,
  \big( 4 \pi |m| y \big)^{\frac{-\kappa}{2}}
  W_{\frac{\sgn(m)\kappa}{2}, \frac{\kappa+2s-1}{2}}\big( 4 \pi |m| y \big)
  e( |m| x )
\tx{.}
\end{multline*}
In the holomorphic case, we have the expansion
\begin{gather*}
  E_\kappa(\tau)
\,=\,
  1
  \,+\,
  \frac{(2 \pi i)^\kappa}{\Gamma(\kappa) \zeta(\kappa)}
  \sum_{n=1}^\infty \sigma_{\kappa-1}(n)\, e(n \tau)
\tx{.}
\end{gather*}
\end{lemma}

\subsection{A specific arithmetic type and the associated Eisenstein series}

For $N \in \ZZ_{\ge 1}$, consider the $\CC$ vector space~$V(\rho_N)$ with basis $\frake_{\alpha, \beta}$, for $\alpha, \beta \in \frac{1}{N} \ZZ \slash \ZZ$ that satisfy $\alpha \ZZ + \beta \ZZ = \frac{1}{N} \ZZ$, i.e., \@ $\gcd(N\alpha, N\beta, N) = 1$ for any representative in~$\QQ$ of $\alpha$ and $\beta$. It yields an arithmetic type~$\rho_N$ via the natural right representation of $\SL{2}(\ZZ)$ arising from the right action on pairs $(\alpha, \beta)$. Observe that the basis of~$V(\rho_N)$ may be equally indexed by pairs $(c,d)$ of co-prime integers mod~$N$, and we denote these bases elements by $\frakf_{c,d}$.

Let $N \in \ZZ_{\ge 1}$, $\kappa \in 2 \ZZ$, and $s \in \CC$. Then the Eisenstein series
\begin{gather}
\label{eq:def:vector-valued-elliptic-eisenstein-series}
  E_{\kappa,N}(s, \tau)
:=
  \sum_{\gamma \in \Gamma_\infty \backslash \SL{2}(\ZZ)}
  \frakf_{0,1}\,  y^s \big|_{\kappa,\rho_N}\,\ga
\tx{}
\end{gather}
is a modular form of weight~$\kappa$ and type~$\rho_N$. We write
\begin{gather*}
\label{eq:def:vector-valued-elliptic-eisenstein-series-holomorphic}
  E_{\kappa,N}(\tau)
:=
  E_{\kappa,N}(0, \tau)
\tx{.}
\end{gather*}

\section{Harmonic Maa\ss-Siegel forms of degree \texpdf{$2$}{2}}
\label{sec:harmonic-siegel-maass-forms}

In this paper, we investigate scalar-valued Siegel modular forms of degree~$2$. For more details see~\cite{freitag-1983,klingen-1990, maass-1953, maass-1971, bringmann-raum-richter-2011}.

\subsection{Preliminaries}
\label{ssec:siegel-modular-forms-preliminaries}

Let us introduce necessary notation to define Siegel modular forms.

\paragraph{The symplectic group}

The symplectic group is given by
\begin{gather*}
  \Sp{2}(\RR)
:=
  \big\{ g \in \Mat{4}(\RR) \,:\,
   \rT g J^{(2)} g = J^{(2)}
  \big\}
\tx{,}\quad
  J^{(2)}
:=
  \begin{psmatrix}
  0 & 0 & -1 & 0 \\
  0 & 0 & 0 & -1 \\
  1 & 0 & 0 & 0 \\
  0 & 1 & 0 & 0
  \end{psmatrix}
\tx{.}
\end{gather*}
We usually write elements~$g\in\Sp{2}(\RR)$ as $g = \begin{psmatrix} A & B \\ C & D \end{psmatrix}$ with $A, B, C, D \in \Mat{2}(\RR)$. For $U \in \GL{2}(\RR)$ and $g = \begin{psmatrix} a & b \\ c & d \end{psmatrix} \in \SL{2}(\RR)$, set
\begin{gather}
\label{eq:def:rot-up-down-embedding}
  \rot(U)
:=
  \begin{psmatrix}
  U & 0 \\ 0 & \rT U^{-1}
  \end{psmatrix}
\tx{,}\qquad
  \up(g)
:=
  \begin{psmatrix}
  a & 0 & b & 0 \\
  0 & 1 & 0 & 0 \\
  c & 0 & d & 0 \\
  0 & 0 & 0 & 1
  \end{psmatrix}
\tx{,}\qquad
  \down(g)
:=
  \begin{psmatrix}
  1 & 0 & 0 & 0 \\
  0 & a & 0 & b \\
  0 & 0 & 1 & 0 \\
  0 & c & 0 & d
  \end{psmatrix}
\tx{.}
\end{gather}
Let~$R$ be a ring. Write~$\MatT{2}(R) := \{ M \in \Mat{2}(R) \,:\, \rT M = M \}$ for the set of symmetric $2 \times 2$ matrices with entries in~$R$. If $B \in \MatT{2}(R)$ with entries $B_{11}$, $B_{12}$, and $B_{22}$, then set
\begin{gather}
  \trans(B)
:=
  \begin{psmatrix}
  1 & 0 & B_{11} & B_{12} \\
  0 & 1 & B_{12} & B_{22} \\
  0 & 0 & 1 & 0 \\
  0 & 0 & 0 & 1
  \end{psmatrix}
\tx{.}
\end{gather}

Let $\Gamma^{(2)} := \Sp{2}(\ZZ)$ be the full Siegel modular group, and
\begin{gather}
\Gamma^{(2)}_\infty
:=
\big\{
\begin{psmatrix} A & B \\ 0 & D \end{psmatrix} \in \Gamma^{(2)}
\big\}
\end{gather}
be the Siegel parabolic subgroup of $\Gamma^{(2)}$.

\paragraph{The Siegel upper half space}

Recall from the introduction that
\begin{gather*}
\HS^{(2)}
:=
\big\{ Z \in \MatT{2}(\CC) \,:\, \Im(Z) \tx{ positive definite} \big\}
\text{,}
\end{gather*}
where $Z = \begin{psmatrix} \tau & z \\ z & \tau' \end{psmatrix} \in \HS^{(2)}$ is a typical variable. Throughout, we denote the real part of $Z$ by $X = \begin{psmatrix} x & u \\ u & x' \end{psmatrix}$, and its imaginary part by $Y = \begin{psmatrix} y & v \\ v & y' \end{psmatrix}$.
The Siegel upper half space carries an action of the symplectic group by
\begin{gather*}
g Z
:=
(A Z + B) (C Z + D)^{-1}
\tx{.}
\end{gather*}

\paragraph{Slash actions}
The classical weight~$k$ slash action is defined for $F :\, \HS^{(2)} \ra \CC$ and $k \in \ZZ$:
\begin{gather*}
\big( F \big|_k\, g \big) (Z)
:=
\det(C Z + D)^{-k}\, F(g Z)
\tx{.}
\end{gather*}
The definition of skew-harmonic Maa\ss-Siegel forms in~\cite{bringmann-raum-richter-2011} requires the weight~$k$ skew slash action, which is also defined for $F :\, \HS^{(2)} \ra \CC$ and $k \in \ZZ$:
\begin{gather}
\label{eq:def:siegel-skew-slash-action}
\big( F \big|^\sk_k\, g \big) (Z)
:=
|\det(C Z + D)|^{-1} \ov{\det(C Z + D)}^{1-k}\, F(g Z)
\tx{.}
\end{gather}
Observe that multiplication by~$\det(Y)^{k-\frac{1}{2}}$ shows that the slash actions $\big|^\sk_k$ and $\big|_{1-k}$ are equivalent (see also Remark~2.4 of~\cite{bringmann-richter-raum-2016}).

\subsection{Skew-harmonic Maa\ss-Siegel forms}

In this section, we briefly recall the notion of harmonic Maa\ss-Siegel form in~\cite{bringmann-raum-richter-2011}. We call these objects skew-harmonic Maa\ss-Siegel forms to emphasize their skew weight $(\frac{1}{2}, k-\frac{1}{2})$.

\paragraph{The matrix-valued Laplace operator}
Let
\begin{gather*}
\partial_Z
:=
\begin{pmatrix} \partial_{\tau} & \frac{1}{2}\partial_z \\
  \frac{1}{2}\partial_z & \partial_{\tau'}
\end{pmatrix}
\quad \tx{and} \quad
\partial_{\ov{Z}}
:=
\begin{pmatrix} \partial_{\ov{\tau}} & \frac{1}{2}\partial_{\ov{z}} \\
  \frac{1}{2}\partial_{\ov{z}} & \partial_{\ov{\tau'}}
\end{pmatrix}
\tx{,}
\end{gather*}
where $\partial_{w}:=\frac{\partial}{\partial w} = \frac{1}{2}\big( \frac{\partial}{\partial a} - i\frac{\partial}{\partial b} \big)$ and $\partial_{\ov{w}} := \frac{\partial}{\partial \ov{w}} = \frac{1}{2} \big( \frac{\partial}{\partial a} + i\frac{\partial}{\partial b} \big)$ for any complex variable $w = a+ib$.  Maass~\cite{maass-1953} introduced the matrix-valued Laplace operator
\begin{gather}
\label{eq:def:matrix-valued-laplace-operator}
\Omega^\sk_k 
:=
- 4 Y \rT \big(Y \partial_{\ov Z} \big) \partial_Z
- i (2 k - 1) Y \partial_Z
+ i Y \partial_{\ov Z}
\tx{.}
\end{gather}
This matrix-valued Laplace operator and also the differential operator
\begin{gather}
\label{eq:def:xi-operator}
\xi^\sk_k
:=
\det(Y)^{k-\frac{3}{2}}
\det (Z - \ov{Z}) (\partial_\tau \partial_{\tau'} - \tfrac{1}{4} \partial_z^2)
\tx{}
\end{gather}
play a vital role in the study of skew-harmonic Maa\ss-Siegel forms in~\cite{bringmann-raum-richter-2011}.  Observe that  \cite{bringmann-raum-richter-2011} denotes $\Omega^\sk_k$ and   $\xi^\sk_k$ by $\Omega_{\frac{1}{2},k-\frac{1}{2}}$ and $\xi^{(2)}_{\frac{1}{2},k-\frac{1}{2}}$, respectively.

\paragraph{The definition of skew-harmonic Maa\ss-Siegel forms}

Recall from the introduction that~$k > 3$ is odd.
\begin{definition}
\label{def:skew-harmonic}
A smooth function $F :\, \HS^{(2)} \ra \CC$ is weight~$k$ skew-harmonic if 
\begin{gather}
\Omega^\sk_k\, F = 0
\quad\tx{and}\quad
\xi^\sk_k\, F = 0
\tx{.}
\end{gather}
\end{definition}

\begin{definition}
\label{def:skew-harmonic-siegel-maass}
A weight~$k$ skew-harmonic function $F :\, \HS^{(2)} \ra \CC$ is a skew-harmonic Maa\ss-Siegel form of weight~$k$ if it satisfies
\begin{enumerateroman}
\item
$F \big|^\sk_k\, \gamma = F$ for all $\gamma \in \Ga^{(2)}$.

\item
$F(Z) = \cO(\tr(Y)^a)$ for some $a \in \RR$ as $\tr(Y) \ra \infty$.
\end{enumerateroman}
\end{definition}
We write $\rmM^\sk_k$ for the space of skew-harmonic Maa\ss-Siegel forms of weight~$k$.
\begin{remark}
Definition~\ref{def:skew-harmonic-siegel-maass} differs slightly from the definition of harmonic Maa\ss-Siegel forms in~\cite{bringmann-raum-richter-2011}, which does not require the skew-harmonic condition $ \xi^\sk_k\, F = 0$.  However, the main result of~\cite{bringmann-raum-richter-2011} also relies on this additional condition, which guarantees multiplicity-one for Fourier series coefficients of nonzero index~$T$ of harmonic Maa\ss-Siegel forms provided that~$T$ is indefinite or positive semi-definite (see also Proposition~\ref{prop:multiplicity-one-for-skew-harmonic-fourier-coefficients-nonzero-nonnegative} below).
\end{remark}

\paragraph{Eisenstein series}

Recall that~$k>3$ is odd. The Eisenstein series
\begin{gather}
\label{eq:def:skew-harmonic-siegel-eisenstein-series}
  E^\sk_k
:=
  \sum_{\gamma \in \Gamma^{(2)}_\infty \backslash \Gamma^{(2)}}
  1 \big|^\sk_k\, \gamma
\end{gather}
is the primary example of a weight~$k$ skew-harmonic Maa\ss-Siegel form in~\cite{bringmann-raum-richter-2011}.

\subsection{Skew-harmonic Fourier series coefficients}

Recall that every real-analytic Siegel modular form~$F$ has a Fourier series expansion of the form
\begin{gather}
  \sum_{T \in \MatT{2}(\frac{1}{2}\ZZ)} \!\!\!
  c(F;\, T;\, Y)\, e(T X)
\tx{,}
\end{gather}
where throughout $e(TX) := \exp(2 \pi i\, \trace(T X))$.

Fourier series coefficients of holomorphic Siegel modular forms satisfy multiplicity-one: For every~$T \in \MatT{2}(\RR)$ there is (up to multiplication by scalars) a unique function $a^\hol(T;Y)$ such that $a^\hol(T;Y) e(TX)$ is holomorphic. In the skew-harmonic setting, the analogue remains true if~$T$ is non-degenerate (i.e., $\det(T) \ne 0$), but for arbitrary~$T$ it no longer holds.  We recall the relevant result from~\cite{maass-1953,bringmann-raum-richter-2011} ($T$ indefinite or positive semi-definite) in Proposition~\ref{prop:multiplicity-one-for-skew-harmonic-fourier-coefficients-nonzero-nonnegative}, and we give the remaining cases in Propositions~\ref{prop:multiplicity-one-for-skew-harmonic-fourier-coefficients-negative} and \ref{prop:multiplicity-one-for-skew-harmonic-fourier-coefficients-zero}.

\paragraph{Non-degenerate and nonzero positive semi-definite Fourier indices}

If $F$ is a skew-harmonic Maa\ss-Siegel form, then the following multiplicity-one result for its Fourier series coefficients $c(F;\, T;\, Y)$ was established in~\cite{maass-1953,bringmann-raum-richter-2011}.
\begin{proposition}[{\cite{maass-1953,bringmann-raum-richter-2011}}]
\label{prop:multiplicity-one-for-skew-harmonic-fourier-coefficients-nonzero-nonnegative}
Let $T \in \MatT{2}(\RR) \setminus \{0\}$ be indefinite or positive semi-definite. If the function $a^\sk_k(T;\,Y) e(TX)$ is weight~$k$ skew-harmonic and also of moderate growth, then $a^\sk_k(T;\,Y)$ is uniquely determined up to scalar multiples. Furthermore, if\/~$T$ is not positive definite, then one can choose this scalar normalization in such a way that $a^\sk_k(T;\,Y)$ is nonzero, and depends only on $\trace(T Y)$, $\det(T Y)$, and the signature of\/~$T$. If\/~$T$ is positive definite, then $a^\sk_k(T;\,Y) = 0$.
\end{proposition}

\begin{remark}
Fourier series coefficients of index~$T = \begin{psmatrix} 0 & 0 \\ 0 & 0 \end{psmatrix}$ violate multiplicity-one, and can occur in the Fourier series expansion of skew-harmonic Siegel modular forms. Observe that Theorem~3~(a) of~\cite{bringmann-raum-richter-2011} incorrectly states that the ``constant term'' of elements of~$\rmM^\sk_k$ is independent of~$Y$, which however does not impact any of the other results of~\cite{bringmann-raum-richter-2011}. Under our assumption that $k > 3$ the Fourier series coefficients of index~$T = \begin{psmatrix} 0 & 0 \\ 0 & 0 \end{psmatrix}$ of forms in~$\rmM^\sk_k$ satisfy multiplicity two due to their invariance under~$\rot(\GL{2}(\ZZ))$ (see Proposition~\ref{prop:multiplicity-one-for-skew-harmonic-fourier-coefficients-zero}).
\end{remark}

\paragraph{Non-zero negative semi-definite Fourier indices}

The case of negative semi-definite Fourier indices violates multiplicity-one, if the Fourier index~$T$ is degenerate. In that case, we can recover a multiplicity-two result. We give the details in the following Proposition.
\begin{proposition}
\label{prop:multiplicity-one-for-skew-harmonic-fourier-coefficients-negative}
Let $T \in \MatT{2}(\RR) \setminus \{0\}$ be negative semi-definite and degenerate. If the function $a^\sk_k(T;\,Y) e(T X)$ is weight~$k$ skew-harmonic and also of moderate growth, then $a^\sk_k(T;\,Y)$ is a linear combination of
\begin{align*}
&
  \det(Y)^{\frac{3}{2} - k}\,
  \trace(T Y)^{k - 2} e^{-\trace(T Y)}
\qquad\tx{and}
\\&
  \big( 4 \pi \trace(T Y) \big)^{-\frac{k}{2}}
  W_{\frac{k-1}{2}, \frac{k-1}{2}}\big( 4 \pi |\trace(T Y)| \big)
\tx{.}
\end{align*}

Let $T \in \MatT{2}(\RR)$ be negative definite. If the function $a^\sk_k(T;\,Y) e(T X)$ is weight~$k$ skew-harmonic and also of moderate growth, then $a^\sk_k(T;\,Y)$ is unique up to scalar multiples.
\end{proposition}
\begin{remark}
If $T<0$ and $a^\sk_k(T;\,Y) e(T X)$ is weight~$k$ skew-harmonic and also of moderate growth, then one can show that~$a^\sk_k(T;\,Y) e(T X)$ is almost anti-holomorphic.
\end{remark}

\begin{proof}
For the semi-definite and degenerate case, we use~(33) and~(34) of~\cite{maass-1953}. Using Maass's notation, weight~$k$ skew-harmonic functions $a^\sk_k(\frac{1}{2\pi} T;\,Y)\, e(\frac{1}{2\pi}T X)$ can be obtained from
\begin{gather*}
  a^\sk_k(T;\,Y)
\;=\;
  \det(Y)^{\frac{3}{2} - \alpha - \beta} \phi(\text{\texttt u})
  \,+\,
  \psi(\text{\texttt u})
\tx{,}
\end{gather*}
where $\tx{\texttt u} = \trace(-T Y)$, $\alpha = k - \frac{1}{2}$, $\beta = \frac{1}{2}$, and $\phi$ and $\psi$ satisfy the following differential equations:
\begin{gather*}
\begin{alignedat}{2}
  \text{\texttt u}\phi''
&
  +
  (3-\alpha-\beta) \phi'
  +
  (\alpha-\beta-\text{\texttt u}) \phi
&&{} = 0
\tx{,}
\\
  \text{\texttt u} \psi''
&
  +
  (\alpha+\beta) \psi'
  +
  (\alpha-\beta-\text{\texttt u})\psi
&&{} = 0
\tx{.}
\end{alignedat}
\end{gather*}
Exactly as in Theorem~3 of~\cite{bringmann-raum-richter-2011} (see also Remark~(b) to Theorem~3), we find that the space of solutions of moderate growth to each is one-dimensional, and spanned by
\begin{gather*}
  \text{\texttt u}^{k - 2} e^{-\text{\texttt u}}
\quad\tx{and}\quad
  ( 2 \text{\texttt u} )^{-\frac{k}{2}}
  W_{\frac{k-1}{2}, \frac{k-1}{2}}\big( 2 \text{\texttt u} \big)
\tx{,}
\end{gather*}
respectively. When replacing $T$ by $2 \pi T$, this settles the case of semi-definite~$T$.

Now consider the case that $T<0$. Starting with~(35), (36), and~(37) of~\cite{maass-1953}, we have to determine the solutions to
\begin{gather*}
  \phi''(\text{\texttt u})
=
  \Big( 1 + \frac{2(1-k)}{\text{\texttt u}} + \frac{(k-1)(k-2)}{\text{\texttt u}^2} \Big) \phi(\tx{\texttt u})
\tx{.}
\end{gather*}
Using, for example, Sage~\cite{sage-8.3}, one verifies that they are linear combinations of
\begin{gather*}
  \text{\texttt u}^{k - 1} e^{-\text{\texttt u}}
\quad\tx{and}\quad
  \text{\texttt u}^{k - 1} e^{-\text{\texttt u}}\, \gamma\big( 3-2k, -2\text{\texttt u} \big)
\tx{,}
\end{gather*}
where $\gamma(s,u)$ is the lower incomplete gamma function. Observe that $\tx{\texttt u} = \trace(-T Y) > 0$. Therefore, by the relation of $\gamma(s,u)$ and the $M$-Whittaker function in~8.5.1 of~\cite{nist-dlmf} and the asymptotic behavior of the $M$-Whittaker function in~13.14.20 of op.\@ cit.\@, the second solution grows exponentially. Hence it can be ruled out. From the first solution, we find that the solutions to
\begin{gather*}
  \psi'(\tx{\texttt u})
=
  \frac{1}{\tx{\texttt u}} \phi(\tx{\texttt u})
\end{gather*}
are linear combinations of
\begin{gather*}
  1
\quad\tx{and}\quad
   \int_{\tx{\texttt u}}^\infty
   t^{k - 1} e^{-t}\; d\!t
=
  \Gamma(k, \tx{\texttt u})
\tx{.}
\end{gather*}
The first one, as in the case of~$T > 0$ treated in~\cite{bringmann-raum-richter-2011}, yields a function~$a^\sk_k(T;\,Y)$ that violates the moderate growth condition (see Lemma~3.2 and~3.3 in~\cite{raum-2012d}), after inserting it into~(35) of~\cite{maass-1953}. This leaves us with an at most one-dimensional space of solutions.
\end{proof}
\begin{remark}
In the case of negative definite~$T$, we do not prove here that $a^\sk_k(T Y)$ is possibly nonzero. This will only follow from our discussion of the Kohnen limit process for index~$m < 0$ in Section~\ref{ssec:kohnen-limit-negative-m}.
\end{remark}

\paragraph{Zero Fourier indices}

The Fourier series coefficients of index~$T = \begin{psmatrix} 0 & 0 \\ 0 & 0 \end{psmatrix}$ are the most subtle, as they lead to infinite multiplicities. In the context of the present paper, we simplify the exposition by restricting to Fourier coefficients of skew-harmonic Maa\ss-Siegel forms, which allow for a further Fourier expansion in terms of a new set of coordinates. Specifically, we introduce the coordinate $\taub = \xb + i \yb \in \HS$ parameterizing~$Y \slash \sqrt{\det(Y)}$ as in~(30) of~\cite{maass-1953}. In contrast to~\cite{maass-1953}, we swap the top left and bottom-right components for better compatibility with the action of $\GL{2}(\ZZ)$ via~$\rot$ in Proposition~\ref{prop:rot-slash-action-on-zeroth-fourier-coefficient}:
\begin{gather}
\label{eq:Y-coordinates:deduction}
  Y
\;=:\;
  \sqrt{\det(Y)}\,
  \begin{pmatrix} 1 \slash \sqrt{\yb} & 0 \\ \xb \slash \sqrt{\yb} & \sqrt{\yb} \end{pmatrix}
  \begin{pmatrix} 1 \slash \sqrt{\yb} & \xb \slash \sqrt{\yb} \\ 0 & \sqrt{\yb} \end{pmatrix}
\tx{.}
\end{gather}
For later use, we record the associated coordinate transforms:
\begin{gather}
\label{eq:def:Y-coordinates}
\begin{alignedat}{3}
  y &= \sqrt{\det(Y)} \big\slash \yb
\tx{,}\quad&
  y' &= \sqrt{\det(Y)} (\xb^2 + \yb^2) \big\slash \yb
\tx{,}\quad&
  v &= \sqrt{\det(Y)} \xb \big\slash \yb
\tx{;}
\\
  \det(Y) &= y y' - v^2
\tx{,}\quad&
  \yb &= \sqrt{y y' - v^2} \big\slash y
\tx{,}\quad&
  \xb &= v \slash y
\tx{.}
\end{alignedat}
\end{gather}

We next state how the action of $\rot(\GL{2}(\RR))$ on functions of~$Y$ can be expressed in terms of M\"obius transformations on~$\taub$. We set
\begin{gather}
  \Im\big( \HS^{(2)} \big)
:=
  \big\{ Y \in \MatT{2}(\RR) \,:\, Y > 0 \big\}
\tx{,}
\end{gather}
and consider a function $\breve{F} :\, \Im(\HS^{(2)}) \ra \CC$ as a function~$F$ on $\HS^{(2)}$ via $F(X + iY) = \breve{F}(Y)$. In particular, the skew slash action of $\rot(\GL{2}(\ZZ))$ extends to functions on~$\Im(\HS^{(2)})$.

\begin{proposition}
\label{prop:rot-slash-action-on-zeroth-fourier-coefficient}
Let $\breve{F} :\, \Im(\HS^{(2)}) \ra \CC$ and $U \in \SL{2}(\RR)$. Then
\begin{gather}
  \big( \breve{F} \big|^\sk_k\, \rot(U) \big)(\det(Y), \taub)
=
  \det(U)^{k-1}\,
  \breve{F}\big( \det(Y), \breve{U} \taub \big)
\quad\tx{with}\quad
  \breve{U}
=
  \begin{psmatrix} 0 & 1 \\ 1 & 0 \end{psmatrix} U \begin{psmatrix} 0 & 1 \\ 1 & 0 \end{psmatrix}
\tx{,}
\end{gather}
where $\breve{U} \in \SL{2}(\RR)$ acts on $\taub$ by M\"obius transformations. Furthermore, we have
\begin{gather}
  \big(
  \breve{F} \big|^\sk_k\,
  \rot\big( \begin{psmatrix} -1 & 0 \\ 0 & 1 \end{psmatrix} \big)
  \big)(\det(Y), \taub)
=
  (-1)^{k-1}\,
  \breve{F}\big( \det(Y), -\ov{\taub} \big)
\tx{.}
\end{gather}
\end{proposition}
\begin{proof}
If $U \in \GL{2}(\ZZ)$, then
\begin{gather*}
  \big( \breve{F} \big|^\sk_k\, \rot(U) \big)(Y)
=
  \det(U)^{k-1}\, \breve{F}(U Y \rT U)
\tx{.}
\end{gather*}
Hence it suffices to express $U Y \rT U$ in terms of the coordinates in~\eqref{eq:def:Y-coordinates}.

Consider the case $U \in \SL{2}(\ZZ)$.  Recall the isomorphism of $\SL{2}(\RR)$-sets $\HS \cong \SL{2}(\RR) \slash \SO{2}(\RR)$ arising from the map $\SL{2}(\RR) \ra \HS,\, g \mto g i$. In particular, for any $g \in \SL{2}(\RR)$, we have $g i = \tau$ if and only if there exists a $h \in \SO{2}(\RR)$ such that $g h = \begin{psmatrix} \sqrt{y} & \xb \slash \sqrt{y} \\ 0 & 1 \slash \sqrt{y} \end{psmatrix}$. Composing this isomorphism with the map $\SL{2}(\RR) \mto \MatT{2}(\RR),\, g \mto s g s \, s \rT g s$, with $s:= \begin{psmatrix} 0 & 1 \\ 1 & 0 \end{psmatrix}$, we obtain an $\SL{2}(\RR)$-equivariant embedding
\begin{gather*}
  \taub
\lmto
  \begin{pmatrix} 1 \slash \sqrt{\yb} & 0 \\ \xb \slash \sqrt{\yb} & \sqrt{\yb} \end{pmatrix}
  \begin{pmatrix} 1 \slash \sqrt{\yb} & \xb \slash \sqrt{\yb} \\ 0 & \sqrt{\yb} \end{pmatrix}
=
  \begin{pmatrix}
  1 \slash \yb & \xb \slash \yb \\
  \xb \slash \yb & \xb \slash \yb + \yb
  \end{pmatrix}
\tx{,}
\end{gather*}
where $\SL{2}(\RR)$ acts on $\taub$ by M\"obius transformations and on the matrices on the right-hand side by $(g,m) \mto s g s\, m\, s \rT g s$. This establishes the first statement.

If $U = \begin{psmatrix} -1 & 0 \\ 0 & 1 \end{psmatrix}$, then
\begin{multline*}
  \begin{pmatrix} -1 & 0 \\ 0 & 1 \end{pmatrix}
  Y
  \begin{pmatrix} -1 & 0 \\ 0 & 1 \end{pmatrix}
=
  \sqrt{\det(Y)}\,
  \begin{pmatrix} -1 \slash \sqrt{\yb} & 0 \\ \xb \slash \sqrt{\yb} & \sqrt{\yb} \end{pmatrix}
  \begin{pmatrix} -1 \slash \sqrt{\yb} & \xb \slash \sqrt{\yb} \\ 0 & \sqrt{\yb} \end{pmatrix}
\\
=
  \sqrt{\det(Y)}\,
  \begin{pmatrix} 1 \slash \sqrt{\yb} & 0 \\ -\xb \slash \sqrt{\yb} & \sqrt{\yb} \end{pmatrix}
  \begin{pmatrix} -1 & 0 \\ 0 & 1 \end{pmatrix}\,
  \begin{pmatrix} -1 & 0 \\ 0 & 1 \end{pmatrix}
  \begin{pmatrix} 1 \slash \sqrt{\yb} & -\xb \slash \sqrt{\yb} \\ 0 & \sqrt{\yb} \end{pmatrix}
\tx{.}
\end{multline*}
This confirms the second statement, since $-\ov{\taub} = -\xb + i \yb$.
\end{proof}

Consider the Fourier series coefficient $c\big(F;\, \begin{psmatrix} 0 & 0 \\ 0 & 0 \end{psmatrix};\, Y\big)$ of a skew-harmonic Siegel modular form~$F$. Modular invariance of~$F$ implies that
\begin{gather*}
  c\big(F;\, \begin{psmatrix} 0 & 0 \\ 0 & 0 \end{psmatrix};\, Y\big)
\;=\;
  c\big(F;\, \begin{psmatrix} 0 & 0 \\ 0 & 0 \end{psmatrix};\, Y\big)
  \big|^\sk_k\, \rot(U)
\tx{.}
\end{gather*}
Express $Y$ in terms of $\det(Y)$ and $\taub$ as in~\eqref{eq:Y-coordinates:deduction}. Then Proposition~\ref{prop:rot-slash-action-on-zeroth-fourier-coefficient} (with $U = \begin{psmatrix} 1 & 0 \\ 1 & 1 \end{psmatrix}$) asserts that
\begin{multline*}
  c\big( F;\, \begin{psmatrix} 0 & 0 \\ 0 & 0 \end{psmatrix};\, \det(Y), \taub \big) 
=
  c\big( F;\, \begin{psmatrix} 0 & 0 \\ 0 & 0 \end{psmatrix};\, \det(Y), \taub \big) 
  \big|^\sk_k\, \rot\big( \begin{psmatrix} 1 & 0 \\ 1 & 1 \end{psmatrix} \big)
\\
=
  c\big( F;\,
  \begin{psmatrix} 0 & 0 \\ 0 & 0 \end{psmatrix};\,
  \det(Y), \begin{psmatrix} 1 & 1 \\ 0 & 1 \end{psmatrix} \taub
  \big) 
=
  c\big( F;\, \begin{psmatrix} 0 & 0 \\ 0 & 0 \end{psmatrix};\, \det(Y), \taub + 1 \big) 
\tx{,}
\end{multline*}
i.e., $c\big(F;\, \begin{psmatrix} 0 & 0 \\ 0 & 0 \end{psmatrix};\, Y\big)$ is periodic with respect to~$\xb$.

Multiplicity-one fails for~$T = \begin{psmatrix} 0 & 0 \\ 0 & 0 \end{psmatrix}$ (cf.\@ the initial remarks in Section~\ref{sec:kohnen-limit-zero-m}). The next proposition provides a replacement for it via a Fourier series expansion with respect to~$\xb$.
\begin{proposition}
\label{prop:multiplicity-one-for-skew-harmonic-fourier-coefficients-zero}
If the function $a^\sk_k(\begin{psmatrix} 0 & 0 \\ 0 & 0 \end{psmatrix};\,Y)$ is weight~$k$ skew-harmonic and also of moderate growth, then it can be written as a linear combination of
\begin{gather}
\label{eq:prop:multiplicity-one-for-skew-harmonic-fourier-coefficients-zero:det-factors}
  1
\quad\tx{and}\quad
  \det(Y)^{\frac{1-k}{2}}
  c\big( \begin{psmatrix} 0 & 0 \\ 0 & 0 \end{psmatrix}_{(1)}, \bullet;\, \taub \big)
\tx{,}
\end{gather}
where $c\big( \begin{psmatrix} 0 & 0 \\ 0 & 0 \end{psmatrix}_{(1)}, \bullet;\, \taub \big)$ has eigenvalue $(k-1)(2-k)$ under the elliptic weight\/~$0$ Laplace operator~$\Delta_0 = - y^2 (\partial_{\xb}^2 + \partial_{\yb}^2)$, i.e., $c\big( \begin{psmatrix} 0 & 0 \\ 0 & 0 \end{psmatrix}_{(1)}, \bullet;\, \taub \big)$ has spectral parameter $k-1$ or $2-k$.

Furthermore, assume that
\begin{gather}
\label{eq:prop:multiplicity-one-for-skew-harmonic-fourier-coefficients-zero:translation-invariance}
  c\big( \begin{psmatrix} 0 & 0 \\ 0 & 0 \end{psmatrix}_{(1)}, \bullet;\, \taub \big)
  \big|^\sk_k \rot\big( \begin{psmatrix} 1 & 0 \\ b & 1 \end{psmatrix} \big)
=
  c\big( \begin{psmatrix} 0 & 0 \\ 0 & 0 \end{psmatrix}_{(1)}, \bullet;\, \taub \big)
\end{gather}
for all $b \in \ZZ$. Then $c\big( \begin{psmatrix} 0 & 0 \\ 0 & 0 \end{psmatrix}_{(1)}, \bullet;\, \taub \big)$ has the following Fourier expansion:
\begin{multline}
\label{eq:prop:multiplicity-one-for-skew-harmonic-fourier-coefficients-zero:sub-fourier-expansion}
  c\big( \begin{psmatrix} 0 & 0 \\ 0 & 0 \end{psmatrix}_{(1)}, 0_{(0)} \big)\,
  \yb^{k-1}
  \,+\,
  c\big( \begin{psmatrix} 0 & 0 \\ 0 & 0 \end{psmatrix}_{(1)}, 0_{(1)} \big)\,
  \yb^{2-k}
\\
  +\,
  \sum_{\nb \in \ZZ \setminus \{0\}}
  c\big( \begin{psmatrix} 0 & 0 \\ 0 & 0 \end{psmatrix}_{(1)}, \nb \big)\,
  \big(4 \pi |\nb| \yb \big)^{\frac{1}{2}}
  K_{k-\frac{3}{2}}(2 \pi |\nb| \yb)
  e(\nb \xb)
\tx{.}
\end{multline}

Moreover, if
\begin{gather*}
  c\big( \begin{psmatrix} 0 & 0 \\ 0 & 0 \end{psmatrix}_{(1)}, \bullet;\, \taub \big)
  \big|^\sk_k \rot(U)
=
  c\big( \begin{psmatrix} 0 & 0 \\ 0 & 0 \end{psmatrix}_{(1)}, \bullet;\, \taub \big)
\end{gather*}
for all $U \in \GL{2}(\ZZ)$, then
\begin{gather}
\label{eq:prop:multiplicity-one-for-skew-harmonic-fourier-coefficients-zero:eisenstein}
  c\big( \begin{psmatrix} 0 & 0 \\ 0 & 0 \end{psmatrix}_{(1)}, \bullet;\, \taub \big)
=
  c\big( \begin{psmatrix} 0 & 0 \\ 0 & 0 \end{psmatrix}_{(1)}, \bullet \big)\,
  E_0(k-1, \taub)
=
  c\big( \begin{psmatrix} 0 & 0 \\ 0 & 0 \end{psmatrix}_{(1)}, \bullet \big)\,
  \sum_{\gamma \in \Gamma_\infty \backslash \SL{2}(\ZZ)} \Im(\gamma \taub)^{k-1}
\tx{,}
\end{gather}
where as before $\gamma \taub$ is the usual M\"obius transformation. In particular,
\begin{gather}
\label{eq:prop:multiplicity-one-for-skew-harmonic-fourier-coefficients-zero:eisenstein-coefficients}
\begin{alignedat}{2}
  c\big( \begin{psmatrix} 0 & 0 \\ 0 & 0 \end{psmatrix}_{(1)}, 0_{(0)} \big)
&=
  c\big( \begin{psmatrix} 0 & 0 \\ 0 & 0 \end{psmatrix}_{(1)}, \bullet \big)\,
\tx{;}
\\
  c\big( \begin{psmatrix} 0 & 0 \\ 0 & 0 \end{psmatrix}_{(1)}, 0_{(1)} \big)
&=
  c\big( \begin{psmatrix} 0 & 0 \\ 0 & 0 \end{psmatrix}_{(1)}, \bullet \big)\,
  \frac{\sqrt{\pi} \Gamma(k-\frac{3}{2}) \zeta(2k - 3)}
       {\Gamma(k-1) \zeta(2k-2)}
\tx{;}
\\
  c\big( \begin{psmatrix} 0 & 0 \\ 0 & 0 \end{psmatrix}_{(1)}, \nb \big)
&=
  c\big( \begin{psmatrix} 0 & 0 \\ 0 & 0 \end{psmatrix}_{(1)}, \bullet \big)\,
  \frac{\pi^{k-\frac{3}{2}}}{\Gamma(k-1) \zeta(2k-2)}\,
  |\nb|^{1-k} \sigma_{2k - 3}(|\nb|)
\tx{,}\quad
&&
  \tx{if $\nb \ne 0$.}
\end{alignedat}
\end{gather}
\end{proposition}
\begin{proof}
Note that \eqref{eq:prop:multiplicity-one-for-skew-harmonic-fourier-coefficients-zero:det-factors} follows directly from~(31) and~(32) of~\cite{maass-1953}, since $\det(Y)^{\frac{3}{2}-k}$ does not vanish under $\xi^\sk_k$. 
Using Proposition~\ref{prop:rot-slash-action-on-zeroth-fourier-coefficient}, we see that~\eqref{eq:prop:multiplicity-one-for-skew-harmonic-fourier-coefficients-zero:translation-invariance} implies translation invariance with respect to~$\xb$. Together with Equation~\eqref{eq:prop:multiplicity-one-for-skew-harmonic-fourier-coefficients-zero:det-factors} this yields~\eqref{eq:prop:multiplicity-one-for-skew-harmonic-fourier-coefficients-zero:sub-fourier-expansion}. Finally, to establish~\eqref{eq:prop:multiplicity-one-for-skew-harmonic-fourier-coefficients-zero:eisenstein}, we can use the fact that $c\big( \begin{psmatrix} 0 & 0 \\ 0 & 0 \end{psmatrix}_{(1)}, \bullet;\, \taub \big)$ is an elliptic Maa\ss\ form with Laplace eigenvalue~$(k-1)(2-k)$.
The space of such Maa\ss\ forms is one-dimensional by Theorem~31 of~\cite{maass-1964} and spanned by $E_0(k-1,\taub)$. The Fourier series expansion of~$E_0(k-1,\taub)$ can be recovered from Lemma~\ref{la:elliptic-eisenstein-series-fourier-expansion} by using the relation $K_\nu(z) = \sqrt{\pi \slash 2 z} W_{0,\nu}(2z)$.
\end{proof}

\paragraph{Normalization of and notation for Fourier series coefficients}

For the remainder of this paper, if~$T \ne 0$ is non-degenerate or positive semi-definite, we set
\begin{gather}
\label{eq:def:fourier-coefficients-siegel-modular-forms-type-skew}
  a^\sk_k(T Y)
:=
  a^\sk_k(T;\,Y)
\tx{,}
\end{gather}
suppressing the dependence on the signature of~$T$ from our notation. For convenience, we will also use the notation
\begin{gather}
  a^\sk_k(T Z)
:=
  a^\sk_k(T Y) e(T X)
\tx{.}
\end{gather}

For positive semi-definite~$T$, we choose the normalization
\begin{gather}
\label{eq:def:fourier-coefficients-skew-positive-semi-definite}
  a^\sk_k(T Y)
:=
  \big( 4 \pi \trace(TY) \big)^{-\frac{k}{2}}
  W_{\frac{1-k}{2}, \frac{k-1}{2}}\big( 4 \pi \trace(TY) \big)
\tx{.}
\end{gather}

If $T$ is negative semi-definite and degenerate, then we have to distinguish two Fourier coefficients. We follow a convention from Eisenstein series and amend the index~$T$ with subscripts~$(0)$ and~$(1)$:
\begin{gather}
\label{eq:def:fourier-coefficients-skew-negative-semi-definite}
\begin{alignedat}{2}
  a^\sk_k(T_{(0)} Y)
\;&:=\;
  a^\sk_k(T_{(0)}; Y)
\;&&:=\;
  \big( 4 \pi |\trace(T Y)| \big)^{-\frac{k}{2}}
  W_{\frac{k-1}{2}, \frac{k-1}{2}}\big( 4 \pi |\trace(T Y)| \big)
\quad\tx{and}
\\
  a^\sk_k(T_{(1)} Y)
\;&:=\;
  a^\sk_k(T_{(1)}; Y)
\;&&:=\;
  \det(Y)^{\frac{3}{2} - k}\,
  |4 \pi \trace(T Y)|^{k - 2} e^{-2 \pi |\trace(T Y)|}
\tx{.}
\end{alignedat}
\end{gather}
Also in this case, we write
\begin{gather}
  a^\sk_k(T_{(0)} Z)
:=
  a^\sk_k(T_{(0)} Y) e(T X)
\quad\tx{and}\quad
  a^\sk_k(T_{(1)} Z)
:=
  a^\sk_k(T_{(1)} Y) e(T X)
\tx{.}
\end{gather}

For~$T = \begin{psmatrix} 0 & 0 \\ 0 & 0 \end{psmatrix}$, we agree on the following notation for the functions in Proposition~\ref{prop:multiplicity-one-for-skew-harmonic-fourier-coefficients-zero}:
\begin{gather}
\label{eq:def:fourier-coefficients-skew-zero}
\begin{aligned}
  a^\sk_k\big( \begin{psmatrix} 0 & 0 \\ 0 & 0 \end{psmatrix}_{(0)};\, Y \big)
\;&:=\;
  1
\tx{,}
\\
  a^\sk_k\big( \begin{psmatrix} 0 & 0 \\ 0 & 0 \end{psmatrix}_{(1)}, 0_{(0)};\, Y \big)
\;&:=\;
  \det(Y)^{\frac{1-k}{2}} \yb^{k-1}
\;=\;
  y^{1-k}
\tx{,}
\\
  a^\sk_k\big( \begin{psmatrix} 0 & 0 \\ 0 & 0 \end{psmatrix}_{(1)}, 0_{(1)};\, Y \big)
\;&:=\;
  \det(Y)^{\frac{1-k}{2}} \yb^{2-k}
\;=\;
  \det(Y)^{\frac{3}{2}-k} y^{k-2}
\tx{,}
\\
  a^\sk_k\big( \begin{psmatrix} 0 & 0 \\ 0 & 0 \end{psmatrix}_{(1)}, \nb;\, Y \big)
\;&:=\;
  \det(Y)^{\frac{1-k}{2}}\,
  (4 \pi |\nb| \yb)^{\frac{1}{2}} K_{k-\frac{3}{2}}(2 \pi |\nb| \yb) e(\nb \xb)
\quad
  \tx{($\nb \ne 0$).}
\end{aligned}
\end{gather}
To emphasize the role of the last three coefficients as Fourier coefficients of a modular form that depend only on~$Y$, we set
\begin{gather}
\label{eq:def:fourier-coefficients-skew-zero:sub-coefficients}
\begin{aligned}
  a^\sk_k\big( \begin{psmatrix} 0 & 0 \\ 0 & 0 \end{psmatrix}_{(1)}, 0_{(0)};\, \yb \big)
\;&:=\;
  \yb^{k-1}
\tx{,}
\\
  a^\sk_k\big( \begin{psmatrix} 0 & 0 \\ 0 & 0 \end{psmatrix}_{(1)}, 0_{(1)};\, \yb \big)
\;&:=\;
  \yb^{2-k}
\tx{,}
\\
  a^\sk_k\big( \begin{psmatrix} 0 & 0 \\ 0 & 0 \end{psmatrix}_{(1)}, \nb;\, \yb \big)
\;&:=\;
  (4 \pi |\nb| \yb)^{\frac{1}{2}} K_{k-\frac{3}{2}}(2 \pi |\nb| \yb)
\quad
  \tx{($\nb \ne 0$).}
\end{aligned}
\end{gather}
Note that the last coefficient does not feature~$e(\nb \xb)$, while the corresponding one in~\eqref{eq:def:fourier-coefficients-skew-zero} does.

We do not give an explicit formula for $a^\sk_k(T Y)$ if $T= \begin{psmatrix} n & r \slash 2 \\ r \slash 2 & m \end{psmatrix}$ is indefinite or negative definite. If~$T$ is indefinite and $m > 0$, we normalize it via the Kohnen limit process~$\Klim^\sk_{k,m}$ defined in Section~\ref{sec:kohnen-limit-nonzero-m}:
\begin{gather}
\label{eq:def:fourier-coefficients-skew-indefinite-positive-m}
  \Klim^\sk_{k,m}\Big(
  a^\sk_k\big( \begin{psmatrix} n & r \slash 2 \\ r \slash 2 & m \end{psmatrix} Y \big)
  \Big)
\;=\;
  a^\Jsk_m\big( n, r; y \big)
  e(r i v)
\tx{,}
\end{gather}
where $a^\Jsk_m(n, r; y)$ will be defined in~\eqref{eq:skew-jacobi-fourier-coefficient}. Since $a^\sk_k(T Y)$ depends only on the product of~$T$ and~$Y$, since the Kohnen limit process is injective for positive~$m$ (see Proposition~\ref{prop:kohnen-limit-positive-m-injective}), and since the Kohnen limit process is covariant with respect to the action of~$\diag(1,a,1,1 \slash a)$ (see Proposition~\ref{prop:kohnen-limit-positive-m-covariance}), this fixes a unique normalization for all indefinite~$T$. Similarly, for negative definite~$T$ (and consequentially~$m < 0$), we choose a normalization by
\begin{gather}
\label{eq:def:fourier-coefficients-skew-negative-definite}
  \Klim^\sk_{k,m}\Big(
  a^\sk_k\big( \begin{psmatrix} n & r \slash 2 \\ r \slash 2 & m \end{psmatrix} Y \big)
  \Big)
\;=\;
  a^{\cJ\sk}_{2-k,|m|}(-n,-r; y)\, e(r i v)
\tx{,}
\end{gather}
where $a^{\cJ\sk}_{2-k,|m|}(-n,-r; y)$ will be defined in~\eqref{eq:def:fourier-coefficients-almost-skew-Jacobi-nonzero}.

\section{Jacobi forms}
\label{sec:jacobi-forms}

\subsection{Preliminaries}

Eichler and Zagier's  book~\cite{eichler-zagier-1985} is the standard reference to the classical theory of Jacobi forms. Following our global convention, we assume throughout that $k > 3$ is an odd integer, although several definitions and statements in this section extend to $k \in \ZZ$.

\paragraph{The Jacobi group}

Let $\GJ := \SL{2}(\RR)  \ltimes (\RR^2 \ltimes \RR)$ be the real Jacobi group (see~\cite{eichler-zagier-1985} or \cite{berndt-schmidt-1998}, for details), and let $g^\rmJ=\big(\begin{psmatrix} a & b \\ c & d \end{psmatrix},\, \lambda, \mu, \kappa \big)\in\GJ$ be a typical element. The extended Jacobi group as a set is
\begin{gather*}
\SL{2}(\ZZ) \ltimes ( \ZZ^2 \ltimes \ZZ )
\subset
\GJ
\tx{.}
\end{gather*}
The usual Jacobi group $\GaJ$ is the quotient of the extended Jacobi group by its center:
\begin{gather*}
\GaJ
=
\SL{2}(\ZZ) \ltimes \ZZ^2
\tx{.}
\end{gather*}
We write $\ga^\rmJ=\big(\begin{psmatrix} a & b \\ c & d \end{psmatrix},\, \lambda, \mu\big)\in\GaJ$ for a typical element. Let
\begin{gather*}
  \GaJ_\infty
:=
  \big\{
  \big( \begin{psmatrix} a & b \\ 0 & d \end{psmatrix}, 0, \mu \big) \in \GaJ
  \big\}
\end{gather*}
be the Jacobi parabolic subgroup of $\GaJ$.

Recall that the real Jacobi group can be embedded into~$\Sp{2}(\RR)$ (see page~74 of~\cite{eichler-zagier-1985}) by
\begin{gather}
\label{eq:embedding-of-jacobi-group}
  \up :\,
  \GJ \ra \Sp{2}(\RR)
,\;
  \big(
  \begin{psmatrix} a & b \\ c & d \end{psmatrix},\,
  \lambda, \mu, \kappa
  \big)
\lmto
  \begin{psmatrix}
  a & 0 & b & \mu \\
  a \lambda + c \mu & 1 & b \lambda + d \mu & \kappa \\
  c & 0 & d & - \lambda \\
  0 & 0 & 0 & 1
  \end{psmatrix}
\tx{,}
\end{gather}
extending the embedding $\up :\, \SL{2}(\RR) \ra \Sp{2}(\RR)$ in~\eqref{eq:def:rot-up-down-embedding}.

The real Jacobi group acts on the Jacobi upper half space
\begin{gather*}
\HS^\rmJ
:=
\HS \times \CC
=
\big\{
(\tau, z) \in \CC^2 \,:\,\
\Im(\tau) > 0
\big\}
\end{gather*}
by
\begin{gather*}
\big(
\begin{psmatrix} a & b \\ c & d \end{psmatrix},\,
\lambda, \mu, \kappa
\big) \cdot (\tau, z)
:=
\big( \frac{a \tau + b}{c \tau + d},\, \frac{z + \lambda \tau + \mu}{c \tau + d} \big)
\tx{.}
\end{gather*}
This action factors through~$\GaJ$. 

\paragraph{Classical slash actions}

If $m \in \CC$, then
\begin{gather*}
\alpha^\rmJ_m\big( g^\rmJ, (\tau, z) \big)
:=
  e\Big(m \Big(
  \frac{-c(z + \lambda \tau + \mu)^2}{c\tau + d}
  + 2 \lambda z + \lambda^2 \tau
  + \lambda \mu + \kappa
  \Big)\Big)
\end{gather*}
is a $1$-cocyle with values in $\rmC^\infty( \HS^\rmJ )$. If $m \in \ZZ$, then it factors through $\GaJ$.

The skew-Jacobi slash action of the real Jacobi group is defined for $\phi :\, \HS^\rmJ \ra \CC$ and $m \in \CC$:
\begin{gather*}
\big( \phi \big|^\Jsk_{k, m}\, g^\rmJ \big)(\tau, z)
:=
|c \tau + d|^{-1}
\ov{(c \tau + d)}^{1-k}
\alpha^\rmJ_m\big(g^\rmJ, (\tau, z) \big) \,
\phi( g^\rmJ (\tau, z) )
\tx{.}
\end{gather*}
If $m\in\ZZ$, then this action factors through~$\GaJ$. 

\paragraph{Extension of the slash action}

Observe~\eqref{eq:embedding-of-jacobi-group} and consider the real Jacobi group as a subgroup of~$\Sp{2}(\RR)$. Write $\GL{2}^\downarrow(\RR) \subset \GL{2}(\RR)$ for the lower triangular matrices of positive determinant. These matrices can be embedded into~$\Sp{2}(\RR)$ via the map~$\rot$ in~\eqref{eq:def:rot-up-down-embedding}. Explicitly, viewed as a subgroup of $\Sp{2}(\RR)$, the group $\GJ \rot\big(\GL{2}^\downarrow(\RR) \big) \subset \Sp{2}(\RR)$ consists of matrices of the form
\begin{gather*}
\begin{psmatrix}
* & 0 & * & * \\
* & * & * & * \\
* & 0 & * & * \\
0 & 0 & 0 & *
\end{psmatrix}
\tx{.}
\end{gather*}

We wish to assign a meaning to
\begin{gather*}
\phi \big|^\Jsk_{k,m} g
\quad
\tx{for any}\;
g \in \GJ\rot\big(\GL{2}^\downarrow(\RR)\big) \subset \Sp{2}(\RR)
\tx{.}
\end{gather*}
Note that representatives for the quotient $\GJ \backslash \GJ \rot(\GL{2}^\downarrow(\RR))$ are given by diagonal matrices $g = \diag(1,a,\,1,d)$ with $a,d \in \RR_{> 0}$ and $ad = 1$. For any such $g$, $\phi :\, \HJ \ra \CC$, and $m \in \CC$ set
\begin{gather}
\label{eq:rot-quotient-action}
  \big( \phi \big|^\Jsk_{k,m} g \big) (\tau, z)
:=
  a^k \phi(\tau, a z)
\tx{,}
\end{gather}
which is not an action in the strict sense. Specifically, we have the covariance
\begin{gather*}
\big( \phi \big|^\Jsk_{k,m} g^\rmJ \big) \big|^\Jsk_{k,m} g
\;=\;
\big( \phi \big|^\Jsk_{k,m} g \big)
\big|^\Jsk_{k,a^2 m} (g^{-1} g^\rmJ\, g)
\quad
\tx{for all}\;
g^\rmJ \in \GJ
\tx{.}
\end{gather*}

If $g' \in \GJ \rot(\GL{2}^\downarrow(\RR))$, then decompose $g'= g^\rmJ g$ with $g^\rmJ \in \GJ$ and $g = \diag(1,a,\,1,d)$ as before, and define
\begin{gather}
\label{eq:slash-extension}
\phi \big|^\Jsk_{k,m} g'
:=
\big( \phi \big|^\Jsk_{k,m} g^\rmJ \big) \big|^\Jsk_{k,m} g
\tx{.}
\end{gather}

\paragraph{Differential operators}

Recall that $k > 3$ is throughout an odd integer. In this section, we consider arbitrary integral weights, which we denote by~$\kappa$. The heat operator of index~$m \in \CC^\times$ is defined by
\begin{gather*}
  \bbL_m
:=
  8 \pi i m \partial_\tau
  -
  \partial_z^2
\tx{.}
\end{gather*}

The skew-Jacobi raising operator
\begin{gather}
\label{eq:def:jacobi-raising-operator}
  \rmR^{\rmJ\sk}_\kappa
\;:=\;
  \rmR^{\rmJ\sk}
\;:=\;
  \partial_{\ov{\tau}}
  + v y^{-1} \partial_{\ov{z}}
  + \tfrac{i}{2}(\kappa-\tfrac{1}{2}) y^{-1}
\tx{,}
\end{gather}
denoted by~$X^{\mathrm{sk};\kappa,m}_-$ in~\cite{bringmann-raum-richter-2015}, is covariant with respect to the skew-Jacobi slash action. More precisely, if $\phi \in \rmC^\infty( \HS^\rmJ )$, then
\begin{gather*}
  \rmR^{\rmJ\sk}_\kappa \big( \phi \big|^\Jsk_{\kappa,m}\, g^\rmJ \big)
=
  \big( \rmR^{\rmJ\sk}_\kappa \phi \big) \big|^\Jsk_{\kappa+2,m}\, g^\rmJ
\quad
\tx{for all $g^\rmJ \in \GJ$.}
\end{gather*}
Similarly, the skew-Jacobi lowering operator
\begin{gather}
\label{eq:def:jacobi-lowering-operator}
  \rmL^{\rmJ\sk}_\kappa
\;:=\;
  \rmL^{\rmJ\sk}
\;:=\;
  y^2 \partial_\tau + y v \partial_z + 2 i m v^2 - \tfrac{i}{4} y
\tx{,}
\end{gather}
denoted by~$X^{\mathrm{sk};\kappa,m}_+$ in~\cite{bringmann-raum-richter-2015}, satisfies
\begin{gather*}
  \rmL^{\rmJ\sk} \big( \phi \big|^\Jsk_{\kappa,m}\, g^\rmJ \big)
=
  \big( \rmL^{\rmJ\sk} \phi \big) \big|^\Jsk_{\kappa-2,m}\, g^\rmJ
\quad
\tx{for all $g^\rmJ \in \GJ$.}
\end{gather*}
As differential operators, we have the following commutation law:
\begin{gather}
\label{eq:jacobi-differential-operator-commutator}
  \big[ \rmL^{\rmJ\sk},\, \rmR^{\rmJ\sk} \big]
=
  \rmL^\Jsk_{\kappa+2} \rmR^\Jsk_\kappa
  \,-\,
  \rmR^\Jsk_{\kappa-2} \rmL^\Jsk_\kappa
=
  \tfrac{1}{4} \big( \tfrac{3}{2} - \kappa \big)
\tx{.}
\end{gather}

The definition of harmonic Maa\ss-Jacobi forms will require the following operators:
\begin{gather}
\label{eq:def:jacobi-usual-casmir-lowering-raising-operator}
  \Delta^\rmJ_\kappa
:=
  \rmR^\rmJ_{\kappa-2} \rmL^\rmJ
\tx{,}\quad
  \rmR^\rmJ_\kappa
:=
  y^{-\frac{3}{2} - \kappa}\, \rmL^\Jsk_{1-\kappa}\, y^{\kappa - \frac{1}{2}}
\tx{,}\quad
  \rmL^\rmJ_\kappa
:=
  y^{\frac{5}{2} - \kappa}\, \rmR^\Jsk_{1-\kappa} \, y^{\kappa - \frac{1}{2}}
\tx{.}
\end{gather}
Observe that the last two operators agree with the corresponding ones in~\cite{bringmann-raum-richter-2015} only up to nonzero scalar multiples. The first one, the weight~$\kappa$ cubic Jacobi-Casimir operator, only agrees with the one in~\cite{bringmann-richter-2010,bringmann-raum-richter-2015} for functions that are holomorphic in~$z$. The Jacobi $\xi$-operator acting on $\bbJ_{3-k,m}$, can be expressed as
\begin{gather}
\label{eq:def:jacobi-xi-operator}
  \xi^\rmJ_{3-k}
:=
  -2i\,
  \rmR^\Jsk_{k-2} \circ y^{\frac{5}{2}-k}
\tx{.}
\end{gather}

For $h \in \ZZ_{> 0}$, we set $\rmR^{\rmJ\sk\,h}_\kappa := \rmR^{\rmJ\sk}_{\kappa+2h-2} \circ\, \cdots \,\circ \rmR^{\rmJ\sk}_\kappa$. By abuse of notation, we sometimes suppress the weight from the notation, and simply write $\rmR^{\rmJ\sk\,h}$. We use analogous notation for the other covariant differential operators in~\eqref{eq:def:jacobi-lowering-operator} and~\eqref{eq:def:jacobi-usual-casmir-lowering-raising-operator}.

\subsection{Skew-holomorphic Jacobi forms}

Skoruppa~\cite{skoruppa-1990, skoruppa-1990b} introduced skew-holo\-morphic Jacobi forms.
\begin{definition}
\label{def:skew-holomorphic-jacobi-forms}
Let $m \in \ZZ$. A smooth function $\phi :\, \HJ \ra \CC$ is a skew-holomorphic Jacobi form of weight~$k$ and index~$m$ if
\begin{enumerateroman}
\item
$\phi \big|^\Jsk_{k,m}\, \gaJ = \phi$ for all $\gaJ \in \GaJ$.

\item
\label{def:skew-holomorphic-jacobi-forms:it:analytic-type}
$\bbL_m\,\phi = 0$ and $\phi(\tau, z)$ is holomorphic in~$z$.

\item
\label{def:skew-holomorphic-jacobi-forms:it:growth}
For all $\lambda, \mu \in \QQ$, the function $\phi \big|^\Jsk_{k,m} \big( \begin{psmatrix} 1 & 0 \\ 0 & 1 \end{psmatrix}, \lambda, \mu, 0 \big)(\tau, 0) = \cO(1)$ as~$y \ra \infty$.
\end{enumerateroman}
If, in addition, the asymptotic in~\ref{def:skew-holomorphic-jacobi-forms:it:growth} vanishes, then $\phi$ is a cusp form.
\end{definition}
We write $\rmJ^\sk_{k,m}$ for the space of skew-holomorphic Jacobi forms of weight~$k$ and index~$m$. Note that if $m < 0$, then $\rmJ^\sk_{k,m} = \{0\}$.

The skew-holomorphic Jacobi-Eisenstein
\begin{gather}
\label{eq:def:skew-holomorphic-jacobi-eisenstein-series}
E^{\rmJ\sk}_{k,m}(\tau, z)
:=
\sum_{\gamma \in \GaJ_\infty \backslash \GaJ}
1 \big|^\Jsk_{k,m}\,\gamma
\end{gather}
is a basic example of a form in~$\rmJ^\sk_{k,m}$ if $m > 0$.

\paragraph{Fourier series expansions}

Skew-holomorphic Jacobi forms have Fourier series expansions of the form
\begin{gather}
\label{eq:skew-jacobi-fourier-expansion}
\phi(\tau, z)
=
\sum_{\substack{n, r \in \ZZ \\ D:= 4 m n  - r^2 \le 0}}
\hspace{-1.5em}
  c(\phi;\, n, r)\,
  a^\Jsk_m(n,r; y)
  e( n x + r z)
\tx{,}
\end{gather}
where
\begin{gather}
\label{eq:skew-jacobi-fourier-coefficient}
  a^\Jsk_m(n,r; y)
:=
  \exp\big( - 2 \pi (n - \tfrac{D}{2 m}) y \big)
\end{gather}
is as in Section~\ref{ssec:siegel-modular-forms-preliminaries}. If the Fourier series expansion in~\eqref{eq:skew-jacobi-fourier-expansion} is only over $D<0$, then $\phi$ is a skew-holomorphic Jacobi cusp form of weight~$k$ and index~$m$.

We use the action in~\eqref{eq:slash-extension} of $\rot\big( \GL{2}^\downarrow(\RR) \big) \subset \rot\big( \GL{2}(\RR) \big)$ to compare Fourier series coefficients of different skew-holomorphic Jacobi forms.
\begin{lemma}
\label{la:skew-holomorphic-jacobi-forms-matching-of-fourier-coefficients}
Suppose that $U \in \GL{2}^\downarrow(\RR)$ such that
\begin{gather*}
  \rT U \begin{psmatrix} n & r \slash 2 \\ r \slash 2 & m \end{psmatrix} U
=
  \begin{psmatrix} m & r \slash 2 \\ r \slash 2 & n \end{psmatrix}
\tx{.}
\end{gather*}
Then
\begin{gather*}
  m^{\frac{-k}{2}}
  a^\Jsk_m(n,r; y)
  e(r i v)
  \big|^\Jsk_{k,m}
  \rot(U)
=
  n^{\frac{-k}{2}}
a^\Jsk_n(m,r; y)
e(r i v)
\tx{,}
\end{gather*}
where $z = u + iv$ is as in Section~\ref{ssec:siegel-modular-forms-preliminaries}. In particular, if $\phi\in\rmJ^\sk_{k,m}$ and $\phi'\in\rmJ^\sk_{k,n}$, then $c(\phi;\, n, r) = c(\phi';\, m, r)$ if and only if
\begin{gather*}
  c(\phi;\, n, r)
  m^{\frac{-k}{2}}
  a^\Jsk_m(n,r; y)
  e(r i v)
  \big|^\Jsk_{k,m}
  \rot(U)
=
  c(\phi';\, m, r)
  n^{\frac{-k}{2}}
  a^\Jsk_n(m,r; y)
  e(r i v)
\tx{.}
\end{gather*}
\end{lemma}
\begin{proof}
This follows from the transformation behavior of $a^\Jsk_m$.
\end{proof}

\paragraph{Hecke operators for skew-holomorphic Jacobi forms} 

If $l>0$, then the definitions of the usual Jacobi-Hecke operator $U_l$ and $V_l$ extend to the skew-holomorphic case.

If $\phi\in\rmJ^\sk_{k,m}$ and $l > 0$, then
\begin{gather}
\label{eq:def:U-hecke-operator}
  \big( \phi \big|^\Jsk_{k,m}\, U_l \big)(\tau, z)
:=
  \phi(\tau, lz)
\tx{}
\end{gather}
and
\begin{gather}
\label{eq:def:V-hecke-operator-positive-l}
  \big( \phi \big|^\Jsk_{k,m}\, V_l \big)(\tau, z)
:=
  l^{k-1} \sum_{\substack{a \isdiv l\\b \pmod{l \slash a}}}
  (l \slash a)^{-k}\,
  \phi\big( \frac{a \tau + b}{l \slash a},\, a z \big)
\tx{.}
\end{gather}
Analogous to Theorem~4.2 of~\cite{eichler-zagier-1985} one finds that if $\phi \in \rmJ^\sk_{k,1}$, then
\begin{gather}
\label{eq:V-hecke-operator-coefficient-formula}
c\big( \phi \big|^\Jsk_{k,1}\, V_m;\,  n,r \big)
\;=\;
\sum_{0 < d \isdiv \gcd(n,r,m)} \hspace{-1em}
d^{k-1} c\Big(\phi;\, \frac{mn}{d^2}, \frac{r}{d} \Big)
\tx{.}
\end{gather}

\subsection{Harmonic Maa\ss-Jacobi forms and almost skew-harmonic Maa\ss-Jacobi forms}
\label{ssec:harmonic-jacobi-forms}

We start by defining harmonic Maa\ss-Jacobi forms, which were first introduced in~\cite{bringmann-richter-2010}.
\begin{definition}
\label{def:harmonic-jacobi-forms}
For $m \in \ZZ$, a harmonic Maa\ss-Jacobi form of weight~$3-k$ and Jacobi index~$m$ is a smooth function~$\phi :\, \HJ \ra \CC$ that satisfies
\begin{enumerateroman}
\item
$\phi \big|^\rmJ_{3-k,m}\, \gamma = \phi$ for all $\gamma \in \GaJ$.

\item
\label{it:def:harmonic-jacobi-forms:analytic-type}
$\Delta^\rmJ_{3-k}\, \phi = 0$ and $\phi(\tau, z)$ is holomorphic in~$z$.

\item
For all $\lambda, \mu \in \QQ$, the function $\phi \big|^\rmJ_{3-k,m} \big(\begin{psmatrix} 1 & 0 \\ 0 & 1 \end{psmatrix}, \lambda, \mu, 0\big)(\tau, 0) = \cO(y^a)$ for some~$a \in \RR$ as $y \ra \infty$.
\end{enumerateroman}
\end{definition}
The space of such functions will be denoted by~$\bbJ_{3-k,m}$. Note that if $m < 0$, then $\bbJ_{3-k,m} = \{0\}$. By definition, elements of~$\bbJ_{3-k,m}$ have moderate growth but no singularities, and hence (using the theta decomposition, spectral theory, and our assumption~$k > 3$) one can show that they are scalar multiples of Eisenstein series.

\paragraph{Fourier expansions of harmonic Maa\ss-Jacobi forms}

Harmonic Maa\ss-Jacobi forms of index~$m$ and weight~$3-k$ have Fourier series expansions of the form (cf.\ \cite{bringmann-raum-richter-2015})
\begin{multline}
  \sum_{\substack{n, r \in \ZZ \\ D = 0}}
  c\big( \phi;\, n_{(0)},r \big)\,
  a^\bbJ_{3-k,m}(n_{(0)},r;\,y) e( n x + r z )
\\
  +\;
  \sum_{\substack{n, r \in \ZZ \\ D= 0}}
  c\big( \phi;\, n_{(1)},r \big)\,
  a^\bbJ_{3-k,m}(n_{(1)},r;\,y) e( n x + r z )
\\
  +\;
  \sum_{\substack{n, r \in \ZZ \\ D \ne 0}}
  c\big( \phi;\, n,r \big)
  a^\bbJ_{3-k,m}(n,r;\, y) e(n x + r z)
\tx{,}
\end{multline}
where we always write $D := 4mn - r^2$, and where
\begin{alignat}{2}
  a^\bbJ_{3-k,m}(n_{(0)},r;\, y)
&:=
  e(n i y)
=
  \exp\big( - \pi \frac{r^2}{2 m} y \big)
\quad
&&
  \tx{for $D = 0$,}
\\
  a^\bbJ_{3-k,m}(n_{(1)},r;\, y)
&:=
  y^{k-\frac{3}{2}}\,
  e(n i y)
=
  y^{k-\frac{3}{2}}\,
  \exp\big( - \pi \frac{r^2}{2 m} y \big)
\quad
&&
  \tx{for $D = 0$,}
\\
  a^\bbJ_{3-k,m}(n,r;\, y)
&:=
  e(n i y)
\quad
&&
  \tx{for $D > 0$.}
\\
\label{eq:def:harmonic-maass-jacobi:fourier-coefficient:negative}
  a^\bbJ_{3-k,m}(n,r;\, y)
&:=
  \Gamma\big( \frac{3}{2} - k,\, \frac{\pi |D| y}{m} \big)\,
  e(n i y)
\quad
&&
  \tx{for $D < 0$.}
\end{alignat}

\paragraph{Almost harmonic Maa\ss-Jacobi forms}

Multiplication by $y^{\frac{3}{2} - k}$ and $(k-2)$-fold iteration of $\rmL^\Jsk$ yields another type of real-analytic skew-Jacobi form. Because of the particular role these Jacobi forms play in the context of the Kohnen limit process for negative Fourier-Jacobi indices, we introduce extra notation for them:
\begin{gather}
  \cJ^\sk_{2-k, m}
\;:=\;
  \big( \rmL^\Jsk \big)^{k-2} \big(
  y^{\frac{5}{2}-k}\, \bbJ_{3-k,m}
  \big)
=
  \rmR^\bbJ_{3-k} \big(
  \bbJ_{3-k,m}
  \big)
\tx{,}
\end{gather}
where
\begin{gather}
\label{eq:def:harmonic-maass-jacobi-raising-operator}
  \rmR^\bbJ_{3-k}
\;:=\;
  \frac{(-1)^{\frac{k-2}{2}} 2^{k-1} \pi^{\frac{1}{2}}}{\Gamma(k-\tfrac{1}{2})}\,
  \big( \rmL^{\Jsk} \big)^{k-2} \circ y^{\frac{5}{2}-k}
\tx{.}
\end{gather}
Note that $\rmR^\bbJ_{3-k}$ is normalized such that it maps the standard Eisenstein series in~$\bbJ_{3-k,m}$ to the standard Eisenstein series in~$\cJ^\sk_{2-k, m}$  (see \eqref{eq:almost-skew-harmonic-eisenstein-series-raising-image}).

\begin{remark}
\label{rem:almost-skew-harmonic-jacobi-forms}
The elements of $\cJ^\sk_{2-k,m}$ could be called almost skew-harmonic Maa\ss-Jacobi forms, i.e., the skew-Jacobi form analogues of almost harmonic Maa\ss\ forms in~\cite{bringmann-folsom-2014}. When combining the theta decomposition of elements in $\cJ^\sk_{2-k,m}$ with the half-integral weight analogue of Theorem~31 in~\cite{maass-1964}, one finds that $\cJ^\sk_{2-k, m}$ consists only of Eisenstein series.

Similarly, $\bbJ_{3-k,m}$ consists only of Eisenstein series. In particular, if $m=1$, then $\bbJ_{3-k,1}$ is spanned by the Eisenstein series~$E^{\bbJ}_{3-k, 1}$ defined in~\eqref{eq:def:skew-harmonic-jacobi-eisenstein-series}.
\end{remark}

In the following we only consider weights and Laplace eigenvalues for which the lowering operator is invertible. Hence we obtain an isomorphism between $\cJ^\sk_{2-k,m}$ and $\bbJ_{3-k,m}$.
\begin{proposition}
\label{prop:harmonic-and-almost-skew-harmonic-jacobi-forms}
The following diagram is commutative up to multiplication with nonzero scalars.
\begin{center}
\begin{tikzpicture}
\matrix(m)[matrix of math nodes,
column sep=10em,
text height=1.5em, text depth=1.25ex]{%
\bbJ_{3-k,m} & \cJ^\sk_{2-k,m} \\
};

\path
(m-1-1) edge[->, bend left = 10] node[above=.2em] {$\rmR^\bbJ_{3-k}$} (m-1-2)
(m-1-2) edge[->, bend left = 10] node[below=.2em] {$y^{k-\frac{3}{2}} \circ \big( \rmR^\Jsk \big)^{k-2} $} (m-1-1)
;
\end{tikzpicture}
\end{center}
\end{proposition}
\begin{proof}
It suffices to show that
\begin{gather}
\label{eq:prop:harmonic-and-almost-skew-harmonic-jacobi-forms:operator}
  y^{k-\frac{3}{2}} \circ \big( \rmR^\Jsk \big)^{k-2}
  \,\circ\,
  \big( \rmL^\Jsk \big)^{k-2} \circ y^{\frac{3}{2}-k}
=
  \rmL^{\rmJ\,k-2} \circ \rmR^{\rmJ\,k-2}_{3-k}
\end{gather}
acts on $\bbJ_{3-k,m}$ by a nonzero scalar. For convenience, if $\kappa \in \ZZ$, then we set
\begin{gather*}
  \alpha_\kappa
:=
  \rmL^\rmJ_{\kappa+2} \rmR^\rmJ_\kappa - \rmR^\rmJ_{\kappa-2} \rmL^\rmJ_\kappa
=
  \tfrac{1}{4} \big( \kappa - \tfrac{3}{2} \big)
\tx{.}
\end{gather*}
 
Observe that for $h \in \ZZ_{> 0}$, $h' \in \ZZ_{\ge 0}$, and $\kappa \in \ZZ$, we have
\begin{gather*}
  \rmR^{\rmJ\, h'}_{\kappa + 2h - 2} \circ \rmL^\rmJ \circ \rmR^{\rmJ\,h}_\kappa
=
  \rmR^{\rmJ\, h'+1}_{k+2h-4} \circ \rmL^\rmJ \circ \rmR^{\rmJ\,h-1}_\kappa
  \,+\,
  \alpha_{\kappa + 2h-2}
  \rmR^{\rmJ\, h+h'-1}_\kappa
\tx{.}
\end{gather*}
Inducing on~$h \in \ZZ_{\ge 0}$, while keeping $h + h'$ fixed, we find that
\begin{gather*}
  \rmL^\rmJ \circ \rmR^{\rmJ\,h}_\kappa
=
  \rmR^{\rmJ\,h}_{\kappa-2} \circ \rmL^\rmJ 
  +
  \sum_{i=0}^{h-1} \alpha_{\kappa + 2i}
  \rmR^{\rmJ\,h-1}_\kappa
\tx{.}
\end{gather*}
With this formula in place, we can once more induce on $h \in \ZZ_{\ge 0}$ to discover that
\begin{gather*}
  \rmL^{\rmJ\,h} \circ \rmR^{\rmJ\,h}_\kappa
=
  \rmL^{\rmJ\,h-1} \circ \rmL^\rmJ \circ \rmR^{\rmJ\,h}_\kappa
=
  \rmL^{\rmJ\,h-1} \circ \rmR^{\rmJ\,h-1}_\kappa \circ
  \Big( \Delta^\rmJ_\kappa + \sum_{i=0}^{h-1} \alpha_{\kappa + 2i} \Big)
=
  \prod_{j=1}^h
  \Big( \Delta^\rmJ_\kappa + \sum_{i=0}^{j-1} \alpha_{\kappa + 2i} \Big)
\tx{.}
\end{gather*}
Letting $\kappa = 3-k$ and $h = k-2$, and observing that $\Delta^\rmJ_{3-k}$ annihilates $\bbJ_{3-k,m}$, we see that
\begin{gather}
\label{eq:prop:harmonic-and-almost-skew-harmonic-jacobi-forms:operator-scalar-action}
  \Big( \rmL^{\rmJ\,k-2} \circ \rmR^{\rmJ\,k-2}_{3-k} \Big)\;
  \bbJ_{3-k,m}
=
  \Big( \prod_{j=1}^{k-2} \sum_{i=0}^{j-1} \alpha_{3-k + 2i} \Big)\;
  \bbJ_{3-k,m}
\tx{.}
\end{gather}
For general $j \in \ZZ_{> 0}$, we have
\begin{gather*}
  \sum_{i=0}^{j-1} \alpha_{3-k + 2i}
=
  \frac{1}{4}
  \sum_{i=0}^{j-1} \big(
  (3-k + 2i) - \frac{3}{2}
  \big)
=
  \frac{1}{4} \big(
  2 \frac{(j-1) j}{2}
  -
  j \frac{2k-3}{2}
  \big)
=
  \frac{j}{4} \big(
  (j-1)
  -
  \frac{2k-3}{2}
  \big)
\tx{,} 
\end{gather*}
which is nonzero for any~$j \in \ZZ$. We conclude that~\eqref{eq:prop:harmonic-and-almost-skew-harmonic-jacobi-forms:operator} acts by a nonzero scalar on~$\bbJ_{3-k,m}$.
\end{proof}
\begin{remark}
For later purposes (see~\eqref{eq:almost-skew-harmonic-eisenstein-series-raising-and-xi}), we record that the last computation of the previous proof shows that
\begin{gather}
\label{eq:prop:harmonic-and-almost-skew-harmonic-jacobi-forms:operator-scalar-action-simplified}
  \Big( \rmL^{\rmJ\,k-2} \circ \rmR^{\rmJ\,k-2}_{3-k} \Big)\;
  \bbJ_{3-k,m}
=
  2^{4-2k} \Gamma(k-1) \Gamma(-\tfrac{1}{2}) \big\slash \Gamma(\tfrac{3}{2}-k)\,
  \bbJ_{3-k,m}
\tx{.}
\end{gather}
\end{remark}

\paragraph{Fourier expansions of almost skew-harmonic Maa\ss-Jacobi forms}

Every $\phi \in \cJ^\sk_{2-k,m}$ has a Fourier series expansion of the form
\begin{multline}
  \sum_{\substack{n, r \in \ZZ \\ D = 0}}
  c\big( \phi;\, n_{(0)},r \big)\,
  a^{\cJ\sk}_{2-k,m}(n_{(0)},r;\,y) e( n x + r z )
\\
  +\;
  \sum_{\substack{n, r \in \ZZ \\ D= 0}}
  c\big( \phi;\, n_{(1)},r \big)\,
  a^{\cJ\sk}_{2-k,m}(n_{(1)},r;\,y) e( n x + r z )
\\
  +\;
  \sum_{\substack{n, r \in \ZZ \\ D \ne 0}}
  c\big( \phi;\, n,r \big)
  a^{\cJ\sk}_{2-k,m}(n,r;\, y) e(n x + r z)
\tx{,}
\end{multline}
where as before $D := 4mn-r^2$ and where
\begin{alignat}{2}
\label{eq:def:fourier-coefficients-almost-skew-Jacobi-zero0}
  a^{\cJ\sk}_{2-k,m}(n_{(0)},r;\, y)
&:=
  y^{k-1}\,
  \exp\big( - \pi \frac{r^2}{2 m} y \big)
\quad
&&
  \tx{for $D = 0$,}
\\
\label{eq:def:fourier-coefficients-almost-skew-Jacobi-zero1}
  a^{\cJ\sk}_{2-k,m}(n_{(1)},r;\, y)
&:=
  y^{\frac{1}{2}}\,
  \exp\big( - \pi \frac{r^2}{2 m} y \big)
\quad
&&
  \tx{for $D = 0$,}
\\
\label{eq:def:fourier-coefficients-almost-skew-Jacobi-nonzero}
  a^{\cJ\sk}_{2-k,m}(n,r;\, y)\,
&:=
  \frac{%
  \rmR^\bbJ_{3-k} \big(
  a^\bbJ_{3-k,m}(n,r;\, y)\,
  e(n x + r z)
  \big)}{e(n x + r z)}
\quad
&&
  \tx{for $D \ne 0$}
\tx{.}
\end{alignat}
For $D = 0$, we have defined $a^{\cJ\sk}_{2-k,m}(n_{(0)},r;\, y)$ and $a^{\cJ\sk}_{2-k,m}(n_{(1)},r;\, y)$ in such a way that they match the corresponding Fourier series coefficients of skew-harmonic Maa\ss-Siegel modular forms (cf.\@ Lemma~\ref{la:kohnen-limit-negative-m-semi-definite-fourier-coefficients}). This implies that the subscripts of Fourier series coefficients change under the raising operator. For example, $\rmR^\bbJ_{3-k}\big( a^\bbJ_{3-k,m}(n_{(1)}, r;\, y) e(n x + r z) \big)$ is a scalar multiple of $a^{\cJ\sk}_{3-k,m}(n_{(0)}, r;\, y) e(n x + r z)$.

\paragraph{Connection to skew-holomorphic Jacobi forms}

Recall that the $\xi$-operator in~\cite{bringmann-richter-2010,bringmann-raum-richter-2015} maps harmonic Maa\ss-Jacobi forms to skew-holomorphic Jacobi forms. Since elements of $\bbJ_{3-k,m}$ are holomorphic in~$z$, Equation~(10) of~\cite{bringmann-raum-richter-2015} shows that there exists a map
\begin{gather*}
  y^{-\frac{1}{2} - k}\, \rmL^\rmJ_{2-k} 
  \,:\,
  \bbJ_{3-k,m}
\lra
  \rmJ^\sk_{k,m}
\tx{.}
\end{gather*}
Proposition~\ref{prop:harmonic-and-almost-skew-harmonic-jacobi-forms} and the fact that $\Delta^\rmJ_{3-k}\, \bbJ_{3-k,m} = 0$ allow us to find the corresponding map from almost skew-harmonic Maa\ss-Jacobi forms to skew-holomorphic Jacobi forms:
\begin{gather}
\label{eq:def:jacobi-raising-operator-power}
  \rmR^\cJsk_{2-k}
\;:=\;
  \frac{(-1)^{\frac{1-k}{2}} 2^{k-1} \pi^{\frac{1}{2}}}{\Gamma(k-\frac{1}{2})}\,
  \big( \rmR^\Jsk_{2-k} \big)^{k-1}
 \,:\,
 \cJ^\sk_{2-k,m}
\lra
  \rmJ^\sk_{k,m}
\tx{.}
\end{gather}
Analogous to the normalization of~$\rmR^\bbJ_{3-k}$, we have normalized $\rmR^\cJsk_{2-k}$ in such a way that it maps standard Eisenstein series to standard Eisenstein series.

\paragraph{Eisenstein series of index~$1$}

The following Jacobi-Eisenstein series converge absolutely and locally uniformly on~$\HJ$ for $m > 0$:
\begin{align}
\label{eq:def:skew-harmonic-jacobi-eisenstein-series}
  E^{\bbJ}_{3-k, m}
\;&:=\;
  \sum_{\gamma \in \GaJ_\infty \backslash \GaJ}
  y^{k-\frac{1}{2}} \big|^\rmJ_{3-k;\, m}\, \gamma
\,\in
  \bbJ_{3-k, m}
\\
\label{eq:def:almost-skew-harmonic-jacobi-eisenstein-series}
  E^{\cJ\sk}_{2-k, m}
\;&:=\;
  \sum_{\gamma \in \GaJ_\infty \backslash \GaJ}
  y^{k-1} \big|^\Jsk_{2 - k;\, m}\, \gamma
\,\in
  \cJ^\sk_{2-k, m}
\tx{.}
\end{align}
The defining expressions directly imply that
\begin{gather}
\label{eq:almost-skew-harmonic-eisenstein-series-raising-image}
  \rmR^\bbJ_{3-k}\, E^{\bbJ}_{3-k,m}
=
  E^{\cJ\sk}_{2-k,m}
\quad\tx{and}\quad
  \rmR^\cJsk_{2-k}\, E^{\cJ\sk}_{2-k,m}
=
  E^{\Jsk}_{k,m}
\tx{.}
\end{gather}
Using~\eqref{eq:def:harmonic-maass-jacobi-raising-operator}, \eqref{eq:def:jacobi-raising-operator-power}, and~\eqref{eq:prop:harmonic-and-almost-skew-harmonic-jacobi-forms:operator-scalar-action-simplified}, we find that for a function~$\phi: \HJ \ra \CC$ that satisfies~\ref{it:def:harmonic-jacobi-forms:analytic-type} of Definition~\ref{def:harmonic-jacobi-forms}, we have
\begin{gather}
\label{eq:almost-skew-harmonic-eisenstein-series-raising-and-xi}
  \rmR^\cJsk_{2-k}\, \rmR^\bbJ_{3-k}\, \phi
=
  \frac{8 i \pi^{\frac{1}{2}}\, \Gamma(k-1)}{\Gamma(k-\frac{1}{2})}\,
  \rmR^\Jsk_{k-2}\, y^{\frac{5}{2}-k}\, \phi
=
  \frac{-4 \pi^{\frac{1}{2}}\, \Gamma(k-1)}{\Gamma(k-\frac{1}{2})}\,
  \xi^\rmJ_{3-k}\, \phi
\tx{,}
\end{gather}
where $\xi^\rmJ_{3-k}$ is the Jacobi $\xi$-operator defined in~\eqref{eq:def:jacobi-xi-operator}. In particular, we find that
\begin{gather}
\label{eq:almost-skew-harmonic-eisenstein-series-xi-image}
  \xi^\rmJ_{3-k}\,
  E^\bbJ_{3-k,m}
=
  \frac{-\pi^{-\frac{1}{2}}\,\Gamma(k-\frac{1}{2})}{4\Gamma(k-1)}\,
  E^\Jsk_{k,m}
\tx{.}
\end{gather}

\begin{lemma}
\label{la:jacobi-eisenstein-series-fourier-coefficients-real}
Consider the Fourier series expansion
\begin{gather*}
  E^{\Jsk}_{k,1}
=
  \sum_{n, r \in \ZZ}
  c(n,r)\, a^\Jsk_{1}(n,r;\,y) e(n x + r z)
\tx{.}
\end{gather*}
We have $c(n,r) \in \RR$. Moreover, there exists an indefinite matrix~$T = \begin{psmatrix} n & r \slash 2 \\ r \slash 2 & 1 \end{psmatrix}$ with $c(n,r) \ne 0$.
\end{lemma}
\begin{proof}
The statement is clear if $D := 4 m n - r^2 = 0$. Consider the case~$D \ne 0$. Proposition~2 of~\cite{mizuno-2008} provides the Fourier series expansion of a Jacobi-Eisenstein series $E^{\mathrm{Mizuno}}_{k,1}(\tau, z, s)$, which features the function $\tau_{d}(y, \alpha, \beta):=\int_{-\infty}^{\infty}e(-du)\tau^{-\alpha}\overline{\tau}^{-\beta}\,du$. Replacing $k \leadsto 1-k$ and $s \leadsto k - \frac{1}{2}$ implies
\begin{gather*}
  E^{\Jsk}_{k,1}
=
  y^{\frac{1}{2}-k}\,
  E^{\mathrm{Mizuno}}_{1-k,1}(\tau, z, k - \tfrac{1}{2})
\tx{,}
\end{gather*}
and~(7) of~\cite{mizuno-2008} gives that $c(n,r) \in i^{-\frac{1}{2}} \tau_{4l-\mu}(\frac{y}{4}, 0, k -\frac{1}{2}) \RR$, since~$D \ne 0$. Proposition~4 of~\cite{mizuno-2005} then shows that $\tau_{4l-\mu}(\frac{y}{4}, 0, k -\frac{1}{2}) \in i^{k-\frac{1}{2}} \RR$, which yields the first claim of the lemma.

If the second claim of the lemma did not hold, then~$E^{\Jsk}_{k,1}$ would be a linear combination of holomorphic theta series by its theta decomposition, which is impossible.
\end{proof}

\begin{lemma}
\label{la:jacobi-eisenstein-series-fourier-coefficients-nonzero}
Consider the Fourier series expansion
\begin{gather*}
  E^{\cJ\sk}_{2-k,1}
=
  \sum_{n, r \in \ZZ}
  c(n,r;\, y) e(n x + r z)
\tx{.}
\end{gather*}
There exist matrices $T = \begin{psmatrix} n & r \slash 2 \\ r \slash 2 & 1 \end{psmatrix}$ that are
\begin{enumerateroman*}
\item indefinite and
\item negative definite
\end{enumerateroman*}
such that $c(n,r;\, y) \ne 0$.
\end{lemma}
\begin{proof}
Observe that $E^{\cJ\sk}_{2-k,1}(\tau, z)$ is holomorphic with respect to~$z$. Hence it has the theta decomposition
\begin{gather*}
  E^{\cJ\sk}_{2-k,1}(\tau, z)
=
  h_0(-\ov{\tau})\theta_{1,0}(\tau,z)
  +
  h_1(-\ov{\tau})\theta_{1,1}(\tau,z)
\tx{,}
\end{gather*} 
where
\begin{gather*}
\theta_{1,\mu}(\tau,z)
:=
\sum_{\substack{r \in \ZZ \\ r \equiv \mu \pmod{2}}}
q^{\frac{r^2}{4}} \zeta^r
\quad
\tx{for $\mu \in \{0,1\}$,}
\end{gather*}
and where $h_0$ and $h_1$ are certain modular forms of weights $\frac{3}{2}-k$. In analogy to Theorem~5.4 of~\cite{eichler-zagier-1985}, the function
\begin{gather*}
  h_0(4\tau) + h_1(4\tau)
=
  \sum_{n \in \ZZ} c'(n,y) e(n x)
\tx{}
\end{gather*}
is essentially the Eisenstein series $E_\infty$ of~\cite{goldfeld-hoffstein-1985} (with $k \leadsto 3 - 2k$, $s \leadsto k - 1$). Proposition~1.4 of~\cite{goldfeld-hoffstein-1985} in combination with 8.447.3 of~\cite{gradshteyn-ryzhik-2007} shows that the coefficients $c'(-4,y)$ and $c'(4,y)$ are nonzero, which yields the claim. Alternatively, instead of~\cite{goldfeld-hoffstein-1985}, one could also apply Proposition~2 of~\cite{mizuno-2008}.
\end{proof}

\begin{lemma}
\label{la:jacobi-eisenstein-series-zeroth-fourier-coefficient}
The Fourier coefficient of index~$n=0$ and $r=0$ of\/ $E^{\cJ\sk}_{2-k,1}(\tau, z)$ equals
\begin{gather}
\label{eq:la:jacobi-eisenstein-series-zeroth-fourier-coefficient}
  a\big(E^{\cJ\sk}_{2-k,1};\, 0,0;\, y \big)
=
  y^{k-1}
  \;+\;
  (-1)^{\frac{k-1}{2}}2^{2-k}
  \frac{\Gamma(\frac{1}{2}) \Gamma(k-\frac{3}{2})}{\Gamma(k-1)}\,
  \frac{\zeta(2k-3)}{\zeta(2k-2)}\,
  y^{\frac{1}{2}}
\tx{.}
\end{gather}
\end{lemma}
\begin{proof}
Arakawa~\cite{arakawa-1990} investigates real-analytic Jacobi Eisenstein series $E_{k,m,r}((\tau,z),s)$. In the special case that $r=0$, $m=1$, $k \leadsto k-1$, and $s=\frac{k}{2}-\frac{1}{4}$, one finds that the Eisenstein series~$\frac{1}{2}y^{k-\frac{3}{2}} E_{k-1,1,0}((\tau,z),\frac{k}{2}-\frac{1}{4})$ equals $E^{\cJ\sk}_{2-k,1}(\tau, z)$. Arakawa's work also determines the part of the Fourier series expansion of real-analytic Jacobi Eisenstein series with $4n - r^2 = 0$, which yields
\begin{multline*}
  \sum_{\substack{n, r \in \ZZ \\ 4n=r^2}}
  a\big(E^{\cJ\sk}_{2-k,1};\, n,r;\, y \big) e(n x + r z)
=
  \Big(
  y^{k-1}
  \;+\;
  (-1)^{\frac{k-1}{2}} 2^{2-k}
  \frac{\Gamma(\frac{1}{2}) \Gamma(k-\frac{3}{2})}{\Gamma(k-1)}\,
  \frac{\zeta(2k-3)}{\zeta(2k-2)}\,
  y^{\frac{1}{2}}
  \Big)\,
  \theta_{1,0}(\tau,z)
\\ 
+
  \Big(
  (-1)^{\frac{k-1}{2}} 2^{2-k}
  \frac{\Gamma(\frac{1}{2}) \Gamma(k-\frac{3}{2})}{\Gamma(k-1)}\,
  \varphi_{0,1} \,y^{\frac{1}{2}}
  \Big)\,
  \theta_{1,1}(\tau,z)
\tx{,}
\end{multline*}
where $\varphi_{0,1}$ is a certain Dirichlet series, and $\theta_{1,0}$ and $\theta_{1,1}$ are as in the proof of Lemma~\ref{la:jacobi-eisenstein-series-fourier-coefficients-nonzero}. This implies~\eqref{eq:la:jacobi-eisenstein-series-zeroth-fourier-coefficient}.
\end{proof}

\paragraph{Eisenstein series of index~$m$}

In the proof of Theorem~\ref{mainthm:skew-maass-lift} we will need the following vanishing statement for Jacobi forms in~$\cJ^\sk_{2-k,m}$.
\begin{lemma}
\label{la:almost-skew-harmonic-jacobi-form-vanishing}
Let $\phi \in \cJ^\sk_{2-k,m}$. We have~$\phi = 0$, if for all $n,r \in \ZZ$ with $4 m n - r^2 = 0$, the $(n,r)$\thdash\ Fourier series coefficient of~$\phi$ is of the form
\begin{gather*}
  b(n,r)\,
  y^{\frac{1}{2}}
  e\big( \frac{r^2}{4m} i y \big)
\end{gather*}
for some constants $b(n,r) \in \CC$.
\end{lemma}
\begin{proof}
Combining the theta decomposition for real-analytic Jacobi forms that are holomorphic in~$z$ with the half-integral weight analogue of Theorem~31 of~\cite{maass-1964}, we find that~$\cJ^\sk_{2-k,m}$ is spanned by Eisenstein series
\begin{gather*}
  \sum_{\gamma \in \GaJ_\infty \backslash \GaJ}
  y^{k-\frac{1}{2}}
  e\big(\frac{r^2}{4m} \tau + r z \big)
  \big|^\Jsk_{2-k;\, m}\, \gamma
\end{gather*}
for $r \in \ZZ$ and $s \in \{ k - 1, \frac{1}{2} \}$. Invariance under the embedded Heisenberg group in~$\GaJ$ and under $-I^{(2)}$ allows us to assume that $0 \le r \le m$. We discover that $\phi \in \cJ^\sk_{2-k,m}$ can be written as
\begin{gather*}
  \phi(\tau, z)
=
  \sum_{\gamma \in \GaJ_\infty \backslash \GaJ}
  \Big(
  \sum_{r = 0}^m \td{b}(r)
  y^{k-\frac{1}{2}}
  e\big(\frac{r^2}{4m} \tau + r z \big) \Big)
  \Big|^\Jsk_{2-k;\, m}\, \gamma
\end{gather*}
for some constants $\td{b}(r) \in \CC$.

Separating the contributions of matrices $\begin{psmatrix} a & b \\ c & d \end{psmatrix} \in \SL{2}(\ZZ)$ with $c = 0$ and $c \ne 0$, we find that the $(n,r)$\thdash\ Fourier coefficient of~$\phi$ with $4 n m - r^2 = 0$ and $0 \le r \le m$ equals
\begin{gather*}
  \big( \td{b}(r) y^{k-\frac{1}{2}} + b(r) y^{\frac{1}{2}} \big)\,
  e\big( \frac{r^2}{4m} i y \big)
\end{gather*}
for some coefficients $b(r) \in \CC$ arising from the associated scattering matrix. The assumptions of the Lemma guarantee that~$\td{b}(r) = 0$ for all~$0 \le r \le m$, which proves that $\phi = 0$.
\end{proof}

\paragraph{Hecke operators}

The action of $V_l$ for $l > 0$ on $a^{\cJ\sk}_{2-k,m}(n,r;\, y) e(n x + r z)$ is analogous to the skew-holomorphic case. We obtain
\begin{multline*}
  \sum_{\substack{n, r \in \ZZ \\ D \ne 0}}
  c\big( n,r \big) a^{\cJ\sk}_{2-k,1}(n,r;\, y) e(n x + r z)
  \Big|^\Jsk_{k,1}\, V_m
\\
=
  \sum_{\substack{n, r \in \ZZ \\ D \ne 0}}
  \Big( \sum_{d \isdiv (n,r,m)} d^{1-k} c\big( \frac{mn}{d^2}, \frac{r}{d} \big) \Big)
  a^{\cJ\sk}_{2-k,1}(n,r;\, y) e(n x + r z)
\tx{,}
\end{multline*}
where $D := 4 n - r^2$. However, the action on coefficients with $D = 0$ is different. In the context of the Kohnen limit process for Eisenstein series, we will need the action of $V_l$ for $l > 0$ on powers of~$y$, described in the next lemma.
\begin{lemma}
\label{eq:V-hecke-operator-real-analytic-coefficient-formula}
Let $l > 0$ and $s \in \CC$. Then
\begin{gather*}
  y^s \big|^\Jsk_{k,m}\, V_l
=
  l^{-s}
  \sigma_{k + 2s - 1}(l)\,
  y^s
\tx{.}
\end{gather*}
\end{lemma}
\begin{proof}
Inserting the definition of~$V_l$, we have
\begin{gather*}
  y^s \big|^\Jsk_{k,m}\, V_l
=
  l^{k-1}
  \sum_{\substack{a \isdiv l \\ b \pmod{l \slash a}}}
  (l \slash a)^{-k}
  \big( \frac{a y}{l \slash a} \big)^s
=
  l^{k-1} l
  \sum_{\substack{a \isdiv l}}
  l^{-k} a^{k-1}
  a^{2s} l^{-s} y^s
= 
  l^{-s}
  \sigma_{k + 2s - 1}(l)\,
  y^s
\tx{.}
\end{gather*}
\end{proof}

\subsection{Abstract Fourier-Jacobi terms}
\label{ssec:abstract-fourier-jacobi-terms}

Abstract Fourier-Jacobi terms are intermediate objects in between Jacobi forms and Siegel modular forms. 
Consider the Fourier-Jacobi expansion of a real-analytic Siegel modular form
\begin{gather*}
  F(Z)
=
  \sum_{m \in \ZZ}
  \wtd\phi_m(\tau, z, \tau')
\tx{,}
\end{gather*}
where $\wtd\phi_m(\tau, z, \tau') = \wht\phi_m(\tau, z, y') e(m x')$ for some function $\wht\phi_m$ (see also~\eqref{eq:introduction:fourier-jacobi}), and
\begin{gather*}
  \wtd\phi_m(\tau, z, \tau' + b)
=
  \wtd\phi_m(\tau, z, \tau') e(mb)
\quad
  \tx{for all $b \in \RR$.}
\end{gather*}

The $m$\thdash\ Fourier-Jacobi term $\wtd\phi_m$ inherits analytic properties of $F$.

\begin{definition}[Skew-harmonic abstract Fourier-Jacobi terms]
\label{def:abstract-fourier-jacobi-term-skew-harmonic}
Let $m \in \ZZ$. Suppose $\wtd\phi_m :\, \HS^{(2)} \ra \CC$ is a function such that 
\begin{gather*}
  \wtd\phi_m(\tau, z, \tau')
=
  \wht\phi_m(\tau, z, y') e(m x')
\quad
  \tx{for some function $\wht\phi_m$.}
\end{gather*}
Then $\wtd\phi_m$ is an abstract skew-harmonic Fourier-Jacobi term of weight~$k$ and Jacobi index~$m$ if it satisfies
\begin{enumerateroman}
\item
$\Omega^\sk_k\, \wtd\phi_m = 0$ and $\xi^\sk_k\, \wtd\phi_m = 0$.

\item
For all $\gaJ$ in the extended Jacobi group $\SL{2}(\ZZ) \ltimes (\ZZ^2 \ltimes \ZZ)$ (viewed as embedded into $\Sp{2}(\ZZ)$), we have the Siegel slash invariance $\wtd\phi_m \big|^\sk_{k} \,\gaJ = \wtd\phi_m$.

\item
$\wtd\phi_m(Z) = \cO(\tr(Y)^a)$ for some $a \in \RR$ as $\tr(Y) \ra \infty$.
\end{enumerateroman}

The space of such functions is denoted by $\rmA\rmJ^\sk_{k,m}$.
\end{definition}

We record the following consequence of Propositions~\ref{prop:multiplicity-one-for-skew-harmonic-fourier-coefficients-nonzero-nonnegative} and~\ref{prop:multiplicity-one-for-skew-harmonic-fourier-coefficients-negative}.
\begin{proposition}
Suppose that $m > 0$ and $\wtd\phi_m \in \rmA\rmJ^\sk_{k,m}$. Then $\wtd\phi_m$ has a Fourier series expansion
\begin{gather*}
  \wtd\phi_m(\tau, z, \tau')
=
  \sum_{n,r \in \ZZ}
  c(n, r)
  a^\sk_k\big( \begin{psmatrix} n & r \slash 2 \\ r \slash 2 & m \end{psmatrix} Y \big)
  e\big( \begin{psmatrix} n & r \slash 2 \\ r \slash 2 & m \end{psmatrix} X \big)
\tx{,}\quad
  c(n,r) \in \CC
\tx{.}
\end{gather*}
If $m < 0$ and $\wtd\phi_m \in \rmA\rmJ^\sk_{k,m}$, then $\wtd\phi_m$ has a Fourier series expansion of the form
\begin{multline*}
  \wtd\phi_m(\tau, z, \tau')
=
  \sum_{\substack{n,r \in \ZZ \\ 4 nm - r^2 = 0}}
  c(n_{(0)}, r)
  a^\sk_k\big( \begin{psmatrix} n & r \slash 2 \\ r \slash 2 & m \end{psmatrix}_{(0)} Z \big)
  \,+
  \sum_{\substack{n,r \in \ZZ \\ 4 nm - r^2 = 0}}
  c(n_{(1)}, r)
  a^\sk_k\big( \begin{psmatrix} n & r \slash 2 \\ r \slash 2 & m \end{psmatrix}_{(1)} Z \big)
\\
  +
  \sum_{\substack{n,r \in \ZZ \\ 4 nm - r^2 \ne 0}}
  c(n, r)
  a^\sk_k\big( \begin{psmatrix} n & r \slash 2 \\ r \slash 2 & m \end{psmatrix} Z \big)
\end{multline*}
for $c(n_{(0)},r), c(n_{(1)},r), c(n,r) \in \CC$. The Fourier series expansions converge absolutely and locally uniformly.
\end{proposition}

\paragraph{Hecke operators}

We generalize the definitions of the Hecke operators in~\eqref{eq:def:U-hecke-operator} and~\eqref{eq:def:V-hecke-operator-positive-l} to the case of abstract Fourier-Jacobi terms.  Specifically, for $\wtd\phi \in \rmA\rmJ^\sk_{k,m}$ and $l > 0$, we set
\begin{gather}
\label{eq:def:U-hecke-operator-abstract-fourier-jacobi-coefficient}
  \big( \wtd\phi \big|^\sk_k\, U_l \big)(\tau, z, \tau')
:=
  \wtd\phi(\tau, lz, l^2 \tau')
\end{gather}
and
\begin{gather}
\label{eq:def:V-hecke-operator-positive-l-abstract-fourier-jacobi-coefficient}
  \big( \wtd\phi \big|^\sk_k\, V_l \big)(\tau, z, \tau')
:=
  l^{k-1} \sum_{\substack{a \isdiv l\\b \pmod{l \slash a}}}
  (l \slash a)^{-k}\,
  \phi\big( \frac{a \tau + b}{l \slash a},\, a z, a^2 \tau' \big)
\tx{.}
\end{gather}
Both can also be defined via $\up(\GaJ)$-cosets in $\GSp{2}(\QQ)$, yielding the maps
\begin{gather}
\label{eq:U-V-hecke-operator-domain-range}
  U_l
:\,
  \rmA\rmJ^\sk_{k,m}
\lra
  \rmA\rmJ^\sk_{k,l^2m}
\quad\tx{and}\quad
  V_l
:\,
  \rmA\rmJ^\sk_{k,m}
\lra
  \rmA\rmJ^\sk_{k,lm}
\tx{.}
\end{gather}
We also extend the coefficient formula~\eqref{eq:V-hecke-operator-coefficient-formula} to the case of abstract Fourier-Jacobi terms: If $\wtd\phi_1\in\rmA\rmJ^\sk_{k,1}$, then 
\begin{gather}
\label{eq:V-hecke-operator-coefficient-formula-abstract-fourier-jacobi-coefficient}
  c\big( \wtd\phi_1 \big|^\sk_k\, V_m;\,  n,r \big)
\;=\;
  \sum_{0 < d \isdiv \gcd(n,r,m)} \hspace{-1em}
  d^{k-1} c\Big(\wtd\phi_1;\, \frac{mn}{d^2}, \frac{r}{d} \Big)
\tx{,}\quad
  m > 0
\tx{.}
\end{gather}

\subsection{Abstract Fourier-Jacobi-Eisenstein series}

We introduce abstract Fourier-Jacobi-Eisen\-stein series, and determine their behavior under the action of the Hecke operator $V_m$.

For $k \in \ZZ$, $m \in \ZZ$, and a smooth function $f : \RR_{> 0} \ra \CC$, we set
\begin{gather}
\label{eq:def:abstract-jacobi-eisenstein-series}
  AE^\Jsk_{k,m}(f;\, Z)
:=
  \sum_{\gaJ \in \GaJ_\infty \backslash \GaJ}
  f(m y') e(m x') \big|^\sk_k\, \gaJ
\end{gather}
provided that the sum converges absolutely, and where $\gaJ$ is viewed again as an element of~$\Sp{2}(\ZZ)$.

Our following Proposition is analogous to Theorem~4.3 of~\cite{eichler-zagier-1985}, and will be needed for the proofs of Lemma~\ref{la:kohnen-limit-positive-m-eisenstein-series} and Lemma~\ref{la:kohnen-limit-negative-m-eisenstein-series}.
\begin{proposition}
\label{prop:abstract-jacobi-eisenstein-series-Vm-relation}
Let $k \in \ZZ$, $m \in \ZZ_{> 0}$, and $f : \RR_{> 0} \ra \CC$ be smooth such that~\eqref{eq:def:abstract-jacobi-eisenstein-series} converges absolutely. Then
\begin{gather}
\label{eq:prop:abstract-jacobi-eisenstein-series-Vm-relation}
  AE^\Jsk_{k,\pm 1}(f) \big|^\sk_k\, V_m
=
  \sum_{\substack{l, l' > 0 \\ l^2 \isdiv m \\ l' \isdiv l}}
  \sigma_{k-1}(m \slash l^2) \mu(l')\,
  AE^\Jsk_{k,\pm m l^{\prime\, 2} \slash l^2}(f)
  \big|^\sk_k\, U_{l \slash l'}
\tx{.}
\end{gather}
\end{proposition}

\begin{proof}[{Proof of Proposition~\ref{prop:abstract-jacobi-eisenstein-series-Vm-relation}}]
We only treat the case of Jacobi index~$+1$; The case of index~$-1$ is analogous. We closely follow the proof of Theorem~4.3 of~\cite{eichler-zagier-1985}. For convenience, set $\td{f}(\tau') := f(y') e(x')$. We start by expanding the left-hand side of~\eqref{eq:prop:abstract-jacobi-eisenstein-series-Vm-relation} completely, and obtain:
\begin{gather*}
  m^{k-1}
  \sum_{\substack{\gamma \in \Gamma_\infty \backslash \Mat{2}(m) \\ \gamma = \begin{psmatrix} a & b \\ c & d \end{psmatrix}}}
  (c \tau + d)^{-\frac{1}{2}}
  (c \ov\tau + d)^{\frac{1}{2}-k}
  \sum_{\lambda \in \ZZ}
  \td{f}\Big(m \tau'
  +
  \lambda^2 \frac{a \tau + b}{c \tau + d}
  +
  2 \lambda \frac{m z}{c \tau + d}
  -
  \frac{m c z^2}{c \tau + d}
  \Big)
\tx{,}
\end{gather*}
where $\Mat{2}(m) := \{ M \in \Mat{2}(\ZZ) \,:\, \det(M) = m \}$.

In analogy with the proof of Theorem~4.3 of~\cite{eichler-zagier-1985}, representatives of $\Gamma_\infty \backslash \Mat{2}(m)$ can be separated into distinct classes via the greatest common divisor~$\delta$ of their bottom row. Write $\Mat{2}(m)_\times \subseteq \Mat{2}(m)$ for the subset of matrices with coprime bottom row. As sets we have
\begin{gather*}
  \Gamma_\infty \big\backslash \Mat{2}(m)
\;\sim\;
  \bigcup_{\delta \isdiv m}
  \Gamma_\infty
  \big\backslash
  \Big\{
  \begin{psmatrix}
  a & b \\
  \delta c & \delta d
  \end{psmatrix}
  \,:\,
  \begin{psmatrix} a & b \\ c & d \end{psmatrix} \in \Mat{2}(m \slash \delta)_\times
  \Big\}
\;\sim\;
  \bigcup_{\delta \isdiv m}
  \Gamma_\infty(\delta)
  \big\backslash
  \Mat{2}(m \slash \delta)_\times
\tx{,}
\end{gather*}
where $\Gamma_\infty(\delta):=\{\begin{psmatrix} 1 & \delta b \\ 0 & 1 \end{psmatrix} \,:\, b \in \ZZ \} \subseteq \Gamma_\infty$.

We insert this into our expression for the left-hand side of~\eqref{eq:prop:abstract-jacobi-eisenstein-series-Vm-relation}:
\begin{gather*}
  \sum_{\delta \isdiv m}
  \frac{m^{k-1}}{\delta^k}
  \sum_{\substack{\gamma \in \Gamma_\infty(\delta) \backslash \Mat{2}(m \slash \delta)_\times \\ \gamma = \begin{psmatrix} a & b \\ c & d \end{psmatrix}}} \hspace{-1.7em}
  (c \tau + d)^{-\frac{1}{2}}
  (c \ov\tau + d)^{\frac{1}{2}-k}
  \sum_{\lambda \in \ZZ}
  \td{f}\Big(m \tau'
  +
  \frac{\lambda^2}{\delta} \frac{a \tau + b}{c \tau + d}
  +
  \frac{2 \lambda}{\delta} \frac{m z}{c \tau + d}
  -
  \frac{m c z^2}{c \tau + d}
  \Big)
	\text{.}
\end{gather*}
Note that if $\gamma \in \Gamma_\infty \backslash \Mat{2}(m \slash \delta)_\times$, then $\gamma' \gamma$ belong to distinct classes of $\Gamma_\infty(\delta) \backslash \Mat{2}(m \slash \delta)_\times$ when $\gamma'$ runs through $\Gamma_\infty(\delta) \backslash \Gamma_\infty$. In particular, we can split the sum over $\Gamma_\infty(\delta) \backslash \Mat{2}(m \slash \delta)_\times$ into sums over $\Gamma_\infty(\delta) \backslash \Gamma_\infty$ and $\Gamma_\infty \backslash \Mat{2}(m \slash \delta)_\times$. When executing the sum over $\Gamma_\infty(\delta) \backslash \Gamma_\infty$, we find from orthogonality of roots of unity that $\lambda$ contributes only if $\delta \isdiv \lambda^2$:
\begin{gather*}
  \sum_{\delta \isdiv m}
  \frac{m^{k-1}}{\delta^{k-1}}
  \sum_{\substack{\gamma \in \Gamma_\infty \backslash \Mat{2}(m \slash \delta)_\times \\ \gamma = \begin{psmatrix} a & b \\ c & d \end{psmatrix}}} \hspace{-1.7em}
  (c \tau + d)^{-\frac{1}{2}}
  (c \ov\tau + d)^{\frac{1}{2}-k}
  \sum_{\substack{\lambda \in \ZZ \\ \delta \isdiv \lambda^2}}
  \td{f}\Big(m \tau'
  +
  \frac{\lambda^2}{\delta} \frac{a \tau + b}{c \tau + d}
  +
  \frac{2 \lambda}{\delta} \frac{m z}{c \tau + d}
  -
  \frac{m c z^2}{c \tau + d}
  \Big)
	\text{.}
\end{gather*}

We write $\delta = \delta_\rms^2 \delta_\rmf$ for square-free~$\delta_\rmf$, and the condition~$\delta \isdiv \lambda^2$ gives~$\delta_\rms \delta_\rmf \isdiv \lambda$. Also observe that
\begin{gather*}
  \Gamma_\infty \big\backslash \Mat{2}(m \slash \delta)_\times
\;\sim\;
  \Gamma_\infty(m \slash \delta)
  \big\backslash
  \Big\{
  \begin{psmatrix}
  m a \slash \delta & m b \slash \delta \\
  c & d
  \end{psmatrix}
  \,:\,
  \begin{psmatrix} a & b \\ c & d \end{psmatrix} \in \SL{2}(\ZZ)
  \Big\}
\;\sim\;
  \Gamma_\infty \big\backslash \SL{2}(\ZZ)
\tx{.}
\end{gather*}
Hence the left-hand side~$AE^\Jsk_{k,\pm 1}(f) \big|^\sk_k\, V_m$ of~\eqref{eq:prop:abstract-jacobi-eisenstein-series-Vm-relation} equals
\begin{multline*}
  \sum_{\delta \isdiv m}
  \frac{m^{k-1}}{\delta^{k-1}} 
  \sum_{\substack{\gamma \in \Gamma_\infty \backslash \SL{2}(\ZZ) \\ \gamma = \begin{psmatrix} a & b \\ c & d \end{psmatrix}}} \hspace{-0.7em}
  (c \tau + d)^{-\frac{1}{2}}
  (c \ov\tau + d)^{\frac{1}{2}-k}
\\
  \sum_{\lambda \in \delta_\rms \delta_\rmf \ZZ}
  \td{f}\Big(
  \frac{m}{\delta_\rms^2} \delta_\rms^2 \tau'
  +
  \frac{\lambda^2}{\delta_\rms^2 \delta_\rmf^2} \frac{m}{\delta_\rms^2} \frac{a \tau + b}{c \tau + d}
  +
  2 \frac{\lambda}{\delta_\rms \delta_\rmf} \frac{m}{\delta_\rms^2} \frac{\delta_\rms z}{c \tau + d}
  -
  \frac{m}{\delta_\rms^2} \frac{c (\delta_\rms z)^2}{c \tau + d}
  \Big)
\tx{,}
\end{multline*}
which simplifies to
\begin{gather*}
  \sum_{\substack{\delta \isdiv m \\ \delta = \delta_\rms^2 \delta_\rmf}}
  \frac{m^{k-1}}{\delta^{k-1}} 
  AE^\Jsk_{k,m \slash \delta_\rms^2}(f) \big|^\sk_k\, U_{\delta_\rms}
\tx{.}
\end{gather*}

Recall from~\eqref{eq:U-V-hecke-operator-domain-range} that $U_l$ increase the Jacobi index by a factor~$l^2$. It therefore suffices to show that the following linear combinations of Jacobi Hecke operators agree:
\begin{gather}
\label{eq:prop:abstract-jacobi-eisenstein-series-Ul-relation}
  \sum_{\substack{\delta \isdiv m \\ d = \delta_\rms^2 \delta_\rmf}}
  \frac{m^{k-1}}{\delta^{k-1}}\,
  U_{\delta_\rms}
\qquad\tx{and}\qquad
\sum_{\substack{l, l' > 0 \\ l^2 \isdiv m \\ l' \isdiv l}}
  \sigma_{k-1}(m \slash l^2) \mu(l')\,
  U_{l \slash l'}
\tx{.}
\end{gather}
Since both expressions in~\eqref{eq:prop:abstract-jacobi-eisenstein-series-Ul-relation} are multiplicative in~$m$, we may assume for the remainder that~$m$ is a power of some prime~$p$.

Consider the second expression in~\eqref{eq:prop:abstract-jacobi-eisenstein-series-Ul-relation}. Write $m = m_\rms^2 m_\rmf$, where $m_\rms$ is a power of~$p$ and with~$m_\rmf \in \{1, p\}$.  The contribution of $l = m_\rms$ and $l' = 1$ is $\sigma_{k-1}(m_\rmf)\, U_{m_\rms}$, and if $l \slash l' \ne m_\rms$, then there are two choices of~$l$ and~$l'$, whose contributions partially cancel. In total, we find for the second expression in~\eqref{eq:prop:abstract-jacobi-eisenstein-series-Ul-relation}:
\begin{gather*}
  \sum_{\substack{l p \isdiv m_\rms}}
  \big( (m \slash l^2)^{k-1} + (m \slash l^2 p)^{k-1} \big)\, U_l
  \;+\;
  \sigma_{k-1}(m_\rmf)\, U_{m_\rms}
\tx{.}
\end{gather*}
Since $m_\rmf \in \{1, p\}$ and hence $\sigma_{k-1}(m_\rmf) = 1$ or $\sigma_{k-1}(m_\rmf) = 1 + p^{k-1}$, the previous expression agrees precisely with the first one in~\eqref{eq:prop:abstract-jacobi-eisenstein-series-Ul-relation}. This concludes the proof.
\end{proof}

\section{The Kohnen limit process for \texpdf{$m \ne 0$}{m not equal 0}}
\label{sec:kohnen-limit-nonzero-m}

Kohnen~\cite{kohnen-1994} introduced a limit process for Siegel modular forms, which was further developed in~\cite{bringmann-raum-richter-2011}. In particular, if $k > 3$ and $m>0$, then \cite{bringmann-raum-richter-2011} established that  Kohnen's limit process applied to skew-harmonic Maa\ss-Siegel forms yields skew-holomorphic Jacobi forms.  In this section, we extend the Kohnen limit process to negative $m$, and we also show that it can be inverted. Our proofs of some of the results in this section are quite technical. For brevity, we occasionally omit a proof, if it follows the pattern of a previous proof.

In the following, we define the Kohnen limit processes for abstract Fourier-Jacobi coefficients. Observe that these limits already converge for Fourier series coefficients of abstract Fourier-Jacobi coefficients. Sometimes we apply Kohnen limit processes to such Fourier series coefficients.

\subsection{The Kohnen limit process for \texpdf{$m > 0$}{m > 0}}
\label{ssec:kohnen-limit-positive-m}

We begin with a minor extension of the Kohnen limit process in~\cite{bringmann-raum-richter-2011}.
\begin{proposition}
\label{prop:kohnen-limit-positive-m-convergence}
Let $\wtd{\phi}_m \in \rmA\rmJ^\sk_{k,m}$, $m > 0$. Then
\begin{gather}
\label{eq:def:kohnen-limit-positive-m}
  \Klim^\sk_{k,m} \big( \wtd{\phi}_m \big) (\tau, z)
:=
  \lim_{y' \ra \infty}
  (4 \pi m y')^{k-\frac{1}{2}} e(- m \tau')\,
  \widetilde{\phi}_m(\tau, z, \tau')
\end{gather}
exists locally uniformly, and we have
\begin{gather*}
  \Klim^\sk_{k,m} \big( \wtd{\phi}_m \big)
\in
  \rmJ^\sk_{k,m}
\tx{.}
\end{gather*}
\end{proposition}
\begin{proof}
If $\wtd{\phi}_m$ is a Fourier-Jacobi term of a skew-harmonic Maa\ss-Siegel form, then~\eqref{eq:def:kohnen-limit-positive-m} was established in~\cite{bringmann-raum-richter-2011}. The extension of~\eqref{eq:def:kohnen-limit-positive-m} to arbitrary $\wtd\phi \in \rmA\rmJ^\sk_{k,m}$ follows from the mul\-ti\-pli\-ci\-ty-one result in Proposition~\ref{prop:multiplicity-one-for-skew-harmonic-fourier-coefficients-nonzero-nonnegative}.
\end{proof}

Note that our notation in~\eqref{eq:def:kohnen-limit-positive-m} is slightly different from~\cite{bringmann-raum-richter-2011}.  Furthermore, we inserted the factor~$(4 \pi m)^{k-\frac{1}{2}}$, which allows for the covariance in the next proposition and which normalizes the Kohnen limit process applied to the skew Eisenstein series in Lemma~\ref{la:kohnen-limit-positive-m-eisenstein-series}.
\begin{proposition}
\label{prop:kohnen-limit-positive-m-covariance}
Let $\wtd\phi_m \in \rmA\rmJ^\sk_{k,m}$, $m > 0$. We have the covariance properties
\begin{gather}
\label{eq:kohnen-limit-positive-m-covariance}
\begin{aligned}
  \Klim^\sk_{k,m} \big( \wtd\phi_m \big|^\sk_{k}\, g^\rmJ \big)
&=
  \Klim^\sk_{k,m} \big( \wtd\phi_m \big) \big|^\Jsk_{k,m}\, g^\rmJ
\tx{,}
\\
  \Klim^\sk_{k,a^2 m} \big( \wtd\phi_m \big|^\sk_{k}\, \diag(1,a,1, 1 \slash a) \big)
&=
  \Klim^\sk_{k,m} \big( \wtd\phi_m \big) \big|^\Jsk_{k,m} \diag(1,a,1, 1 \slash a)
\tx{,}
\end{aligned}
\end{gather}
for all $g^\rmJ \in \GJ$ and for all $a \in \RR^\times$. For all $l > 0$, we have
\begin{gather}
\label{eq:kohnen-limit-positive-m-covariance-hecke-operators}
\begin{aligned}
  \Klim^\sk_{k,l^2 m} \big( \wtd\phi_m \big|^\sk_{k}\, U_l \big)
=
  \Klim^\sk_{k,m} \big( \wtd\phi_m \big) \big|^\Jsk_{k,m}\, U_l
\tx{,}\\
  \Klim^\sk_{k,l m} \big( \wtd\phi_m \big|^\sk_{k}\, V_l \big)
=
  \Klim^\sk_{k,m} \big( \wtd\phi_m \big) \big|^\Jsk_{k,m}\, V_l
\tx{.}
\end{aligned}
\end{gather}
\end{proposition}
\begin{proof}
The covariance of Hecke operators in~\eqref{eq:kohnen-limit-positive-m-covariance-hecke-operators} follows directly from~\eqref{eq:kohnen-limit-positive-m-covariance}.

The uniform convergence of the Kohnen limit process yields the first equation in~\eqref{eq:kohnen-limit-positive-m-covariance}. To establish the second one, observe that
\begin{align*}
{}&
  (\Klim^\sk_{k,m} \wtd\phi_m) \big|^\Jsk_{k,m} \diag(1,a,1,1 \slash a)
=
  a^k\,
  (\Klim^\sk_{k,m} \wtd\phi_m)(\tau, a z)
\\
={}&
  a^k\,
  \lim_{y' \ra \infty}
  (4 \pi m y')^{k-\frac{1}{2}}\,
  e( - m \tau' )\,
  \wtd\phi_m(\tau, a z, \tau')
\\
={}&
  a^k\, 
  \lim_{y' \ra \infty}
  (4 \pi m a^2 y')^{k-\frac{1}{2}}
  e( - m a^2 \tau' )\,
  \wtd\phi_m(\tau, a z, a^2 \tau')
\\
={}&
  a^k\, \Klim^\sk_{k,a^2m} \big( \wtd\phi_m(\tau, a z, a^2 \tau') \big)
=
  \Klim^\sk_{k,a^2m} \big( \wtd\phi_m \big|^\sk_{k}\, \diag(1,a,1,1 \slash a) \big)
\tx{.}
\end{align*}
\end{proof}

Next we point out that the Kohnen limit process is injective.

\begin{proposition}
\label{prop:kohnen-limit-positive-m-injective}
Let $m > 0$ and $\wtd\phi_m : \HS^{(2)} \ra \CC$ be skew-harmonic such that
\begin{gather*}
  \wtd\phi_m(\tau, z, \tau')
=
  \wht\phi_m(\tau, z, y') e(m x')
\quad
  \tx{for some function $\wht\phi_m$.}
\end{gather*}
If\/ $\Klim^\sk_{k,m}(\wtd\phi_m) = 0$, then $\wtd\phi_m = 0$.
\end{proposition}
\begin{proof}
The proof is similar to the proof of Proposition~\ref{prop:kohnen-limit-negative-m-injective}, which gives the analogous statement in the case that $m < 0$. However, note that Fourier series coefficients~$a^\sk_k(T Y)$ with positive definite~$T = \begin{psmatrix} n & r \slash 2 \\ r \slash 2 & m \end{psmatrix}$ vanish by Proposition~\ref{prop:multiplicity-one-for-skew-harmonic-fourier-coefficients-nonzero-nonnegative}.  We omit further details, since we give a complete proof of Proposition~\ref{prop:kohnen-limit-negative-m-injective}.
\end{proof}

In the case of semi-definite Fourier indices, the explicit formulas for Fourier series coefficients of skew-harmonic Siegel modular forms allow us to compute the Kohnen limit process directly.
\begin{lemma}
\label{la:kohnen-limit-positive-m-semi-definite-fourier-coefficients}
For~$m > 0$, we have
\begin{gather}
  \Klim^\sk_{k,m} \big(
  a^\sk_k\big( \begin{psmatrix} 0 & 0 \\ 0 & m \end{psmatrix}; Z \big)
  \big)
=
  1
\tx{.}
\end{gather}
\end{lemma}
\begin{proof}
Expansion~13.19.3 in~\cite{nist-dlmf} implies that
\begin{gather*}
  W_{\frac{1-k}{2}, \frac{k-1}{2}}( 4 \pi m y' )
\asymp
  (4 \pi m y')^{\frac{1-k}{2}} \exp(2 \pi m y')
\quad
  \tx{as $y' \ra \infty$.}
\end{gather*}
Now, the claim follows directly after inserting this asymptotic into the defining formula~\eqref{eq:def:fourier-coefficients-skew-positive-semi-definite} of~$a^\sk_k(T; Y)$.
\end{proof}

We need the next lemma to relate Fourier-Jacobi coefficients of positive and negative indices (see Sections~\ref{ssec:kohnen-limit-zero-m-eisenstein-series} and~\ref{ssec:proof-of-skew-maass-lift}) and to define the Hecke operator~$V_0$ for abstract Fourier-Jacobi terms (see Section~\ref{ssec:hecke-operator-V-0}).
\begin{lemma}
\label{la:kohnen-limit-positive-m-eisenstein-series}
Let $e^\sk_{k,m}$, $m > 0$ be the $m$\thdash\ Fourier-Jacobi coefficient of the Eisenstein series $E^\sk_k$. Then
\begin{gather*}
  \Klim^\sk_{k,m}\big( e^\sk_{k,m} \big)
=
  \frac{(-1)^{\frac{1-k}{2}} (2 \pi)^k}{\Gamma(\frac{1}{2}) \zeta(k)}\,
  E^\Jsk_{k,1} \big|^\Jsk_{k,1}\, V_m
\tx{,}
\end{gather*}
where $E^\Jsk_{k,1}$ is the Jacobi-Eisenstein series defined in~\eqref{eq:def:skew-holomorphic-jacobi-eisenstein-series}. In particular, we have
\begin{gather*}
  e^\sk_{k,m}
=
  e^\sk_{k,1} \big|^\sk_k\, V_m
\tx{.}
\end{gather*}
\end{lemma}
\begin{proof}
The first statement follows from (17) of~\cite{kohnen-1994} and Proposition~\ref{prop:abstract-jacobi-eisenstein-series-Vm-relation}. The second statement is a consequence of the first one, since $\Klim$ is injective (Proposition~\ref{prop:kohnen-limit-positive-m-injective}) and since $\Klim$ and $V_m$ intertwine (Proposition~\ref{prop:kohnen-limit-positive-m-covariance}).
\end{proof}

\subsection{The Kohnen limit process for \texpdf{$m < 0$}{m < 0}}
\label{ssec:kohnen-limit-negative-m}

The Kohnen limit process in~\cite{bringmann-raum-richter-2011} is defined only for $m>0$.  In this section, we extend it to negative~$m$.

Let $\wtd\phi_m \in \rmA\rmJ^\sk_{k,m}$, $m < 0$.  Set
\begin{gather}
\label{eq:def:kohnen-limit-negative-m}
  \Klim^\sk_{k,m} \big( \wtd\phi_m \big) (\tau, z)
:=
  y^{k-1}
  \lim_{y' \ra \infty}
  (4 \pi |m| y')^{\frac{1}{2}} e(m \tau')\,
  \ov{\wtd{\phi}_m(\tau, z, \tau')}
\tx{.}
\end{gather}
In contrast to the case~$m > 0$, the Kohnen limit process~\eqref{eq:def:kohnen-limit-negative-m} does not yield skew-holomorphic Jacobi forms.  More precisely, we have the next proposition, which is the main result of this section. We postpone its proof and also the proofs of the next two propositions until the end of this section.
\begin{proposition}
\label{prop:kohnen-limit-negative-m-convergence}
Let $\wtd\phi_m \in \rmA\rmJ^\sk_{k,m}$, $m < 0$. Then $\Klim^\sk_{k,m} \big( \wtd\phi_m \big)$ converges locally uniformly and
\begin{gather*}
  \Klim^\sk_{k,m} \big( \wtd\phi_m \big)
\in
  \cJ^\sk_{2-k, |m|}
\tx{.}
\end{gather*}
\end{proposition}
\begin{remark}
\label{rm:kohnen-limit-negative-m-convergence}
Recall the discussion of the space~$\cJ^\sk_{2-k, |m|}$ in Remark~\ref{rem:almost-skew-harmonic-jacobi-forms}, and in particular, the fact that it contains only Eisenstein series. Proposition~\ref{prop:kohnen-limit-negative-m-injective} shows that $\Klim^\sk_{k,m}$ is injective for~$m < 0$. Hence it is not feasible to lift skew-holomorphic Jacobi cusp forms to~$\rmA\rmJ^\sk_{k,m}$.
\end{remark}

We have the following covariance.
\begin{proposition}
\label{prop:kohnen-limit-negative-m-covariance}
Let $\wtd\phi_m \in \rmA\rmJ^\sk_{k,m}$, $m < 0$. We have the covariance properties
\begin{gather}
\label{eq:kohnen-limit-negative-m-covariance}
\begin{aligned}
  \Klim^\sk_{k,m} \big( \wtd\phi_m \big|^\sk_{k}\, g^\rmJ \big)
&=
  \Klim^\sk_{k,m} \big( \wtd\phi_m \big) \big|^\Jsk_{2-k,|m|}\, g^\rmJ
\tx{,}
\\
  \Klim^\sk_{k,a^2 m} \big( \wtd\phi_m \big|^\sk_{k}\, \diag(1,a,1, 1 \slash a) \big)
&=
  \Klim^\sk_{k,m} \big( \wtd\phi_m \big) \big|^\Jsk_{2-k,|m|} \diag(1,a,1, 1 \slash a)
\tx{,}
\end{aligned}
\end{gather}
for all $g^\rmJ \in \GJ$ and for all $a \in \RR^\times$. For all $l > 0$, we have
\begin{gather}
\label{eq:kohnen-limit-negative-m-covariance-hecke-operators}
\begin{aligned}
  \Klim^\sk_{k,l^2 m} \big( \wtd\phi_m \big|^\sk_{k}\, U_l \big)
&=
  |l|^{2k-2}\,
  \Klim^\sk_{k,m} \big( \wtd\phi_m \big) \big|^\Jsk_{2-k,|m|}\, U_l
\tx{,}\\
  \Klim^\sk_{k,l m} \big( \wtd\phi_m \big|^\sk_{k}\, V_l \big)
&=
  |l|^{k-1}\,
  \Klim^\sk_{k,m} \big( \wtd\phi_m \big) \big|^\Jsk_{2-k,|m|}\, V_l
\tx{.}
\end{aligned}
\end{gather}
\end{proposition}

Our next Proposition is analogous to Proposition~\ref{prop:kohnen-limit-positive-m-injective}, and it asserts that the Kohnen limit process is also injective for $m<0$.

\begin{proposition}
\label{prop:kohnen-limit-negative-m-injective}
Let $m < 0$ and $\wtd\phi_m : \HS^{(2)} \ra \CC$ be skew-harmonic such that
\begin{gather*}
  \wtd\phi_m(\tau, z, \tau')
=
  \wht\phi_m(\tau, z, y') e(m x')
\quad
  \tx{for some function $\wht\phi_m$.}
\end{gather*}
If\/~$\Klim^\sk_{k,m}(\wtd\phi_m) = 0$, then $\wtd\phi_m = 0$.
\end{proposition}

As in the case of $m > 0$, explicit formulas for Fourier coefficients of skew-harmonic Siegel modular forms are available if the Fourier index~$T$ is degenerate. The next lemma will be instrumental in the second of our two proofs of Proposition~\ref{prop:fourier-jacobi-coefficient-zero-m-eisenstein-series-semi-definite-horizontal}.
\begin{lemma}
\label{la:kohnen-limit-negative-m-semi-definite-fourier-coefficients}
For~$m < 0$, we have
\begin{gather}
\label{eq:la:kohnen-limit-negative-m-semi-definite-fourier-coefficients}
\begin{alignedat}{3}
&
  \Klim^\sk_{k,m} \big(
  a^\sk_k\big( \begin{psmatrix} 0 & 0 \\ 0 & m \end{psmatrix}_{(0)}; Z \big)
  \big)
&&\;=\;
  y^{k-1}
&&\;=\;
  a^{\cJ\sk}_{2-k,|m|}(0_{(0)},0;\, y)
\tx{,}
\\
&
  \Klim^\sk_{k,m} \big(
  a^\sk_k\big( \begin{psmatrix} 0 & 0 \\ 0 & m \end{psmatrix}_{(1)}; Z \big)
  \big)
&&\;=\;
  (4 \pi |m|)^{k-\frac{3}{2}}\, y^{\frac{1}{2}}
&&\;=\;
  (4 \pi |m|)^{k-\frac{3}{2}}\,
  a^{\cJ\sk}_{2-k,|m|}(0_{(1)},0;\, y)
\tx{.}
\end{alignedat}
\end{gather}
\end{lemma}
\begin{proof}
Recall from the proof of~Lemma~\ref{la:kohnen-limit-positive-m-semi-definite-fourier-coefficients} the asymptotic expansion of the $W$-Whittaker function.  Inserting that asymptotics into the defining formulas~\eqref{eq:def:fourier-coefficients-skew-negative-semi-definite} of~$a^\sk_k(T_{(0)}; Y)$ and $a^\sk_k(T_{(1)}; Y)$ yields~\eqref{eq:la:kohnen-limit-negative-m-semi-definite-fourier-coefficients}.
\end{proof}

We proceed as in~\cite{bringmann-raum-richter-2011} to prove Proposition~\ref{prop:kohnen-limit-negative-m-convergence}: We first prove Proposition~\ref{prop:kohnen-limit-negative-m-convergence} in the case of Eisenstein series (see Lemma~\ref{la:kohnen-limit-negative-m-eisenstein-series}) and then invoke multiplicity-one for Fourier series coefficients of skew-harmonic Maa\ss-Siegel forms (see Propositions~\ref{prop:multiplicity-one-for-skew-harmonic-fourier-coefficients-nonzero-nonnegative} and~\ref{prop:multiplicity-one-for-skew-harmonic-fourier-coefficients-negative}). The case of degenerate Fourier indices requires an additional argument that rests on Lemma~\ref{la:kohnen-limit-negative-m-semi-definite-fourier-coefficients}.

\begin{lemma}
\label{la:kohnen-limit-negative-m-eisenstein-series}
Let $e^\sk_{k,m}$ be the $m$\thdash\ Fourier-Jacobi coefficient of $E^\sk_k$ with~$m<0$. Then
\begin{gather*}
  \Klim^\sk_{k,m}\big( e^\sk_{k,m} \big)
\end{gather*}
exists locally uniformly, and we have
\begin{gather*}
  \Klim^\sk_{k,m}\big( e^\sk_{k,m} \big)
=
  |m|^{k-1}\,
  \frac{(-1)^{\frac{k-1}{2}} (2 \pi)^k}{\Gamma(k-\frac{1}{2}) \zeta(k)}\,
  E^\cJsk_{2-k,1} \big|_{2-k, 1} V_{|m|}
\in
  \cJ^\sk_{2-k,|m|}
\tx{,}
\end{gather*}
where $E^\cJsk_{k,1}$ is the real-analytic Jacobi-Eisenstein series defined in~\eqref{eq:def:almost-skew-harmonic-jacobi-eisenstein-series}. In particular, we have
\begin{gather*}
  e^\sk_{k,m}
=
  e^\sk_{k,-1} \big|^\sk_k\, V_{|m|}
\end{gather*}
and
\begin{gather}
  \rmR^\cJsk\,
  \Klim^\sk_{k,m} \big( e^\sk_{k,m} \big)
=
  |m|^{k-1}\,
  \frac{\Gamma(\frac{1}{2})}{\Gamma(k-\frac{1}{2})}\,
  \Klim^\sk_{k,|m|} \big( e^\sk_{k,|m|} \big)
\tx{.}
\end{gather}
\end{lemma}
\begin{proof}
We omit the proof, which is completely analogous to the proofs of the corresponding results for $m>0$ in~\cite{kohnen-1994} and Lemma~\ref{la:kohnen-limit-positive-m-eisenstein-series}. However, note that $V_{|m|}$ does not intertwine with the Kohnen limit process due to the additional factor of~$|m|^{k-1}$.
\end{proof}

\begin{proof}[{Proof of Proposition~\ref{prop:kohnen-limit-negative-m-convergence}}]
Let $m<0$. Let $\wtd\phi_m \in \rmA\rmJ^\sk_{k,m}$ and let~$e^\sk_{k,m}$ be the $m$\thdash\ Fourier-Jacobi coefficient of $E^\sk_k$. Observe that Proposition~\ref{prop:kohnen-limit-negative-m-convergence} in the special case of~$\wtd\phi_m = e^\sk_{k,m}$ reduces to Lemma~\ref{la:kohnen-limit-negative-m-eisenstein-series}, whose proof is already completed. Moveover, Lemma~\ref{la:kohnen-limit-negative-m-eisenstein-series} suffices to establish Proposition~\ref{prop:kohnen-limit-negative-m-covariance} for~$\wtd\phi_m = e^\sk_{k,m}$ (cf.\@ the proof of Proposition~\ref{prop:kohnen-limit-negative-m-covariance} below). In particular, we may apply the statement of Proposition~\ref{prop:kohnen-limit-negative-m-covariance} for~$\wtd\phi_m = e^\sk_{k,m}$ in this current proof.

Consider the Fourier series expansion
\begin{gather*}
  e^\sk_{k,m}(Z)
=
  \sum_{T = \begin{psmatrix} n & r \slash 2 \\ r \slash 2 & m \end{psmatrix}}
  c(T)\, a( T Z )
\tx{.}
\end{gather*}
Uniform convergence of $\Klim^\sk_{k,m}(a(T Z))$ for negative semi-definite, degenerate~$T$ follows from Lemma~\ref{la:kohnen-limit-negative-m-semi-definite-fourier-coefficients}. If~$T$ is indefinite or negative definite, then Propositions~\ref{prop:multiplicity-one-for-skew-harmonic-fourier-coefficients-nonzero-nonnegative} and~\ref{prop:multiplicity-one-for-skew-harmonic-fourier-coefficients-negative} guarantee that $a(T Z)$ is a scalar multiple of~$a^\sk_k(T Z)$. By Lemmas~\ref{la:jacobi-eisenstein-series-fourier-coefficients-nonzero} and~\ref{la:kohnen-limit-negative-m-eisenstein-series}, $\Klim(e^\sk_{k,1}) \ne 0$, and there exist matrices $T = \begin{psmatrix} n & r \slash 2 \\ r \slash 2 & 1 \end{psmatrix}$ that are
\begin{enumerateroman*}
\item indefinite and
\item negative definite
\end{enumerateroman*}
such that $\Klim(a^\sk_k(T Z))$ converges and is nonzero. By Proposition~\ref{prop:kohnen-limit-negative-m-covariance} for the already established case~$\wtd\phi_m = e^\sk_{k,m}$, we have the covariance
\begin{multline*}
  \Klim^\sk_{k,m}( a^\sk_k(T Z) ) \big|^\Jsk_{2-k,m}\, \rot(U)
=
  \Klim^\sk_{k,a^2 m}\big( a^\sk_k(T Z) \big|^\sk_k\, \rot(U) \big)
\\
=
  \Klim^\sk_{k,a^2 m}\big( a^\sk_k(\rT U T U\, Z) \big)
\end{multline*}
for every $U \in \GL{2}^\downarrow(\RR)$ with bottom-right entry~$a$. In particular, the right-hand side converges. Since any indefinite or negative definite~$T'$ with negative bottom-right entry can be written as $\rT U T U$ with $U \in \GL{2}^\downarrow(\RR)$, we find that the limit $\Klim(a(T' Z))$ exists for any such~$T'$. Then the locally uniform convergence of the Fourier series expansion
\begin{gather*}
  \wtd{\phi}_m(Z)
=
  \sum_{T = \begin{psmatrix} n & r \slash 2 \\ r \slash 2 & m \end{psmatrix}}
  c(\wtd{\phi}_m;\, T) a( T Z )
\end{gather*}
implies that $\Klim(\wtd{\phi}_m)$ converges.

Finally, Lemmas~\ref{la:kohnen-limit-negative-m-semi-definite-fourier-coefficients} and~\ref{la:kohnen-limit-negative-m-eisenstein-series} together with Propositions~\ref{prop:multiplicity-one-for-skew-harmonic-fourier-coefficients-nonzero-nonnegative} and~\ref{prop:multiplicity-one-for-skew-harmonic-fourier-coefficients-negative} yield that $\Klim(\wtd{\phi}_m) \in \cJ^\sk_{2-k,|m|}$.
\end{proof}

\begin{proof}[{Proof of Proposition~\ref{prop:kohnen-limit-negative-m-covariance}}]
The identities follow exactly as in the proof of Proposition~\ref{prop:kohnen-limit-positive-m-covariance}, using the locally uniform convergence of Proposition~\ref{prop:kohnen-limit-negative-m-convergence}. Observe that the additional intertwining factors arise from the different weights for the actions of~$\diag(1, l, 1, 1 \slash l)$ and~$\diag(1, \sqrt{l}, 1, 1 \slash \sqrt{l})$, respectively. 
\end{proof}

\begin{proof}[{Proof of Proposition~\ref{prop:kohnen-limit-negative-m-injective}}]
We prove the statement for Fourier series coefficients $a^\sk_k(T Y)$ with $T = \begin{psmatrix} n & r \slash 2 \\ r \slash 2 & m \end{psmatrix}$. This suffices, since the Kohnen limit process is linear, Fourier series coefficients are expressible by compact integrals. In the case of semi-definite~$T$ with~$\det(T) = 0$, Lemma~\ref{la:kohnen-limit-negative-m-semi-definite-fourier-coefficients} reveals that only zero linear combinations of $a^\sk_k(T_{(0)} Z)$ and~$a^\sk_k(T_{(1)} Z)$ vanish under $\Klim^\sk_{k,m}$. In the case of non-degenerate~$T$, we employ multiplicity-one for Fourier series coefficients established in Propositions~\ref{prop:multiplicity-one-for-skew-harmonic-fourier-coefficients-nonzero-nonnegative} and~\ref{prop:multiplicity-one-for-skew-harmonic-fourier-coefficients-negative}. In analogy to the case~$m > 0$, the claim then follows from Lemma~\ref{la:jacobi-eisenstein-series-fourier-coefficients-nonzero}, Proposition~\ref{prop:kohnen-limit-negative-m-covariance}, and Lemma~\ref{la:kohnen-limit-negative-m-eisenstein-series}.
\end{proof}

\subsection{The inverse Kohnen limit processes}
\label{ssec:inverse-kohnen-limit}

Recall from Equations~\eqref{eq:def:fourier-coefficients-siegel-modular-forms-type-skew} and~\eqref{eq:skew-jacobi-fourier-expansion} our definitions of $a^\sk_k(T Y)$ and $a^\Jsk_m(n, r; y)$, respectively. If
\begin{gather*}
  \phi(\tau, z)
=
  \sum_{n,r}
  c(n,r)\, a^\Jsk_m(n,r;y) e(nx + rz)
\end{gather*}
is a skew-holomorphic Jacobi form of weight~$k$ and index~$m > 0$, then set
\begin{multline}
\label{eq:def:inverse-kohnen-limit-positive-m}
  \Klim^{-1}_{k,m}(\phi)
=
  \Klim^{-1}_{k,m}\Big(
  \sum_{n,r}
  c(n,r)
  a^\Jsk_m(n,r; y) e(nx + r z)
  \Big)
\\
:=\;
  \sum_{n,r}
  c(n,r)\,
  a^\sk_k\big( \begin{psmatrix} n & r \slash 2 \\ r \slash 2 & m \end{psmatrix}Z \big)
\tx{.}
\end{multline}
Note that the inverse Kohnen limit process for skew weights is merely defined for~$m > 0$ (cf.\ Proposition~\ref{prop:kohnen-limit-negative-m-convergence} and Remark~\ref{rm:kohnen-limit-negative-m-convergence}).
\begin{proposition}
\label{prop:inverse-kohnen-limit-positive-m}
Let~$m > 0$. For every $\phi \in \rmJ^\sk_{k,m}$ the defining series on the right-hand side of~\eqref{eq:def:inverse-kohnen-limit-positive-m} converges absolutely and locally uniformly. The map $\Klim^{-1}_{k,m}$ is inverse to the map $\Klim^\sk_{k,m}$. Specifically, the arrows in the following diagram are well-defined isomorphisms:
\begin{center}
\begin{tikzpicture}
\matrix(m)[matrix of math nodes,
column sep=10em,
text height=1.5em, text depth=1.25ex]{%
\rmJ^\sk_{k,m} & \rmA\rmJ^\sk_{k,m} \\
};

\path
(m-1-1) edge[->, bend left = 10] node[above=.2em] {$\Klim^{-1}_{k,m}$} (m-1-2)
(m-1-2) edge[->, bend left = 10] node[below=.2em] {$\Klim^\sk_{k,m}$} (m-1-1)
;
\end{tikzpicture}
\end{center}
\end{proposition}
\begin{proof}
Prior work in~\cite{bringmann-raum-richter-2011} shows that the bottom arrow is well-defined, and Proposition~\ref{prop:kohnen-limit-positive-m-injective} gives its injectivity. It remains to show that the top map is well-defined and injective, and that the composition of $\Klim^\sk_{k,m}$ and $\Klim^{-1}_{k,m}$ is the identity map.

Absolute and locally uniform convergence of the right-hand side of~\eqref{eq:def:inverse-kohnen-limit-positive-m} follows from Shimura's~\cite{shimura-1982} estimates for $a^\sk_k(TY)$ and the polynomial growth of the Fourier series coefficients $c(\phi;\, n, r)$, where as always~$T = \begin{psmatrix} n & r \slash 2 \\ r \slash 2 & m \end{psmatrix}$. Hence $\Klim^{-1}_{k,m}$ is well-defined as a map to skew-harmonic functions. Shimura's estimates also imply the growth condition in Definition~\ref{def:abstract-fourier-jacobi-term-skew-harmonic}. Since $a^\sk_k(T Y) \ne 0$, we also find that $\Klim^{-1}_{k,m}$ is injective.

We next argue that $\Klim^\sk_{k,m} \circ \Klim^{-1}_{k,m}$ is the identity map. If $T$ is indefinite, then~\eqref{eq:def:fourier-coefficients-skew-indefinite-positive-m} asserts that $\Klim^\sk_{k,m}$ maps $a^\sk_k(T Y)$ to $a^\Jsk_m(n, r; y) e(r i v)$. For semi-definite~$T = \begin{psmatrix} 0 & 0 \\ 0 & m \end{psmatrix}$, we deduce from the asymptotic behavior of the $W$-Whittaker function that
\begin{gather*}
  \Klim^\sk_{k,m} \big( a^\sk_k(TY) \big)
=
  1
\tx{.}
\end{gather*}
This extends to any semi-definite~$T$ via the covariance of~$\Klim^\sk_{k,m}$ in Proposition~\ref{prop:kohnen-limit-positive-m-covariance}.

To conclude the proof, we have to show that the image of~$\phi$ under~$\Klim^{-1}_{k,m}$ is invariant under the action of the embedded Jacobi group, which is generated by its center, translations of $\tau$ and~$z$ by integers, and by the transformation~$S := \up(\begin{psmatrix} 0 & -1 \\ 1 & 0 \end{psmatrix})$. Invariance under the center and translations of $\tau$ and~$z$ follows directly from the defining expression of~$\Klim^{-1}_{k,m}(\phi)$.

Set
\begin{gather*}
  \wtd\phi
:=
  \Klim^{-1}_{k,m}(\phi)
\quad\tx{and}\quad
  \wtd\phi'
:=
  \Klim^{-1}_\pm(\phi) \big|^\sk_k\, S
\tx{.}
\end{gather*}
Shimura's estimates employed above imply moderate growth of $\wtd\phi'$. By Proposition~\ref{prop:kohnen-limit-positive-m-injective}, it suffices to prove the vanishing of
\begin{gather*}
  \Klim^\sk_{k,m}\big( \wtd\phi - \wtd\phi' \big)
\tx{.}
\end{gather*}
This limit exists locally uniformly, because~$S$ acts as a compact map on~$\HJ$. Moreover, the Kohnen limit process is linear and covariant (see Proposition~\ref{prop:kohnen-limit-positive-m-covariance}), and we discover that
\begin{multline*}
  \Klim^\sk_{k,m}\big( \wtd\phi - \wtd\phi' \big)
=
  \Klim^\sk_{k,m}(\wtd\phi)
  -
  \Klim^\sk_{k,m}(\wtd\phi')
\\
=
  \Klim^\sk_{k,m}(\wtd\phi)
  -
  \Klim^\sk_{k,m}(\wtd\phi) \big|^\Jsk_{k,m}\,
  \begin{psmatrix} 0 & -1 \\ 1 & 0 \end{psmatrix}
=
  \wtd\phi
  -
  \wtd\phi \big|^\Jsk_{k,m}\,
  \begin{psmatrix} 0 & -1 \\ 1 & 0 \end{psmatrix}
=
  0
\tx{.}
\end{multline*}
This completes the proof.
\end{proof}

We will need a similar inverse Kohnen limit process connecting $\rmA\rmJ^\sk_{k,m}$ and $\cJ^\sk_{2-k,|m|}$ in the case that $m < 0$. Suppose that~$\phi \in \cJ^\sk_{2-k,|m|}$ has the Fourier series expansion
\begin{align*}
  \phi(\tau, z)
=\;&
  \sum_{\substack{n,r \in \ZZ \\ D = 0}}
  c(n_{(0)},r)\, a^{\cJ\sk}_{2-k,|m|}(n_{(0)},r;y) e(nx + rz)
\\
  +\,&
  \sum_{\substack{n,r \in \ZZ \\ D = 0}}
  c(n_{(1)},r)\, a^{\cJ\sk}_{2-k,|m|}(n_{(1)},r;y) e(nx + rz)
\\
  +\,&
  \sum_{\substack{n,r \in \ZZ \\ D \ne 0}}
  c(n,r)\, a^{\cJ\sk}_{2-k,|m|}(n,r;y) e(nx + rz)
\tx{,}
\end{align*}
where $D := 4 n |m| - r^2$. Set
\begin{align}
\label{eq:def:inverse-kohnen-limit-negative-m}
  \Klim^{-1}_{k,m}(\phi)
:=\;&
  \sum_{\substack{n,r \in \ZZ \\ 4mn - r^2 = 0}}
  c(-n_{(0)}, -r)\,
  a^\sk_k\big( \begin{psmatrix} n & r \slash 2 \\ r \slash 2 & m \end{psmatrix}_{(0)} Z \big)
\\\nonumber
  +\,&
  \sum_{\substack{n,r \in \ZZ \\ 4mn - r^2 = 0}}
  (4 \pi |m|)^{\frac{3}{2}-k} c(-n_{(1)}, -r)\,
  a^\sk_k\big( \begin{psmatrix} n & r \slash 2 \\ r \slash 2 & m \end{psmatrix}_{(1)} Z \big)
\\\nonumber
  +\,&
  \sum_{\substack{n,r \in \ZZ \\ 4mn - r^2 \ne 0}}
  |m|^{k-1}\,
  2^{9-2k} \pi^{\frac{5}{2}-k} \Gamma(k-1)\,
  c(-n,-r)\,
  a^\sk_k\big( \begin{psmatrix} n & r \slash 2 \\ r \slash 2 & m \end{psmatrix} Z \big)
\tx{.}
\end{align}

Recall from Section~\ref{ssec:harmonic-jacobi-forms} the connection between $\cJ^\sk_{2-k,|m|}$ and the isomorphic space $\bbJ_{3-k,|m|}$ of harmonic Maa\ss-Jacobi forms. If $\phi \in \bbJ_{3-k,|m|}$, then we define for convenience
\begin{gather}
\label{eq:def:inverse-kohnen-limit-negative-m-maass-jacobi}
  \Klim^{-1}_{k,m} (\phi)
\;:=\;
  \Klim^{-1}_{k,m} \big( \rmR^\bbJ_{3-k}\, \phi \big)
\tx{.}
\end{gather}

The next lemma complements Lemma~\ref{la:kohnen-limit-negative-m-semi-definite-fourier-coefficients} and it is needed in the proof of Proposition~\ref{prop:inverse-kohnen-limit-negative-m}.
\begin{lemma}
\label{la:kohnen-limit-negative-m-indefinite-fourier-coefficients}
If $m < 0$ and $\begin{psmatrix} n & r \slash 2 \\ r \slash 2 & m \end{psmatrix}$ is indefinite, then we have
\begin{gather}
\label{eq:la:kohnen-limit-negative-m-indefinite-fourier-coefficients}
  \Klim^\sk_{k,m}\Big(
  a^\sk_k\big( \begin{psmatrix} n & r \slash 2 \\ r \slash 2 & m \end{psmatrix} Z \big)
  \Big)
=
  |m|^{1-k}
  \frac{2^{2k-9} \pi^{k-\frac{5}{2}} }{\Gamma(k-1)}\,
  a^\cJsk_{2-k,|m|}(-n, -r;\, y) e(-nx - rz)
\tx{.}
\end{gather}
\end{lemma}
\begin{proof}
We will establish the claim as follows: We first employ Lemmas~\ref{la:kohnen-limit-positive-m-eisenstein-series} and~\ref{la:kohnen-limit-negative-m-eisenstein-series} to deduce the Kohnen limit process of $a^\sk_k\big( \begin{psmatrix} 1 & 0 \\ 0 & -1 \end{psmatrix} Z \big)$. Then the covariance in Propositions~\ref{prop:kohnen-limit-positive-m-covariance} and~\ref{prop:kohnen-limit-negative-m-covariance} allows us to extend the result to arbitrary~$\begin{psmatrix} n & r \slash 2 \\ r \slash 2 & m \end{psmatrix}$ with~$m < 0$.

Observe that the Fourier series expansion of the first Fourier-Jacobi series coefficient~$e^\sk_{k,1}$ of~$E^\sk_k$ equals
\begin{gather*}
  e^\sk_{k,1}(Z)
=
  \sum_{n,r \in \ZZ}
  c\big( E^\sk_k;\, \begin{psmatrix} n & r \slash 2 \\ r \slash 2 & 1 \end{psmatrix} \big)\,
  a^\sk_k\big( \begin{psmatrix} n & r \slash 2 \\ r \slash 2 & 1 \end{psmatrix} Z \big)
\tx{.}
\end{gather*}
Uniform convergence of the Kohnen limit process in Proposition~\ref{prop:kohnen-limit-positive-m-convergence} together with the at most polynomial growth of the Fourier series coefficients of~$E^\sk_k$ justifies the interchanging of the summation in the Fourier series expansion and the Kohnen limit process. This yields the equality
\begin{gather*}
  \Klim^\sk_{k,1}\big( e^\sk_{k,1} \big)
=
  \sum_{n,r \in \ZZ}
  c\big( E^\sk_k;\, \begin{psmatrix} n & r \slash 2 \\ r \slash 2 & 1 \end{psmatrix} \big)\,
  \Klim^\sk_{k,1}\Big(
  a^\sk_k\big( \begin{psmatrix} n & r \slash 2 \\ r \slash 2 & 1 \end{psmatrix} Z \big)
  \Big)
\tx{.}
\end{gather*}
The normalization of~$a^\sk_k( T Z )$ in~\eqref{eq:def:fourier-coefficients-skew-indefinite-positive-m} for indefinite Fourier indices~$T$ allows us to compute the $(-1,0)$\thdash\ Fourier series coefficient in the above expression:
\begin{gather*}
  c\Big(
  \Klim^\sk_{k,1}\big( e^\sk_{k,1} \big);\;
  -1, 0
  \Big)
=
  c\big( E^\sk_k;\, \begin{psmatrix} -1 & 0 \\ 0 & 1 \end{psmatrix} \big)
\tx{.}
\end{gather*}

Recall from Lemma~\ref{la:kohnen-limit-positive-m-eisenstein-series} that
\begin{gather*}
  \Klim^\sk_{k,1}\big( e^\sk_{k,1} \big)
=
  \frac{(-1)^{\frac{1-k}{2}} (2 \pi)^k}{\Gamma(\frac{1}{2}) \zeta(k)}\,
  E^\Jsk_{k,1}
\tx{,}
\end{gather*}
and then comparing Fourier series coefficients gives
\begin{gather*}
  c\big( E^\sk_k;\, \begin{psmatrix} -1 & 0 \\ 0 & 1 \end{psmatrix} \big)\,
=
  \frac{(-1)^{\frac{1-k}{2}} (2 \pi)^k}{\Gamma(\frac{1}{2}) \zeta(k)}\,
  c\big( E^\Jsk_{k,1};\, -1,0 \big)
\tx{.}
\end{gather*}

A similar argument applies to the $(-1,0)$\thdash\ Fourier-Jacobi coefficient~$e^\sk_{k,-1}$ of~$E^\sk_k$, and with the help of Lemma~\ref{la:kohnen-limit-negative-m-eisenstein-series} we find that
\begin{gather*}
  \ov{c\big( E^\sk_k;\, \begin{psmatrix} 1 & 0 \\ 0 & -1 \end{psmatrix} \big)}\,
  \Klim^\sk_{k,-1}\Big(
  a^\sk_k\big( \begin{psmatrix} 1 & 0 \\ 0 & -1 \end{psmatrix} Z \big)
  \Big)
=
  \frac{(-1)^{\frac{k-1}{2}} (2 \pi)^k}{\Gamma(k-\frac{1}{2}) \zeta(k)}\,
  c\big( E^\cJsk_{2-k,1};\, -1, 0;\, y \big) e(-x)
\tx{.}
\end{gather*}
Note that $E^\sk_k$ is invariant under the slash action of $\rot(\begin{psmatrix} 0 & 1 \\ 1 & 0 \end{psmatrix})$, and hence (since~$k$ is odd) $c\big( E^\sk_k;\, \begin{psmatrix} 1 & 0 \\ 0 & -1 \end{psmatrix} \big) = c\big( E^\sk_k;\, \begin{psmatrix} -1 & 0 \\ 0 & 1 \end{psmatrix} \big)$. We insert the equation for~$c\big( E^\sk_k;\, \begin{psmatrix} -1 & 0 \\ 0 & 1 \end{psmatrix} \big)$ to obtain
\begin{gather*}
  \Klim^\sk_{k,1}\Big(
  a^\sk_k\big( \begin{psmatrix} 1 & 0 \\ 0 & -1 \end{psmatrix} Z \big)
  \Big)
=
  \frac{\Gamma(\frac{1}{2})}{\Gamma(k-\frac{1}{2})}\,
  \frac{c\big( E^\cJsk_{2-k,1};\, -1, 0;\, y \big) e(-x)}
       {\ov{c\big( E^\Jsk_{k,1};\, -1,0 \big)}}
\tx{.}
\end{gather*}
Recall from Lemma~\ref{la:jacobi-eisenstein-series-fourier-coefficients-real} that $E^\Jsk_{k,1}$ has real Fourier series coefficients, and hence the complex conjugation in the denominator can be ignored.

We write
\begin{gather}
  c\big( E^\cJsk_{2-k,1};\, -1, 0;\, y \big)
=
  c\big( E^\cJsk_{2-k,1};\, -1, 0 \big)\,
  a^\cJsk_{2-k,1}(-1, 0;\, y)
\tx{,}
\end{gather}
where $a^\cJsk_{2-k,1}$ is defined in~\eqref{eq:def:fourier-coefficients-almost-skew-Jacobi-nonzero}. This transforms the above equation to
\begin{gather}
\label{eq:la:kohnen-limit-negative-m-indefinite-fourier-coefficients:partially-evaluated}
  \Klim^\sk_{k,1}\Big(
  a^\sk_k\big( \begin{psmatrix} 1 & 0 \\ 0 & -1 \end{psmatrix} Z \big)
  \Big)
=
  \frac{\Gamma(\frac{1}{2})}{\Gamma(k-\frac{1}{2})}\,
  \frac{c\big( E^\cJsk_{2-k,1};\, -1, 0 \big)}
       {c\big( E^\Jsk_{k,1};\, -1,0 \big)}\;
  a^\cJsk_{2-k,1}(-1, 0;\, y) e(-x)
\tx{.}
\end{gather}

To complete the proof, we have to analyze the relation between the Fourier series coefficients of $E^\cJsk_{2-k,1}$ and $E^\Jsk_{k,1}$. In~\eqref{eq:almost-skew-harmonic-eisenstein-series-raising-image}, we provided the relation
\begin{gather*}
  \rmR^\cJsk_{2-k}\, E^{\cJ\sk}_{2-k,m}
=
  E^{\Jsk}_{k,m}
\tx{.}
\end{gather*}
In particular, we have
\begin{multline*}
  c\big(E^{\cJsk}_{2-k,1};\, -1, 0 \big)\,
  \rmR^\cJsk_{2-k}\big(
  a^\cJsk_{2-k,1}(-1,0;\, y) e(-x)
  \big)
=
  \rmR^\cJsk_{2-k}\big(
  c\big(E^{\cJsk}_{2-k,1};\, -1, 0;\, \tau, z \big)
  \big)
\\
=
  c\big(E^{\Jsk}_{k,1};\, -1, 0;\, \tau, z \big)
=
  c\big(E^{\Jsk}_{k,1};\, -1, 0 \big)
  a^\Jsk_1(-1, 0;\, y) e(-x)
\tx{,}
\end{multline*}
where $c(E^{\cJsk}_{2-k,1};\, -1, 0;\, \tau, z )$ and $c(E^{\Jsk}_{k,1};\, -1, 0;\, \tau, z )$ are the $(-1,0)$\thdash\ Fourier terms of~$E^{\cJsk}_{2-k,1}$ and~$E^{\Jsk}_{k,1}$. Combining the covariance of the raising operator with respect to the extended Jacobi group with the second equation in~\eqref{eq:almost-skew-harmonic-eisenstein-series-raising-image} and Lemma~\ref{la:jacobi-eisenstein-series-fourier-coefficients-real}, we observe that the Fourier term~$\rmR^\cJsk_{2-k}\big( a^\cJsk_{2-k,1}(-1,0;\, y) e(-x) \big)$ does not vanish identically. We find that
\begin{gather*}
  \frac{c\big(E^{\cJsk}_{2-k,1};\, -1, 0 \big)}
       {c\big(E^{\Jsk}_{k,1};\, -1, 0 \big)}
=
  \frac{a^\Jsk_1(-1, 0;\, y) e(-x)}
       {\rmR^\cJsk_{2-k}\big(
        a^\cJsk_{2-k,1}(-1,0;\, y) e(-x)
        \big)}
\tx{,}
\end{gather*}
and inserting the definition of~$a^\cJsk_{2-k,1}$ in~\eqref{eq:def:fourier-coefficients-almost-skew-Jacobi-nonzero} gives:
\begin{gather*}
  \frac{c\big(E^{\cJsk}_{2-k,1};\, -1, 0 \big)}
       {c\big(E^{\Jsk}_{k,1};\, -1, 0 \big)}
=
  \frac{a^\Jsk_1(-1, 0;\, y) e(-x)}
       {\rmR^\cJsk_{2-k} \circ \rmR^\bbJ_{3-k} \big(
        a^\bbJ_{3-k,1}(-1,0;\, y) e(- x)
        \big)}
\tx{.}
\end{gather*}
Note that the composition of the two raising operators can be expressed as a single raising operator by~\eqref{eq:almost-skew-harmonic-eisenstein-series-raising-and-xi}, since $a^\bbJ_{3-k,1}(-1, 0;\, y) e(-x)$ is annihilated by the differential operator $\Delta^\rmJ_{3-k}$ and it is also holomorphic in~$z$:
\begin{gather*}
  \frac{c\big(E^{\cJsk}_{2-k,1};\, -1, 0 \big)}
       {c\big(E^{\Jsk}_{k,1};\, -1, 0 \big)}
=
  \frac{- i \Gamma(k-\frac{1}{2})}{8 \pi^{\frac{1}{2}}\, \Gamma(k-1)}\,
  \frac{a^\Jsk_1(-1, 0;\, y) e(-x)}
       {\rmR^\Jsk_{k-2} \big(
        y^{\frac{5}{2}-k}
        a^\bbJ_{3-k,1}(-1,0;\, y) e(- x)
        \big)}
\tx{.}
\end{gather*}

Next we recall the definition~\eqref{eq:def:jacobi-raising-operator} of the raising operator and the definition~\eqref{eq:def:harmonic-maass-jacobi:fourier-coefficient:negative} of~$a^\bbJ_{3-k,1}$. A straightforward computation shows that
\begin{multline*}
  \rmR^\Jsk_{k-2} \big(
  y^{\frac{5}{2}-k}
  a^\bbJ_{3-k,1}(-1,0;\, y) e(-x)
  \big)
\\
=
  \Big(
  \partial_{\ov{\tau}}
  + v y^{-1} \partial_{\ov{z}}
  + \tfrac{i}{2}(k-2-\tfrac{1}{2}) y^{-1}
  \Big) \Big(
  y^{\frac{5}{2}-k}
  \Gamma\big( \tfrac{3}{2} - k,\, 4 \pi y \big)
  e(-\tau)
  \Big)
=
  \tfrac{i}{2} (4 \pi)^{k - \frac{5}{2}}\, e(-\ov{\tau})
\tx{.}
\end{multline*}
Moreover, $a^\Jsk_{1}(-1,0;\,y) = e(i y)$, and hence
\begin{gather*}
  \frac{c\big(E^{\cJsk}_{2-k,1};\, -1, 0 \big)}
       {c\big(E^{\Jsk}_{k,1};\, -1, 0 \big)}
=
  \frac{- i \Gamma(k-\frac{1}{2})}{8 \pi^{\frac{1}{2}}\, \Gamma(k-1)}\,
  \frac{i}{2} (4 \pi)^{k - \frac{5}{2}}
=
  \frac{2^{2k-9} \pi^{k-3}\, \Gamma(k-\frac{1}{2})}{\Gamma(k-1)}
\tx{.}
\end{gather*}
We insert this last expression into~\eqref{eq:la:kohnen-limit-negative-m-indefinite-fourier-coefficients:partially-evaluated} to obtain
\begin{gather*}
  \Klim^\sk_{k,1}\Big(
  a^\sk_k\big( \begin{psmatrix} 1 & 0 \\ 0 & -1 \end{psmatrix} Z \big)
  \Big)
=
  \frac{2^{2k-9} \pi^{k-\frac{5}{2}} }{\Gamma(k-1)}\,
  a^\cJsk_{2-k,1}(-1, 0;\, y) e(-x)
\tx{.}
\end{gather*}

It remains to employ the covariance from Proposition~\ref{prop:kohnen-limit-negative-m-covariance} to extend the result to general Fourier indices~$\begin{psmatrix} n & r \slash 2 \\ r \slash 2 & m \end{psmatrix}$ with~$m<0$. First, Proposition~\ref{prop:multiplicity-one-for-skew-harmonic-fourier-coefficients-nonzero-nonnegative} guarantees that $a^\sk_k(T Z)$ for indefinite~$T$ depends only on the trace and determinant of~$T Z$. This implies the equality
\begin{gather*}
  \Klim^\sk_{k,1}\Big(
  a^\sk_k\big( \begin{psmatrix} 1 & 0 \\ 0 & -m \end{psmatrix} Z \big)
  \Big)
=
  |m|^{-\frac{k}{2}}
  \Klim^\sk_{k,1}\Big(
  a^\sk_k\big( \begin{psmatrix} 1 & 0 \\ 0 & -1 \end{psmatrix} Z \big)
  \big|^\sk_k\, \diag(1, |m|^{\frac{1}{2}}, 1, |m|^{-\frac{1}{2}})
  \Big)
\tx{.}
\end{gather*}
Next, the application of the second equation in~\eqref{eq:kohnen-limit-negative-m-covariance} and the extension of the skew-Jacobi slash action in~\eqref{eq:rot-quotient-action} show that this is equal to
\begin{align*}
&
  |m|^{-\frac{k}{2}}
  \Klim^\sk_{k,1}\Big(
  a^\sk_k\big( \begin{psmatrix} 1 & 0 \\ 0 & -1 \end{psmatrix} Z \big)
  \Big)
  \Big|^\Jsk_{2-k,1}\, \diag(1, |m|^{\frac{1}{2}}, 1, |m|^{-\frac{1}{2}})
\\
={}&
  |m|^{-\frac{k}{2}}
  \frac{2^{2k-9} \pi^{k-\frac{5}{2}} }{\Gamma(k-1)}\,
  a^\cJsk_{2-k,1}(-1, 0;\, y) e(-x)
  \big|^\Jsk_{2-k,1}\, \diag(1, |m|^{\frac{1}{2}}, 1, |m|^{-\frac{1}{2}})
\\
={}&
  |m|^{1-k}
  \frac{2^{2k-9} \pi^{k-\frac{5}{2}} }{\Gamma(k-1)}\,
  a^\cJsk_{2-k,|m|}(-1, 0;\, y) e(-x)
\tx{.}
\end{align*}
For the last equality, we have used the covariance of Fourier terms with respect to the action of~$\diag(1, |m|^{\frac{1}{2}}, 1, |m|^{-\frac{1}{2}})$ that follows directly from the defining equations~\eqref{eq:def:harmonic-maass-jacobi:fourier-coefficient:negative} and~\eqref{eq:def:fourier-coefficients-almost-skew-Jacobi-nonzero}. The first equation in~\eqref{eq:kohnen-limit-negative-m-covariance} now yields the claim.

\end{proof}

\begin{proposition}
\label{prop:inverse-kohnen-limit-negative-m}
Let $m < 0$. For every $\phi \in \cJ^\sk_{2-k,|m|}$ the defining series on the right-hand side of~\eqref{eq:def:inverse-kohnen-limit-negative-m} converges absolutely and locally uniformly. The map $\Klim^{-1}_{k,m}$ is inverse to the map $\Klim^\sk_{k,m}$. Specifically, the arrows in the following diagram are well-defined isomorphisms:
\begin{center}
\begin{tikzpicture}
\matrix(m)[matrix of math nodes,
column sep=10em,
text height=1.5em, text depth=1.25ex]{%
\bbJ_{3-k,|m|} & \cJ^\sk_{2-k,|m|} & \rmA\rmJ^\sk_{k,m} \\
};

\path
(m-1-1) edge[->, bend left = 10] node[above=.2em] {$\rmR^\bbJ_{3-k}$} (m-1-2)
(m-1-2) edge[->, bend left = 10] node[below=.2em] {$y^{k-\frac{3}{2}} \circ \big( \rmR^\Jsk \big)^{k-2} $} (m-1-1)
(m-1-2) edge[->, bend left = 10] node[above=.2em] {$\Klim^{-1}_{k,m}$} (m-1-3)
(m-1-3) edge[->, bend left = 10] node[below=.2em] {$\Klim^\sk_{k,m}$} (m-1-2)
;
\end{tikzpicture}
\end{center}
\end{proposition}
\begin{proof}
The left half of the diagram appears in Proposition~\ref{prop:harmonic-and-almost-skew-harmonic-jacobi-forms}. Consider the right half of the diagram: Convergence and injectivity of the inverse Kohnen limit process follow as in the case of $m > 0$ in Proposition~\ref{prop:inverse-kohnen-limit-positive-m}. The argument that $\Klim^\sk_{k,m} \circ \Klim^{-1}_{k,m}$ is the identity map for Fourier series terms of degenerate index is also the same as in Proposition~\ref{prop:inverse-kohnen-limit-positive-m} (see Lemma~\ref{la:kohnen-limit-negative-m-semi-definite-fourier-coefficients}). However, a separate argument for Fourier series terms of indefinite index is required, since the analogue of~\eqref{eq:def:fourier-coefficients-skew-indefinite-positive-m} for $m < 0$ does not hold. The remaining part of the proof is then again similar to the one of Proposition~\ref{prop:inverse-kohnen-limit-positive-m}.

Thus, it remains to justify that
\begin{gather*}
  \Klim^\sk_{k,m}\Big(
  \Klim^{-1}_{k,m}\big(
  a^{\cJ\sk}_{2-k,m}(n,r;y) e(nx + rz)
  \big)
  \Big)
=
  a^{\cJ\sk}_{2-k,m}(n,r;y) e(nx + rz)
\tx{.}
\end{gather*}
This follows after inserting the definition of the inverse Kohnen limit process in~\eqref{eq:def:inverse-kohnen-limit-negative-m}, and then applying Lemma~\ref{la:kohnen-limit-negative-m-indefinite-fourier-coefficients}.
\end{proof}

\section{The Kohnen limit process for \texpdf{$m = 0$}{m = 0}}
\label{sec:kohnen-limit-zero-m}

The Kohnen limit process for~$m = 0$ is difficult, and we have to overcome five subtleties:
\begin{enumeratearabic}
\item
Elements of~$\rmA\rmJ^\sk_{k,0}$ are not determined uniquely by the leading asymptotic term as $y' \ra \infty$. In particular, the naive analogue of $\Klim^\sk_{k,m}$ for~$m \ne 0$ would not be injective. For this reason, we define two ``semi-definite'' Kohnen limit processes~$\Klim[0,0]^\sk_{k,0}$ and~$\Klim[0,1]^\sk_{k,0}$ in Section~\ref{ssec:kohnen-limit-zero-m-semi-definite} and a family of ``indefinite'' Kohnen limit processes~$\Klim[N]^\sk_{k,0}$ for~$N \ge 1$ in Section~\ref{ssec:kohnen-limit-zero-m-indefinite}.

\item
In contrast to the case~$m \ne 0$, the defining series of the $0$\thdash\ Fourier-Jacobi coefficient~$e^\sk_{k,0}$ of the skew Eisenstein series~$E^\sk_k$ has non-uniform asymptotic behavior as $y' \ra \infty$. In the case~$m \ne 0$, we interchanged limit and summation to directly compute the Kohnen limit process of~$e^\sk_{k,m}$. For~$m = 0$, we first have to apply the Poisson summation formula in Section~\ref{ssec:kohnen-limit-zero-m-eisenstein-series}.

\item
The cases~$m \ne 0$ rely heavily on multiplicity-one for Fourier series coefficients (see Propositions~\ref{prop:multiplicity-one-for-skew-harmonic-fourier-coefficients-nonzero-nonnegative} and~\ref{prop:multiplicity-one-for-skew-harmonic-fourier-coefficients-negative}). Fourier series coefficients of index~$T = \begin{psmatrix} 0 & 0 \\ 0 & 0 \end{psmatrix}$ violate multiplicity-one, and can occur in the Fourier series expansions of elements of~$\rmA\rmJ^\sk_{k,0}$. Observe that Theorem~3~(a) of~\cite{bringmann-raum-richter-2011} incorrectly states that the ``constant term'' of elements of~$\rmM^\sk_k$ is independent of~$Y$, which however does not impact any of the other results of~\cite{bringmann-raum-richter-2011}. Proposition~\ref{prop:multiplicity-one-for-skew-harmonic-fourier-coefficients-zero} corrects this inaccuracy of~\cite{bringmann-raum-richter-2011} by providing details on the Fourier series coefficients of weight~$k$ skew-harmonic functions of index~$T = \begin{psmatrix} 0 & 0 \\ 0 & 0 \end{psmatrix}$.

\item
The Kohnen limit processes for $m \ne 0$ are defined via the leading asymptotic term as~$y' \ra \infty$. The semi-definite Kohnen limit processes~$\Klim[0,0]^\sk_{k,0}$ and~$\Klim[0,1]^\sk_{k,0}$ capture the first two terms in the asymptotic behavior as~$y' \ra \infty$. However, the remaining Kohnen limit processes~$\Klim[N]^\sk_{k,0}$ for~$N \ge 1$ are defined via a limit~$y' \ra 0$. We do not provide details, but the third asymptotic term as~$y' \ra \infty$ would yield a discontinuous Kohnen limit process that does not converge uniformly.

\item
When defining~$\Klim[N]^\sk_{k,0}$, the Jacobi-variable~$z$ is specialized at torsion points. This collapses parts of the Fourier series expansion of an abstract Fourier-Jacobi term of index~$0$. For a given~$n$, there are infinitely many~$r$ such that Fourier series coefficients of index~$\begin{psmatrix} n & r \slash 2 \\ r \slash 2 & 0 \end{psmatrix}$ contribute to the $n$\thdash\ Fourier series coefficient of the image under the Kohnen limit process. This complication will be addressed in detail in a sequel.
\end{enumeratearabic}

\subsection{Abstract elliptic functions}
\label{ssec:abstract-elliptic-functions}

Abstract Fourier-Jacobi terms of index~$m \ne 0$ live in between skew-harmonic Siegel modular forms and skew-Jacobi forms. Abstract elliptic functions in the next definition are analogous objects in between skew-harmonic Siegel modular forms and elliptic functions. In particular, we do not impose any transformation invariance with respect to~$\SL{2}(\ZZ)$.
\begin{definition}[Abstract elliptic functions]
\label{def:abstract-elliptic-function-skew}
Suppose $\wtd\phi_0 :\, \HS^{(2)} \ra \CC$ is a function such that $\wtd\phi_0(\tau, z, \tau') = \wht\phi_0(\tau, z, y')$ is independent of~$x'$. Then $\wtd\phi_0$ is an abstract skew-harmonic elliptic function of weight~$k$ if it satisfies 
\begin{enumerateroman}
\item
\label{it:def:abstract-elliptic-function-skew:differential-operators}
$\wtd\phi_0$ is skew-harmonic, i.e., \@ $\Omega^\sk_k\, \wtd\phi_0 = 0$ and $\xi^\sk_k\, \wtd\phi_0 = 0$.

\item
\label{it:def:abstract-elliptic-function-skew:invariance}
For all $\gaJ$ in the Heisenberg group $\ZZ^2 \ltimes \ZZ$ (viewed as embedded into $\Sp{2}(\ZZ)$), we have the Siegel invariance $\wtd\phi_0 \circ \gaJ = \wtd\phi_0$.

\item
$\wtd\phi_0(Z) = \cO(\tr(Y)^a)$ for some $a \in \RR$ as $\tr(Y) \ra \infty$.
\end{enumerateroman}
The space of such functions is denoted by $\rmA\rmEll^\sk_k$.
\end{definition}
\begin{remark}
Note that Condition~\ref{it:def:abstract-elliptic-function-skew:differential-operators} of Definition~\ref{def:abstract-elliptic-function-skew} depends on $k$, while the Siegel invariance in~\ref{it:def:abstract-elliptic-function-skew:invariance} is independent of the weight~$k$.
\end{remark}

An abstract skew-harmonic elliptic function~$\wtd\phi_0$ and its transforms $\wtd\phi_0 |^\sk_k\, \up(\gamma)$ with $\gamma \in \SL{2}(\ZZ)$ are invariant with respect to the Heisenberg group~$\ZZ^2 \ltimes \ZZ$ as in Condition~\ref{it:def:abstract-elliptic-function-skew:invariance} in Definition~\ref{def:abstract-elliptic-function-skew}. They have Fourier coefficients with respect to~$x$, and by invariance under the translation~$z \mto z + 1$, these have a Fourier series expansion with respect to~$u$. If $n \in \RR_{>0}$, then that expansion is of the form
\begin{gather}
\label{eq:abstract-elliptic-function-skew-fourier-expansion-positive-n}
  c\big(\wtd\phi_0\big|^\sk_k\,\up(\gamma);\, n, 0\big)\,
  a^\sk_k\big( \begin{psmatrix} n & 0 \\ 0 & 0 \end{psmatrix} Y \big)
  \,+\,
  \sum_{r \in \ZZ \setminus \{0\}}
  c\big(\wtd\phi_0\big|^\sk_k\,\up(\gamma);\, n, r\big)\,
  a^\sk_k\big( \begin{psmatrix} n & r \slash 2 \\ r \slash 2 & 0 \end{psmatrix} Y \big) e(r u)
\text{,}
\end{gather}
if $n \in \RR_{<0}$, then it is of the form
\begin{multline}
\label{eq:abstract-elliptic-function-skew-fourier-expansion-negative-n}
  c\big(\wtd\phi_0\big|^\sk_k\,\up(\gamma);\, n_{(0)}, 0\big)\,
  a^\sk_k\big( \begin{psmatrix} n & 0 \\ 0 & 0 \end{psmatrix}_{(0)} Y \big)
  \,+\,
  c\big(\wtd\phi_0\big|^\sk_k\,\up(\gamma);\, n_{(1)}, 0\big)\,
  a^\sk_k\big( \begin{psmatrix} n & 0 \\ 0 & 0 \end{psmatrix}_{(1)} Y \big)
\\
  +\,
  \sum_{r \in \ZZ \setminus \{0\}}
  c\big(\wtd\phi_0\big|^\sk_k\,\up(\gamma);\, n, r\big)\,
  a^\sk_k\big( \begin{psmatrix} n & r \slash 2 \\ r \slash 2 & 0 \end{psmatrix} Y \big) e(r u)
\tx{,}
\end{multline}
and if $n = 0$, then it is of the form
\begin{multline}
\label{eq:abstract-elliptic-function-skew-fourier-expansion-zero-n}
  c\big(\wtd\phi_0\big|^\sk_k\,\up(\gamma);\, 0_{(0)}, 0\big)\,
  a^\sk_k\big( \begin{psmatrix} 0 & 0 \\ 0 & 0 \end{psmatrix}_{(0)} Y \big)
  \,+\,
  c\big(\wtd\phi_0\big|^\sk_k\,\up(\gamma);\, 0_{(1)}, 0; Y\big)
\\
  +\,
  \sum_{r \in \ZZ \setminus \{0\}}
  c\big(\wtd\phi_0\big|^\sk_k\,\up(\gamma);\, 0, r\big)\,
  a^\sk_k\big( \begin{psmatrix} 0 & r \slash 2 \\ r \slash 2 & 0 \end{psmatrix} Y \big) e(r u)
\tx{.}
\end{multline}
Recall from Proposition~\ref{prop:multiplicity-one-for-skew-harmonic-fourier-coefficients-zero} that invariance of $\wtd\phi_0 |^\sk_k\, \up(\ga)$ under the action of $\rot(\begin{psmatrix} 1 & 0 \\ 1 & 1\end{psmatrix})$, i.e., the transformation~$z \mto z + \tau$, implies that the second term in~\eqref{eq:abstract-elliptic-function-skew-fourier-expansion-zero-n} can be expanded in terms of Fourier coefficients in~\eqref{eq:def:fourier-coefficients-skew-zero:sub-coefficients} as
\begin{gather}
\label{eq:abstract-elliptic-function-skew-fourier-expansion-zero-term}
\begin{aligned}
  &
  c\big(\wtd\phi_0\big|^\sk_k\,\up(\gamma);\, (0_{(1)}, 0_{(0)}), 0 \big)
  a^\sk_k\big( \begin{psmatrix} 0 & 0 \\ 0 & 0 \end{psmatrix}_{(1)}, 0_{(0)};\, \yb \big)
\\
  +\,&
  c\big(\wtd\phi_0\big|^\sk_k\,\up(\gamma);\, (0_{(1)}, 0_{(1)}), 0 \big)
  a^\sk_k\big( \begin{psmatrix} 0 & 0 \\ 0 & 0 \end{psmatrix}_{(1)}, 0_{(1)};\, \yb \big)
\\
  +\,&
  \sum_{\nb \in \ZZ \setminus \{0\}}
  c\big(\wtd\phi_0\big|^\sk_k\,\up(\gamma);\, (0_{(1)}, \nb), 0 \big)
  a^\sk_k\big( \begin{psmatrix} 0 & 0 \\ 0 & 0 \end{psmatrix}_{(1)}, \nb;\, \yb \big) e(\nb \xb)
\tx{.}
\end{aligned}
\end{gather}

We naturally have $\rmA\rmJ^\sk_{k,0} \subset \rmA\rmEll^\sk_k$. In addition, we define the subspace $\rmA\rmEll^{\skmd}_k \subset \rmA\rmEll^\sk_k$ of abstract elliptic functions that have a Fourier series expansion with respect to~$x$ and whose Fourier series coefficients are of the following moderate growth:
\begin{multline}
\label{eq:def:abstract-elliptic-function-skew-moderate-growth}
  \rmA\rmEll^{\skmd}_k
:=
  \Big\{
  \wtd\phi_0 \in \rmA\rmEll^\sk_k \,:\,
  c\big( \wtd\phi_0;\, n, r \big)
  \tx{\@ grows polynomially with respect to~$r$,}
\\
  c\big( \wtd\phi_0;\, (0_{(1)}, \nb), 0 \big)
  \tx{\@ grows polynomially with respect to~$\nb$}
  \Big\}
\tx{.}
\end{multline}
For convenience and to unify notation, set
\begin{gather}
\label{eq:def:abstract-fourier-jacobi-skew-moderate-growth}
  \rmA\rmJ^\skmd_{k,0}
\;:=\;
  \rmA\rmJ^\sk_{k,0}
  \,\cap\,
  \rmA\rmEll^\skmd_k
\tx{.}
\end{gather}

\begin{proposition}
\label{prop:fourier-jacobi-coefficients-zero-m-moderate-growth}
If\/ $F(Z) = \sum_{m \in \ZZ} \wtd\phi_m(\tau, z, y') e(m x') \in \rmM^\sk_k$, then
\begin{gather*}
  \wtd\phi_0 \in \rmA\rmJ^\skmd_{k,0}
\tx{.}
\end{gather*}
\end{proposition}
\begin{proof}
It is clear that $\wtd\phi_0 \in \rmA\rmJ^\sk_{k,0}$.

Since $F$ is a Siegel modular form, it is in particular invariant under $\rot(\SL{2}(\ZZ))$.  Observing the Fourier series expansion in~\eqref{eq:abstract-elliptic-function-skew-fourier-expansion-zero-term}, we discover that~$c\big(\wtd\phi_0\big|^\sk_k\,\up(\gamma);\, 0_{(1)}, 0;\, Y\big)$ is up to the factor~$\det(Y)^{\frac{1-k}{2}}$ a Maa\ss\ form of eigenvalue $(k-1)(2-k)$ with respect to~$\taub$. In particular, its Fourier series coefficients~$c\big(\wtd\phi_0\big|^\sk_k\,\up(\gamma);\, (0_{(1)}, \nb), 0\big)$ grow at most polynomially with respect to~$\nb$.

We next argue that~$c\big( \wtd\phi_0;\, n, r \big)$ grows polynomially with respect to~$r$. If $n > 0$, then Proposition~\ref{prop:inverse-kohnen-limit-positive-m} and the normalization~\eqref{eq:def:fourier-coefficients-skew-indefinite-positive-m} imply that $c\big( \wtd\phi_0;\, n, r \big) = c\big( \wtd\phi_n;\, 0, r \big)$ is the Fourier series coefficient of a skew-holomorphic Jacobi form. When combining the theta decomposition of weight~$k$ skew-holomorphic Jacobi forms with the trivial bound on the growth of Fourier series coefficients of half-integral weight modular forms (cf.\ the introduction of~\cite{iwaniec-1987}, where a stronger bound is provided), we find that for fixed~$n > 0$
\begin{gather*}
  c\big( \wtd\phi_0;\, n, r \big)
=
  \cO_\epsilon\big(
  \big(|r|^2\big)^{\frac{k - 1 \slash 2}{2} - \frac{1}{4} + \epsilon}
  \big)
=
  \cO_\epsilon\big(
  |r|^{k - 1 + 2\epsilon}
  \big)
\quad
\tx{for any $\epsilon > 0$ as $r \rightarrow \infty$.}
\end{gather*}
For general, fixed $n \in \ZZ$, if $r \ne 0$, we have
\begin{gather*}
  c\big( \wtd\phi_0;\, n, r \big)
=
  c\big( \wtd\phi_0;\, n + 2hr, r \big)
=
  c\big( \wtd\phi_{n + 2hr};\, 0, r \big)
=
  \cO_{h,\epsilon}\big( |r|^{k - 1 + 2\epsilon} \big)
\quad
\tx{for any $\epsilon > 0$ as $r \rightarrow \infty$,}
\end{gather*}
provided that $h \in \ZZ$ with $n + 2hr > 0$. For fixed $n$, we can choose $h = \sgn(n)\sgn(r) n$ and hence obtain an estimate that only depends on~$n$.
\end{proof}

\subsection{The semi-definite Kohnen limit process}
\label{ssec:kohnen-limit-zero-m-semi-definite}

In this section, we define two Kohnen limit processes for $m=0$. We refer to them as the semi-definite Kohnen limit processes, because they allow us to extract the Fourier series coefficients of semi-definite index. However, they will also capture the Fourier coefficients of indices~$( \begin{psmatrix} 0 & 0 \\ 0 & 0 \end{psmatrix}_{(0)}, 0 )$, $( \begin{psmatrix} 0 & 0 \\ 0 & 0 \end{psmatrix}_{(1)}, 0_{(0)} )$, and $( \begin{psmatrix} 0 & 0 \\ 0 & 0 \end{psmatrix}_{(1)}, 0_{(1)} )$ with notation as in~\eqref{eq:def:fourier-coefficients-skew-zero}.

For $\wtd\phi_0 \in \rmA\rmEll^\sk_k$, set
\begin{gather}
\label{eq:def:kohnen-limit-zero-m-semi-definite-0}
  \Klim[0,0]^\sk_{k,0}
  \big(
  \wtd\phi_0
  \big) (\tau)
\;:=\;
  y^{k-\frac{1}{2}}\,
  \lim_{y' \ra \infty}
  \wtd{\phi}_0 (\tau, 0, \tau')
\tx{.}
\end{gather}
Proposition~\ref{prop:kohnen-limit-zero-m-semi-definite-convergence} asserts that it converges absolutely and locally uniformly to a function on the Poincar\'e upper half plane~$\HS$ that has spectral parameter~$k-\frac{1}{2}$, i.e., eigenvalue $\frac{1}{2} (k - \frac{1}{2})$ with respect to the weight~$1-k$ hyperbolic Laplace operator.

The Kohnen limit process~$\Klim[0,0]^\sk_{k,0}$ for index~$m = 0$ enters with an adjoint that we denote by $\Klim[0,0]^{\sk\,\#}_{k,0}$. Consider a periodic function $\phi_0 :\, \HS \ra \CC$ of eigenvalue $\frac{1}{2} (k - \frac{1}{2})$ with respect to the weight~$1-k$ hyperbolic Laplace operator, which has a Fourier series expansion of the form
\begin{gather*}
  c(\phi_0;\, 0_{(0)}) y^{k-\frac{1}{2}}
  \,+\,
  c(\phi_0;\, 0_{(1)}) y^{\frac{1}{2}}
  \,+\,
  \sum_{n \in \ZZ \setminus \{0\}}
  c(\phi_0;\, n)\,
  (4 \pi |n| y)^{\frac{k-1}{2}} W_{\sgn(n)\frac{1-k}{2}, \frac{k-1}{2}}(4 \pi |n| y)\, e(nx)
\tx{.}
\end{gather*}
If the Fourier series coefficients of~$\phi_0$ grow at most polynomially, we set
\begin{multline}
\label{eq:def:kohnen-limit-zero-m-adjoint-semi-definite-0}
  \Klim[0,0]^{\sk\,\#}_{k,0}
  \big(
  \phi_0
  \big) (Z)
\;:=\;
  c(\phi_0, 0_{(0)}
  a^\sk_k\big( \begin{psmatrix} 0 & 0 \\ 0 & 0 \end{psmatrix}_{(0)}; Y \big)
  \,+\,
  c(\phi_0, 0_{(1)}
  a^\sk_k\big( \begin{psmatrix} 0 & 0 \\ 0 & 0 \end{psmatrix}_{(1)}, 0_{(0)}; Y \big)
\\
  +\,
  \sum_{n = 1}^\infty
  c(\phi_0, n)\,
  a^\sk_k\big( \begin{psmatrix} n & 0 \\ 0 & 0 \end{psmatrix}; Y \big)
  \,+\,
  \sum_{n = 1}^\infty
  c(\phi_0, -n)\,
  a^\sk_k\big( \begin{psmatrix} -n & 0 \\ 0 & 0 \end{psmatrix}_{(0)}; Y \big)
\tx{.}
\end{multline}
The key property of this adjoint is given in the next lemma.
\begin{lemma}
\label{la:kohnen-limit-zero-m-adjoint-semi-definite-0}
Let $\phi_0 :\, \HS \ra \CC$ be a periodic function of eigenvalue $\frac{1}{2} (k - \frac{1}{2})$ with respect to the weight~$1-k$ hyperbolic Laplace operator. Assume that the Fourier series coefficients of~$\phi_0$ grow moderately in the sense of~\eqref{eq:def:abstract-elliptic-function-skew-moderate-growth}. Then we have
\begin{gather}
\label{eq:la:kohnen-limit-zero-m-adjoint-semi-definite-0}
  \Klim[0,0]^\sk_{k,0}\big(
  \Klim[0,0]^{\sk\,\#}_{k,0} \big(\phi_0\big)
  \big)
=
  \phi_0
\tx{.}
\end{gather}
\end{lemma}
\begin{proof}
This follows directly from the defining equations~\eqref{eq:def:fourier-coefficients-skew-positive-semi-definite}, \eqref{eq:def:fourier-coefficients-skew-negative-semi-definite}, and~\eqref{eq:def:fourier-coefficients-skew-zero} of the Fourier series terms that appear in~\eqref{eq:def:kohnen-limit-zero-m-adjoint-semi-definite-0} (see also Lemma~\ref{la:kohnen-limit-zero-m-semi-definite-fourier-coefficients}).
\end{proof}
For convenience, we set
\begin{gather}
  \Klim[0,0]^{\sk\,\pi}_{k,0}
\;:=\;
  \Klim[0,0]^{\sk\,\#}_{k,0} \circ \Klim[0,0]^\sk_{k,0}
\tx{.}
\end{gather}
If that Kohnen limit process converges, then by Lemma~\ref{la:kohnen-limit-zero-m-adjoint-semi-definite-0}, we can view it as a projection of the Fourier expansion to contributions of indices~$\begin{psmatrix} 0 & 0 \\ 0 & 0 \end{psmatrix}_{(0)}$, $( \begin{psmatrix} 0 & 0 \\ 0 & 0 \end{psmatrix}_{(1)}, 0_{(0)} )$, $\begin{psmatrix} n & 0 \\ 0 & 0 \end{psmatrix}$ for $n > 0$, and $\begin{psmatrix} n & 0 \\ 0 & 0 \end{psmatrix}_{(0)}$ for $n < 0$.

We next use $\Klim[0,0]^{\sk\,\pi}_{k,0}$ to remove contributions that are already captured by the first semi-definite Kohnen limit process~$\Klim[0,0]^\sk_{k,0}$, and we define a second Kohnen limit process that extracts the remaining Fourier series coefficients of semi-definite index. Again, let $\wtd\phi_0 \in \rmA\rmEll^\sk_k$. Then
\begin{gather}
\label{eq:def:kohnen-limit-zero-m-semi-definite-1}
  \Klim[0,1]^\sk_{k,0}
  \big(
  \wtd\phi_0
  \big) (\tau)
\;:=\;
  y^{\frac{1}{2}}
  \lim_{y' \ra \infty}
  y^{\prime\, k-\frac{3}{2}}
  \Big(
  \ov{%
  \wtd{\phi}_0 (\tau, 0, \tau')
  -
  \Klim[0,0]^{\sk\,\pi}_{k,0} \big( \wtd{\phi}_0 \big) (\tau, 0, \tau')
  }\Big)
\tx{.}
\end{gather}
Proposition~\ref{prop:kohnen-limit-zero-m-semi-definite-convergence} shows that \eqref{eq:def:kohnen-limit-zero-m-semi-definite-1} converges absolutely and locally uniformly to a holomorphic function on the Poincar\'e upper half plane~$\HS$.

In analogy with~\eqref{eq:def:kohnen-limit-zero-m-adjoint-semi-definite-0}, if $\phi_0$ is a periodic and holomorphic function on~$\HS$ such that its Fourier series coefficients have at most polynomial growth, then the adjoint to $\Klim[0,1]^\sk_{k,0}$ is defined by
\begin{multline}
\label{eq:def:kohnen-limit-zero-m-adjoint-semi-definite-1}
  \Klim[0,1]^{\sk\,\#}_{k,0}\Big(
  \sum_{n = 0}^\infty c(n) e(n \tau)
  \Big)(Z)
\;:=\;
\\
  \ov{c(\phi_0; 0)}
  a^\sk_k\big(\begin{psmatrix} 0 & 0 \\ 0 & 0 \end{psmatrix}_{(1)}, 0_{(1)};\, Y \big)
  \,+\,
  \sum_{n = 1}^\infty
  \frac{\ov{c(\phi_0; n)}}{(4 \pi n)^{k-2}} a^\sk_k\big(\begin{psmatrix} -n & 0 \\ 0 & 0 \end{psmatrix}_{(1)};\, Z \big)
\tx{.}
\end{multline}
The next lemma is analogous to Lemma~\ref{la:kohnen-limit-zero-m-adjoint-semi-definite-0}, asserting that $\Klim[0,1]^{\sk\,\#}_{k,0}$ is indeed an adjoint to $\Klim[0,1]^{\sk}_{k,0}$.
\begin{lemma}
\label{la:kohnen-limit-zero-m-adjoint-semi-definite-1}
Let $\phi_0 :\, \HS \ra \CC$ be periodic and holomorphic. Assume that the Fourier series coefficients of~$\phi_0$ grow moderately in the sense of~\eqref{eq:def:abstract-elliptic-function-skew-moderate-growth}. Then we have
\begin{gather}
\label{eq:la:kohnen-limit-zero-m-adjoint-semi-definite-1}
  \Klim[0,1]^\sk_{k,0}\big(
  \Klim[0,1]^{\sk\,\#}_{k,0}\, \phi_0
  \big)
=
  \phi_0
\tx{.}
\end{gather}
\end{lemma}
\begin{proof}
This follows again directly from the defining equations~\eqref{eq:def:fourier-coefficients-skew-positive-semi-definite}, \eqref{eq:def:fourier-coefficients-skew-negative-semi-definite}, and~\eqref{eq:def:fourier-coefficients-skew-zero} of the Fourier series terms that appear in~\eqref{eq:def:kohnen-limit-zero-m-adjoint-semi-definite-0} (see also Lemma~\ref{la:kohnen-limit-zero-m-semi-definite-fourier-coefficients}).
\end{proof}
The associated projection on Fourier series expansions is
\begin{gather}
  \Klim[0,1]^{\sk\,\pi}_{k,0}
\;:=\;
  \Klim[0,1]^{\sk\,\#}_{k,0} \circ \Klim[0,1]^\sk_{k,0}
\tx{.}
\end{gather}
It maps a skew-harmonic Fourier series expansion to contributions of indices~$( \begin{psmatrix} 0 & 0 \\ 0 & 0 \end{psmatrix}_{(1)}, 0_{(1)} )$ and $\begin{psmatrix} n & 0 \\ 0 & 0 \end{psmatrix}_{(1)}$ for $n < 0$. 

Analogous to the case of~$m \ne 0$, the Kohnen limit processes~$\Klim[0,0]^\sk_{k,0}$ and~$\Klim[0,1]^\sk_{k,0}$ converge already on Fourier series coefficients. The next lemma specifies the image of the semi-definite Kohnen limit processes on Fourier series coefficients. We will use it frequently in the next sections.
\begin{lemma}
\label{la:kohnen-limit-zero-m-semi-definite-fourier-coefficients}
Let $n > 0$. We have
\begin{alignat}{2}
&
  \Klim[0,0]^\sk_{k,0} \big(
  a^\sk_k\big( \begin{psmatrix} 0 & 0 \\ 0 & 0 \end{psmatrix}_{(0)};\, Z \big)
  \big)
&&=
  y^{k-\frac{1}{2}}
\tx{,}
\\&
  \Klim[0,0]^\sk_{k,0} \big(
  a^\sk_k\big( \begin{psmatrix} 0 & 0 \\ 0 & 0 \end{psmatrix}_{(1)}, 0_{(0)};\, Z \big)
  \big)
&&=
  y^{\frac{1}{2}}
\tx{,}
\\&
  \Klim[0,0]^\sk_{k,0} \big(
  a^\sk_k\big( \begin{psmatrix} n & 0 \\ 0 & 0 \end{psmatrix};\, Z \big)
  \big)
&&=
  ( 4 \pi n )^{\frac{1}{2}-k}\,
  ( 4 \pi n y )^{\frac{k-1}{2}}
  W_{\frac{1-k}{2}, \frac{k-1}{2}}\big( 4 \pi n y \big)
  e(n x)
\tx{,}
\\&
  \Klim[0,0]^\sk_{k,0} \big(
  a^\sk_k\big( \begin{psmatrix} -n & 0 \\ 0 & 0 \end{psmatrix}_{(0)};\, Z \big)
  \big)
&&=
  ( 4 \pi n )^{\frac{1}{2}-k}\,
  ( 4 \pi n y )^{\frac{k-1}{2}}
  W_{\frac{k-1}{2}, \frac{k-1}{2}}\big( 4 \pi n y \big)
  e(- n x)
\tx{;}
\\[.3em]&
  \Klim[0,1]^\sk_{k,0} \big(
  a^\sk_k\big( \begin{psmatrix} 0 & 0 \\ 0 & 0 \end{psmatrix}_{(1)}, 0_{(1)};\, Z \big)
  \big)
&&=
  1
\tx{,}
\\&
  \Klim[0,1]^\sk_{k,0} \big(
  a^\sk_k\big( \begin{psmatrix} -n & 0 \\ 0 & 0 \end{psmatrix}_{(1)};\, Z \big)
  \big)
&&=
  (4 \pi n)^{k-2} e(n \tau)
\tx{.}
\end{alignat}
The remaining Fourier series coefficients vanish under the semi-definite Kohnen limit processes.
\end{lemma}
\begin{proof}
For Fourier series coefficients of index~$T$ with $\det(T) = 0$, this follows directly from the defining formulas for Fourier series coefficients in~\eqref{eq:def:fourier-coefficients-skew-positive-semi-definite}, \eqref{eq:def:fourier-coefficients-skew-negative-semi-definite}, and~\eqref{eq:def:fourier-coefficients-skew-zero}. In particular, the asymptotic expansion of the modified $K$-Bessel functions suffices to conclude that the Fourier series terms of index~$(\begin{psmatrix} 0 & 0 \\ 0 & 0 \end{psmatrix}_{(1)}, \nb)$ with~$\nb \ne 0$ in~\eqref{eq:def:fourier-coefficients-skew-zero} do not contribute to the image. Shimura's estimate for the decay of Fourier series terms of indefinite index implies that they decay faster than any polynomial in~$y'$ and therefore do not contribute to the semi-definite Kohnen limit processes.
\end{proof}

In contrast to the case~$m \ne 0$, one can examine $\Klim[0,0]$ and $\Klim[0,1]$ without inspecting the skew Eisenstein series.
\begin{proposition}
\label{prop:kohnen-limit-zero-m-semi-definite-convergence}
Let $\wtd\phi_0 \in \rmA\rmEll^\skmd_k$. Then $\Klim[0,0]^\sk_{k,0}(\wtd\phi_0)$ and $\Klim[0,1]^\sk_{k,0}(\wtd\phi_0)$ converge absolutely and locally uniformly. If\/ $\wtd\phi_0 \in \rmA\rmJ^\sk_{k,0}$, then
\begin{enumerateroman}
\item
\label{it:prop:kohnen-limit-zero-m-semi-definite-convergence:0}
$\Klim[0,0]^\sk_{k,0}(\wtd\phi_0)$ is an elliptic Maa\ss\ form of weight~$1-k$ and eigenvalue $\frac{1}{2} (k - \frac{1}{2})$ with respect to the weight~$1-k$ hyperbolic Laplace operator, and
\item
\label{it:prop:kohnen-limit-zero-m-semi-definite-convergence:1}
$\Klim[0,1]^\sk_{k,0}(\wtd\phi_0)$ is a holomorphic elliptic modular form of weight~$k-1$.
\end{enumerateroman}
\end{proposition}
\begin{proof}
Consider the Fourier series expansion of~$\wtd\phi_0$, whose terms are given in~\eqref{eq:def:fourier-coefficients-skew-positive-semi-definite}, \eqref{eq:def:fourier-coefficients-skew-negative-semi-definite}, \eqref{eq:def:fourier-coefficients-skew-zero}, and~\eqref{eq:def:fourier-coefficients-skew-indefinite-positive-m}. Since $\wtd\phi_0$ has Fourier series coefficients that grow moderately, we can compute the Kohnen limit process by employing Lemma~\ref{la:kohnen-limit-zero-m-semi-definite-fourier-coefficients}. The images of the Fourier series coefficients given in Lemma~\ref{la:kohnen-limit-zero-m-semi-definite-fourier-coefficients} also reveal the Laplace eigenvalues of~$\Klim[0,0]^\sk_{k,0}(\wtd\phi_0)$. Similarly, it follows that~$\Klim[0,1]^\sk_{k,0}(\wtd\phi_0)$ is holomorphic.

The claimed modular transformation behavior in~\ref{it:prop:kohnen-limit-zero-m-semi-definite-convergence:0} and~\ref{it:prop:kohnen-limit-zero-m-semi-definite-convergence:1} will follow from the covariance of the semi-definite Kohnen limit processes stated in Proposition~\ref{prop:kohnen-limit-zero-m-semi-definite-covariance}.
\end{proof}
\begin{remark}
Proposition~\ref{prop:kohnen-limit-zero-m-semi-definite-convergence} could also be proved by using techniques from~Section~\ref{sec:kohnen-limit-nonzero-m} (see Lemma~\ref{la:kohnen-limit-zero-m-eisenstein-series}).
\end{remark}

\begin{proposition}
\label{prop:kohnen-limit-zero-m-semi-definite-covariance}
Let $\wtd\phi_0 \in \rmA\rmEll^\skmd_k$. Then
\begin{gather}
\label{eq:prop:kohnen-limit-zero-m-semi-definite-covariance}
\begin{aligned}
  \Klim[0,0]^\sk_{k,0} \big( \wtd\phi_0 \big)
  \big|_{1-k}\, g
&=
  \Klim[0,0]^\sk_{k,0} \big( \wtd\phi_0 \big|^\sk_{k}\, \up(g) \big)
\tx{,}
\\
  \Klim[0,1]^\sk_{k,0} \big( \wtd\phi_0 \big)
  \big|_{k-1}\, g
&=
  \Klim[0,1]^\sk_{k,0} \big( \wtd\phi_0 \big|^\sk_{k}\, \up(g) \big)
\tx{,}
\end{aligned}
\end{gather}
for all $g \in \SL{2}(\QQ)$.

If $\phi_0$ is an elliptic Maa\ss\ form of weight~$1-k$ and eigenvalue $\frac{1}{2} (k - \frac{1}{2})$ with respect to the weight~$1-k$ hyperbolic Laplace operator, then
\begin{gather}
\label{eq:prop:kohnen-limit-zero-m-adjoint-semi-definite-0-covariance}
  \Klim[0,0]^{\sk\,\#}_{k,0} \big( \phi_0 \big|_{1-k}\, g \big)
=
  \Klim[0,0]^{\sk\,\#}_{k,0} \big( \phi_0 \big) \big|^\sk_{k}\, \up(g)
\end{gather}
for all $g \in \SL{2}(\QQ)$. If $\phi_0$ is a holomorphic modular form on~$\HS$ of weight~$k-1$, then
\begin{gather}
\label{eq:prop:kohnen-limit-zero-m-adjoint-semi-definite-1-covariance}
  \Klim[0,1]^{\sk\,\#}_{k,0} \big( \phi_0 \big|_{k-1}\, g \big)
=
  \Klim[0,1]^{\sk\,\#}_{k,0} \big( \phi_0 \big) \big|^\sk_{k}\, \up(g)
\end{gather}
for all $g \in \SL{2}(\QQ)$.
\end{proposition}
\begin{remark}
We have to restrict the covariance statement to transformations by~$g \in \SL{2}(\QQ)$, since we have defined the adjoint Kohnen limit process only for periodic functions. For example, the right-hand side of the second equation in~\eqref{eq:prop:kohnen-limit-zero-m-semi-definite-covariance} is not defined for general~$g \in \SL{2}(\RR)$.
\end{remark}

\begin{proof}[{Proof of Proposition~\ref{prop:kohnen-limit-zero-m-semi-definite-covariance}}]
The definitions~\eqref{eq:def:kohnen-limit-zero-m-semi-definite-0} of $\Klim[0,0]^\sk_{k,0}$ and~\eqref{eq:def:kohnen-limit-zero-m-semi-definite-1} of $\Klim[0,1]^\sk_{k,0}$ together with the uniform convergence in Proposition~\ref{prop:kohnen-limit-zero-m-semi-definite-convergence} imply the covariances in~\eqref{eq:prop:kohnen-limit-zero-m-semi-definite-covariance}.

Next consider~\eqref{eq:prop:kohnen-limit-zero-m-adjoint-semi-definite-0-covariance}. Note that we can apply the adjoint Kohnen limit process to the left-hand side, since $g \in \SL{2}(\QQ)$ and $\phi_0 |_{k-1}\, g$ is periodic in~$x$ with period~$\frac{1}{N}$ for some~$N \in \ZZ_{> 0}$. Lemma~\ref{la:kohnen-limit-zero-m-adjoint-semi-definite-0} implies that~$\Klim[0,0]^\sk_{k,0}$ is injective on the image of its adjoint. Therefore, the covariance in~\eqref{eq:prop:kohnen-limit-zero-m-adjoint-semi-definite-0-covariance} is equivalent to
\begin{gather*}
  \Klim[0,0]^\sk_{k,0} \Big( \Klim[0,0]^{\sk\,\#}_{k,0} \big( \phi_0 \big|_{1-k}\, g \big) \Big)
=
  \Klim[0,0]^\sk_{k,0} \Big( \Klim[0,0]^{\sk\,\#}_{k,0} \big( \phi_0 \big) \big|^\sk_{k}\, \up(g) \Big)
\tx{,}
\end{gather*}
where we have applied the Kohnen limit process to both the left and right-hand side. We employ covariance of the Kohnen limit process and~\eqref{eq:la:kohnen-limit-zero-m-adjoint-semi-definite-0} to find that this follows from
\begin{multline*}
  \Klim[0,0]^\sk_{k,0} \Big( \Klim[0,0]^{\sk\,\#}_{k,0} \big( \phi_0 \big|_{1-k}\, g \big) \Big)
\\
=
  \phi_0 \big|_{1-k}\, g
=
  \Klim[0,0]^\sk_{k,0} \Big( \Klim[0,0]^{\sk\,\#}_{k,0} \big( \phi_0 \big) \Big) \big|_{1-k}\, g
\\
=
  \Klim[0,0]^\sk_{k,0} \Big( \Klim[0,0]^{\sk\,\#}_{k,0} \big( \phi_0 \big) \big|^\sk_{k}\, \up(g) \Big)
\tx{.}
\end{multline*}
Hence \eqref{eq:prop:kohnen-limit-zero-m-adjoint-semi-definite-0-covariance} holds.

The covariance in~\eqref{eq:prop:kohnen-limit-zero-m-adjoint-semi-definite-1-covariance} follows as the covariance in~\eqref{eq:prop:kohnen-limit-zero-m-adjoint-semi-definite-0-covariance} when also using Lemma~\ref{la:kohnen-limit-zero-m-adjoint-semi-definite-1}.
\end{proof}

The next operator can be thought of as a projection to semi-definite Fourier series coefficients. Note that some but not all Fourier series terms of index~$\begin{psmatrix} 0 & 0 \\ 0 & 0 \end{psmatrix}$ contribute to its image (see Lemma~\ref{la:kohnen-limit-zero-m-projection-semi-definite}).
We set
\begin{gather}
\label{eq:def:kohnen-limit-zero-m-semi-definite-projection}
  \Klim[0]^{\sk\,\pi}_{k,0} \big( \wtd\phi_0 \big)
\;:=\;
  \Klim[0,0]^{\sk\,\pi}_{k,0} \big( \wtd\phi_0 \big)
  +
  \Klim[0,1]^{\sk\,\pi}_{k,0} \big( \wtd\phi_0 \big)
\tx{.}
\end{gather}

The covariance in~\eqref{eq:prop:kohnen-limit-zero-m-adjoint-semi-definite-0-covariance} and~\eqref{eq:prop:kohnen-limit-zero-m-adjoint-semi-definite-1-covariance} implies that, if $\wtd\phi_0 \in \rmA\rmJ^\sk_k$, then
\begin{gather}
\label{eq:prop:kohnen-limit-zero-m-projection-semi-definite-covariance}
  \Klim[0]^{\sk\,\pi}_{k,0} \big( \wtd\phi_0 \big|^\sk_{k}\, \up(g) \big)
=
  \Klim[0]^{\sk\,\pi}_{k,0} \big( \wtd\phi_0 \big) \big|^\sk_{k}\, \up(g)
\end{gather}
for all $g \in \SL{2}(\QQ)$.

\begin{lemma}
\label{la:kohnen-limit-zero-m-projection-semi-definite}
Let $\wtd\phi_0 \in \rmA\rmEll^\skmd_k$. Then
\begin{gather*}
  \Big( \wtd\phi_0 - \Klim[0]^{\sk\,\pi}_{k,0} \big( \wtd\phi_0 \big) \Big)(Z)
\end{gather*}
has a Fourier series expansion supported on indefinite~$T$ and indices $\big( \begin{psmatrix} 0 & 0 \\ 0 & 0 \end{psmatrix}_{(1)}, \nb \big)$ with $\nb \ne 0$ as in~\eqref{eq:def:fourier-coefficients-skew-zero}.
\end{lemma}
\begin{proof}
This follows directly from~\eqref{eq:la:kohnen-limit-zero-m-adjoint-semi-definite-0} in Lemma~\ref{la:kohnen-limit-zero-m-adjoint-semi-definite-0} and~\eqref{eq:la:kohnen-limit-zero-m-adjoint-semi-definite-1} in Lemma~\ref{la:kohnen-limit-zero-m-adjoint-semi-definite-1}, and the explicit expressions for the Fourier series terms in Lemma~\ref{la:kohnen-limit-zero-m-semi-definite-fourier-coefficients}.
\end{proof}

\subsection{The indefinite Kohnen limit process}
\label{ssec:kohnen-limit-zero-m-indefinite}

We now examine indefinite Fourier coefficients of $\wtd\phi_0 \in \rmA\rmEll^\skmd_k$. If $N \in \ZZ_{\ge 1}$, then set
\begin{multline}
\label{eq:def:kohnen-limit-zero-m-indefinite}
  \Klim[N]^\sk_{k,0}
  \big(
  \wtd\phi_0
  \big) (\tau)
\\:=\;
  y^{\frac{1}{2}}
  \lim_{y' \ra 0}
  y^{\prime\,k-\frac{3}{2}}
  \Bigg(
  \ov{\Big(
  \big( \wtd\phi_0 - \Klim[0]^{\sk\,\pi}_{k,0} (\wtd\phi_0) \big)
  \big|^\sk_k\, \up\big(\begin{psmatrix} 1 & 0 \\ 0 & 1 \end{psmatrix}, (\alpha,\beta,0) \big)
  \Big)\, (\tau, 0, \tau')}
  \Bigg)_{\substack{\alpha,\beta \in \frac{1}{N}\ZZ \slash \ZZ \\ \alpha \ZZ + \beta \ZZ = \frac{1}{N} \ZZ}}
\tx{.}
\end{multline}
\begin{remarks}
\begin{enumeratearabic}
\item	
Observe that \eqref{eq:def:kohnen-limit-zero-m-indefinite} defines a vector-valued function $\HS \ra V(\rho_N)$, where $\rho_N$ is given in Section~\ref{ssec:elliptic-vector-valued-modular-forms}.

\item
In a sequel, we will establish the following: Let the limit $\wtd\phi_0 \in \rmA\rmEll^\skmd_k$ and $N \in \ZZ_{\ge 1}$. Then $\Klim[N]^\sk_{k,0}(\wtd\phi_0)$ converges absolutely and locally uniformly. If $\wtd\phi_0 \in \rmA\rmJ^\skmd_{k,0}$, then $\Klim[N]^\sk_{k,0}(\wtd\phi_0)$ is an elliptic Maa\ss\ form of weight~$1-k$ and eigenvalue $0$ with respect to the weight~$1-k$ hyperbolic Laplace operator.
\end{enumeratearabic}
\end{remarks}

\subsection{The Kohnen limit process for Eisenstein series}
\label{ssec:kohnen-limit-zero-m-eisenstein-series}

The proof of the convergence of the limit $\Klim[N]^\sk_{k,0}$ for Fourier-Jacobi coefficients of Eisenstein series is much more subtle in comparison to the case $m \ne 0$, since the defining series has no uniform asymptotic as $y' \ra 0$.
\begin{lemma}
\label{la:kohnen-limit-zero-m-eisenstein-series}
Let $e^\sk_{k,0}$ be the $0$\thdash\ Fourier-Jacobi coefficient of $E^\sk_k$ and $N \in \ZZ_{\ge 1}$. Then
\begin{gather*}
  \Klim[N]^\sk_{k,0}\big( e^\sk_{k,0} \big)
\end{gather*}
exists locally uniformly and is an Eisenstein series. More precisely, we have
\begin{align}
\label{eq:la:kohnen-limit-zero-m-eisenstein-series-semi-definite-0}
  \Klim[0,0]^\sk_{k,0} \big( e^\sk_{k,0} \big) (\tau)
&=
  E_{1-k}\big( k-\tfrac{1}{2}, \tau \big)
\tx{,}
\\
\label{eq:la:kohnen-limit-zero-m-eisenstein-series-semi-definite-1}
  \Klim[0,1]^\sk_{k,0} \big( e^\sk_{k,0} \big) (\tau)
&=
  \frac{(-1)^{\frac{1-k}{2}} 2^{2-k} \pi \Gamma(k-\frac{3}{2})}
       {\Gamma(k - \frac{1}{2})}\,
  \frac{\zeta(k-1) \zeta(2k-3)}{\zeta(k) \zeta(2k-2)}\,
  E_{k-1}(\tau)
\tx{,}
\\
\label{eq:la:kohnen-limit-zero-m-eisenstein-series-indefinite}
  \Klim[N]^\sk_{k,0} \big( e^\sk_{k,0} \big) (\tau)
&=
  \frac{(-1)^{\frac{1+k}{2}} 2^{6-3k} \pi^{\frac{3}{2}}\, \Gamma(k)}{\Gamma(k-\frac{1}{2})}\,
  \frac{\zeta(k-1)}{\zeta(k) \zeta(2k-2)}\;
  H \big( E_{k-1,N}(\tau) \big)
\tx{,}
\end{align}
where $E_{k-1,N}(\tau)$ is the Eisenstein series in~\eqref{eq:def:vector-valued-elliptic-eisenstein-series-holomorphic} and $H$ is the linear map from $V(\rho_N)$ to~$V(\rho_N)$ that maps $\frakf_{c,d}$ (with $\gcd(c,d,N) = 1$) to the following linear combination of~$\frake_{\alpha,\beta}$'s:
\begin{gather}
\label{eq:la:kohnen-limit-zero-m-eisenstein-series-indefinite:matrix}
  \sum_{\substack{\alpha,\beta \in \frac{1}{N}\ZZ \slash \ZZ \\ \gcd(N \alpha, N \beta, N) = 1}}\!\!
  \frake_{\alpha,\beta}\,
  \sum_{\substack{\ov{l} = 1 \\ \ov{l} (\alpha d - \beta c) \not\in \ZZ}}^N
  \Bigg( \frac{2}{1 - e\big( \ov{l} (\beta c - \alpha d) \big)} - 1 \Bigg)
  \frac{\zeta\big( 2k-3, \frac{\ov{l}}{N} \big)}{N^{2k - 3}}
\tx{.}
\end{gather}
\end{lemma}
\begin{proof}
We compute the asymptotic expansion of the Eisenstein series~$E^\sk_k$ from its defining series, as in~\cite{kohnen-1994}:
\begin{gather*}
  \sum_{C,D} \det(C Z + D)^{k-1} |\det(C Z + D)|^{1-2k}
\tx{.}
\end{gather*}
As a first step, we have to determine the contributions to
\begin{gather}
  e^{\sk\,\pi}_{k,0}
:=
  \Klim[0]^{\sk\,\pi}_{k,0}(e^\sk_{k,0})
\quad\tx{and}\quad
  e^{\sk\,\indef}_{k,0}
:=
  e^\sk_{k,0} - e^{\sk\,\pi}_{k,0}
\tx{.}
\end{gather}

We split the sum into contributions for which the bottom-right subblocks of $C$ of sizes~$2$, $1$, and~$0$ vanish:
\begin{gather*}
  E^\sk_k
=
  (E^\sk_k)_{(0)} + (E^\sk_k)_{(1)} + (E^\sk_k)_{(2)}
\tx{,}
\end{gather*}
where
\begin{gather}
\label{eq:def:skew-harmonic-siegel-eisenstein-subseries}
  (E^\sk_k)_{(\nu)}
:=
  \sum_{\substack{C,D \\ C_{i\!j} = 0 \tx{, if } i,j > \nu \\ C_{\nu\!\nu} \ne 0}}
  \det(C Z + D)^{k-1} |\det(C Z + D)|^{1-2k}
\tx{.}
\end{gather}

In this decomposition, $(E^\sk_k)_{(0)} = 1$ is constant. In particular, it does not contribute to $e^{\sk\,\indef}_{k,0}$. Furthermore, we have
\begin{gather*}
  (E^\sk_k)_{(1)}(Z)
=
 \sum_{U \in \GL{2}^\downarrow(\ZZ) \backslash \GL{2}(\ZZ)}
  \Bigg(
  \sum_{\begin{psmatrix} 1 & 0 \\ 0 & 1 \end{psmatrix} \ne \gamma = \begin{psmatrix} a & b \\ c & d \end{psmatrix} \in \Gamma_\infty \backslash \SL{2}(\ZZ)}\hspace{-2em}
  (c \tau + d)^{k-1} |c \tau + d|^{1-2k}
  \Bigg)
  \Big|^\sk_k\, \rot(U)
\tx{.}
\end{gather*}
Lemma~\ref{la:elliptic-eisenstein-series-fourier-expansion} with~$\kappa \leadsto 1-k$ and $s \leadsto k - \frac{1}{2}$ yields the Fourier series expansion
\begin{align}
\nonumber
  \sum_{U \in \GL{2}^\downarrow(\ZZ) \backslash \GL{2}(\ZZ)}
  \Bigg(
&
  \hphantom{+\;}
  \frac{(-1)^{\frac{1-k}{2}} 2^{2-k} \pi^{\frac{1}{2}} \Gamma(k-1)}
       {\Gamma(k - \frac{1}{2})}
  \frac{\zeta(k-1)}{\zeta(k)}
  y^{1-k}
\\\nonumber
&
  +\;
  \frac{(-1)^{\frac{1-k}{2}} 2^k \pi^k}
       {\Gamma(\frac{1}{2})\, \zeta(k)}
  \sum_{n = 1}^\infty
  \sigma_{k-1}(n)\,
  a^\sk_k\big( \begin{psmatrix} n & 0 \\ 0 & 0 \end{psmatrix} Y \big) e(nx)
\\
&
  +\;
  \frac{(-1)^{\frac{1-k}{2}} 2^k \pi^k}
       {\Gamma(k - \frac{1}{2})\, \zeta(k)}
  \sum_{n = 1}^\infty
  \sigma_{k-1}(n)\,
  a^\sk_k\big( \begin{psmatrix} -n & 0 \\ 0 & 0 \end{psmatrix}_{(0)} Y \big) e(-nx)
  \Bigg)
  \Big|^\sk_k\, \rot(U)
\tx{.}
\label{eq:fourier-coefficients-siegel-eisenstein-series-(1)-type-skew}
\end{align}
Since the following expression will also appear again later, we set for convenience
\begin{gather*}
  c_1
:=
  \frac{(-1)^{\frac{1-k}{2}} 2^{2-k} \pi^{\frac{1}{2}}\, \Gamma(k-1)}{\Gamma(k-\frac{1}{2})}\,
  \frac{\zeta(k-1)}{\zeta(k)}
\tx{.}
\end{gather*}
The second and third term in the parentheses of~\eqref{eq:fourier-coefficients-siegel-eisenstein-series-(1)-type-skew} contribute to~$e^\sk_{k,0}$ only if~$U = \begin{psmatrix} 1 & 0 \\ 0 & 1 \end{psmatrix}$. The first one yields an Eisenstein series, expressible in terms of~$\xb$ and~$\yb$. More precisely, the coordinate transform in~\eqref{eq:def:Y-coordinates} allows us to write
\begin{gather*}
  c_1
  \sum_{U \in \GL{2}^\downarrow(\ZZ) \backslash \GL{2}(\ZZ)}
  y^{1-k}
  \big|^\sk_k\, \rot(U)
=
  c_1
  \det(Y)^{\frac{1-k}{2}}
  \sum_{U \in \GL{2}^\downarrow(\ZZ) \backslash \GL{2}(\ZZ)}
  \yb^{k-1}
  \big|^\sk_k\, \rot(U)
\tx{.}
\end{gather*}
Observe that any system of representatives for~$\SL{2}^\downarrow(\ZZ) \backslash \SL{2}(\ZZ)$ is also a system of representatives for~$\GL{2}^\downarrow(\ZZ) \backslash \GL{2}(\ZZ)$. In other words, we can assume that~$U \in \SL{2}(\ZZ)$. Proposition~\ref{prop:rot-slash-action-on-zeroth-fourier-coefficient} then yields
\begin{multline*}
  c_1
  \det(Y)^{\frac{1-k}{2}}
  \sum_{U \in \SL{2}^\downarrow(\ZZ) \backslash \SL{2}(\ZZ)}
  \yb^{k-1}
  \big|^\sk_k\, \rot(U)
\\
=
  c_1
  \det(Y)^{\frac{1-k}{2}}
  \sum_{U \in \SL{2}^\downarrow(\ZZ) \backslash \SL{2}(\ZZ)}
  \Im\Big( \begin{psmatrix} 0 & 1 \\ 1 & 0 \end{psmatrix} U \begin{psmatrix} 0 & 1 \\ 1 & 0 \end{psmatrix} \taub \Big)^{k-1}
=
  c_1
  \det(Y)^{\frac{1-k}{2}}
  \sum_{\ga \in \Gamma_\infty \backslash \SL{2}(\ZZ)}
  \Im\big( \ga \taub \big)^{k-1}
\tx{.}
\end{multline*}
The second equality holds, since $\begin{psmatrix} 0 & 1 \\ 1 & 0 \end{psmatrix} \SL{2}^\downarrow(\ZZ) \begin{psmatrix} 0 & 1 \\ 1 & 0 \end{psmatrix} = \Gamma_\infty$. Employing Lemma~\ref{la:elliptic-eisenstein-series-fourier-expansion} with~$\kappa \leadsto 0$ and $s \leadsto k-1$ and the defining formulas for index~$\begin{psmatrix} 0 & 0 \\ 0 & 0 \end{psmatrix}$ Fourier coefficients in~\eqref{eq:def:fourier-coefficients-skew-zero}, we find that this equals
\begin{multline}
\label{eq:fourier-coefficients-siegel-eisenstein-series-(1)-type-skew-first-summand}
  c_1
  a^\sk_k\big( \begin{psmatrix} 0 & 0 \\ 0 & 0 \end{psmatrix}_{(1)}, 0_{(0)};\, Y \big)
  \,+\,
  \frac{c_1\, 2^{4-2k} \pi \Gamma(2k-3)}
       {\Gamma(k-1)^2}
  \frac{\zeta(2k-3)}{\zeta(2k-2)}\,
  a^\sk_k\big( \begin{psmatrix} 0 & 0 \\ 0 & 0 \end{psmatrix}_{(1)}, 0_{(1)};\, Y \big)
\\
  +\,
  \frac{c_1\, \pi^{k-\frac{3}{2}}}{\zeta(2k-2)}
  \sum_{\nb \in \ZZ \setminus \{0\}}
  \frac{|\nb|^{1-k}\, \sigma_{2k-3}(|\nb|)}
       {\Gamma(k-1)}\,
  a^\sk_k\big( \begin{psmatrix} 0 & 0 \\ 0 & 0 \end{psmatrix}_{(1)}, \nb;\, Y \big)
  e\big( 4 \pi |\nb| \xb \big)
\tx{.}
\end{multline}

The first two summands contribute to $e^{\sk\,\pi}_{k,0}$ (the first term to~$\Klim[0,0]^\sk_{k,0}(e^\sk_{k,0})$ and the second one to $\Klim[0,1]^\sk_{k,0}(e^\sk_{k,0})$) and the remaining sum contributes to~$e^{\sk\,\indef}_{k,0}$. Observe that the third term in~\eqref{eq:fourier-coefficients-siegel-eisenstein-series-(1)-type-skew-first-summand} can be written as
\begin{gather}
\label{eq:fourier-coefficients-siegel-eisenstein-series-(1)-type-skew-first-summand:indefinite-contribution}
  c_1\,
  \det(Y)^{\frac{1-k}{2}}\,
  \Big(
  E_0(k-1,\taub)
  \,-\,
  1
  \,-\,
  \frac{2^{4-2k} \pi \Gamma(2k-3)}
       {\Gamma(k-1)^2}
  \frac{\zeta(2k-3)}{\zeta(2k-2)}\,
  \yb^{2-k}
  \Big)
\tx{.}
\end{gather}

It remains to inspect $(E^\sk_{k})_{(2)}$.  We do not derive the Fourier series expansion of $(E^\sk_{k})_{(2)}$, but we give an asymptotic expansion with respect to both~$y' \ra \infty$ and~$y' \ra 0$. The normal form for the matrix block~$C$ derived in~\cite{kohnen-1994} guarantees that its bottom-right entry is nonzero, matching the condition $C_{2,2} \ne 0$ in the definition of~$(E^\sk_k)_{(2)}$ in~\eqref{eq:def:skew-harmonic-siegel-eisenstein-subseries}. In particular, we can consider the expansion~(15) of~\cite{kohnen-1994} to deduce the Kohnen limit processes~\eqref{eq:def:kohnen-limit-zero-m-semi-definite-0}, \eqref{eq:def:kohnen-limit-zero-m-semi-definite-1}, and \eqref{eq:def:kohnen-limit-zero-m-indefinite}. Notice that the sum in the formula after~(15) does not include $c = 0$, which has already been subsumed in the expansion of $(E^\sk_{k})_{(1)}$.  In particular, the $0$\thdash\ Fourier series coefficient of the elliptic Eisenstein series $G_{1-k}(\tau', k - \frac{1}{2})$ (in the notation of~\cite{kohnen-1994} as opposed to the notation from Section~\ref{sec:elliptic-modular-forms}) equals $c_1 y^{\prime\, 1-k}$, by using Lemma~\ref{la:elliptic-eisenstein-series-fourier-expansion} with $\kappa \leadsto 1-k$ and $s \leadsto k - \frac{1}{2}$.

Generalizing the formula after (16) of~\cite{kohnen-1994}, we have to study the asymptotic behavior of
\begin{gather}
\label{eq:la:kohnen-limit-zero-m-eisenstein-series-indefinite-expansion}
  c_1\,
  \sum_{\substack{\begin{psmatrix} a & b \\ c & d \end{psmatrix} \in \Gamma_\infty \backslash \SL{2}(\ZZ) \\ l \in \ZZ_{> 0},\, \lambda \in \ZZ \\ (\lambda, l) = 1}}
  (c \tau + d)^{k - 1} |c \tau + d|^{1 - 2k}\,
  \Bigg(
  l^2 y'
  \,+\,
  \Im\Big(
  \lambda^2 \frac{a \tau + b}{c \tau + d}
  +
  \frac{2 l \lambda z}{c \tau + d}
  -
  \frac{c l^2 z^2}{c \tau + d}
  \Big)
  \Bigg)^{1-k}
\tx{.}
\end{gather}
This function is homogeneous of degree~$2 - 2k$ as a rational function of~$l$ and~$\lambda$, and multiplication by $\zeta(2k - 2)$ allows us to remove the condition that $l$ and $\lambda$ are co-prime.

Consider the rightmost expression in the parentheses of~\eqref{eq:la:kohnen-limit-zero-m-eisenstein-series-indefinite-expansion}. Writing $z = \alpha \tau + \beta$ for $\alpha, \beta \in \RR$ and completing the square then give
\begin{multline}
\label{eq:la:kohnen-limit-zero-m-eisenstein-series-indefinite-expansion-inner-expression}
  l^2 y'
  \,+\,
  \Im\Big(
  \lambda^2 \frac{a \tau + b}{c \tau + d}
  +
  \frac{2 l \lambda z}{c \tau + d}
  -
  \frac{c l^2 z^2}{c \tau + d}
  \Big)
\\
=
  l^2 \Big( y' - \frac{y (\alpha d - \beta c)^2 + c \Im(z^2 (c \ov\tau + d))}{|c \tau + d|^2} \Big)
  \,+\,
  \frac{y (\lambda + l (\alpha d - \beta c) )^2}{|c \tau + d|^2}
\tx{.}
\end{multline}
Next, we insert the expression
\begin{gather*}
  \Im(z^2 (c \ov\tau + d))
=
  y \big( \alpha^2 c |\tau|^2 + 2 \alpha^2 d x + 2 \alpha \beta d - \beta^2 c \big)
\end{gather*}
into
\begin{multline*}
  y (\alpha d - \beta c)^2 + c \Im(z^2 (c \ov\tau + d))
\\
=
  y
  \big( \alpha^2 d^2 - 2 \alpha \beta c d + \beta^2 c^2
  \,+\,
  \alpha^2 c^2 (y^2 + x^2)
  +
  2 \alpha^2 c d x
  +
  2 \alpha \beta c d
  -
  \beta^2 c^2
  \big)
\\
=
  \alpha^2 y (d^2 + c^2 y^2 + c^2 x^2 + 2 c d x)
=
  \alpha^2 y |c \tau + d|^2
\tx{,}
\end{multline*}
and we find that~\eqref{eq:la:kohnen-limit-zero-m-eisenstein-series-indefinite-expansion-inner-expression} is equal to
\begin{gather*}
  l^2 \big( y' - \alpha^2 y \big)
  \,+\,
  \frac{y\, (\lambda + l (\alpha d - \beta c) )^2}{|c \tau + d|^2}
=
  \frac{l^2 \det(Y)}{y}
  \,+\,
  \frac{y\, (\lambda + l (\alpha d - \beta c) )^2}{|c \tau + d|^2}
\tx{.}
\end{gather*}
We employ this identity to rewrite~\eqref{eq:la:kohnen-limit-zero-m-eisenstein-series-indefinite-expansion} as follows:
\begin{gather}
\label{eq:la:kohnen-limit-zero-m-eisenstein-series-indefinite-expansion-rewrite}
  \frac{c_1}{\zeta(2k - 2)}\,
  \sum_{\substack{\begin{psmatrix} a & b \\ c & d \end{psmatrix} \in \Gamma_\infty \backslash \SL{2}(\ZZ) \\ l \in \ZZ_{> 0},\, \lambda \in \ZZ}}
  (c \tau + d)^{k - 1} |c \tau + d|^{1 - 2k}\,
  \Big(
  \frac{l^2 \det(Y)}{y}
  \,+\,
  \frac{y\, (\lambda + l (\alpha d - \beta c) )^2}{|c \tau + d|^2}
  \Big)^{1-k}
\tx{.}
\end{gather}

We state and prove Lemma~\ref{la:kohnen-limit-zero-m-eisenstein-series:fourier-transform-lambda} after finishing this proof of Lemma~\ref{la:kohnen-limit-zero-m-eisenstein-series}. Lemma~\ref{la:kohnen-limit-zero-m-eisenstein-series:fourier-transform-lambda} (with $a \leadsto 0$, $b \leadsto 1-k$, $x \leadsto l^2 \det(Y) \slash y$, and $y \leadsto y \slash |c \tau + d|^2$) allows us to apply Poisson summation with respect to~$\lambda$ in~\eqref{eq:la:kohnen-limit-zero-m-eisenstein-series-indefinite-expansion-rewrite} to find that
\begin{multline}
\label{eq:la:kohnen-limit-zero-m-eisenstein-series:transformed-expression-xi-zero}
  \frac{c_1 \pi^{\frac{1}{2}} \Gamma(k - \frac{3}{2})}{\zeta(2k - 2)\, \Gamma(k-1)}\,
  \sum_{\substack{\begin{psmatrix} a & b \\ c & d \end{psmatrix} \in \Gamma_\infty \backslash \SL{2}(\ZZ) \\ l \in \ZZ_{> 0}}} \hspace{-2em}
  (c \tau + d)^{k - 1} |c \tau + d|^{1 - 2k}\,
  \Big( \frac{l^2 \det(Y)}{y} \Big)^{\frac{3}{2}-k}
  \Big( \frac{y}{|c \tau + d|^2} \Big)^{-\frac{1}{2}}
\\
=
  \frac{c_1 \pi^{\frac{1}{2}} \Gamma(k - \frac{3}{2})}{\Gamma(k-1)}\,
  \frac{\zeta(2k-3)}{\zeta(2k - 2)}\,
  \det(Y)^{\frac{3}{2}-k} y^{-1}\,
  \sum_{\ga \in \Gamma_\infty \backslash \SL{2}(\ZZ)} y^{k-1} \big|_{1-k}\,\ga
\tx{,}
\end{multline}
arising from $\xi = 0$, and
\begin{multline}
\label{eq:la:kohnen-limit-zero-m-eisenstein-series:transformed-expression-xi-nonzero}
  \frac{2 \pi i\, c_1}{\zeta(2k-2)}\,
  \sum_{\substack{\begin{psmatrix} a & b \\ c & d \end{psmatrix} \in \Gamma_\infty \backslash \SL{2}(\ZZ) \\ l \in \ZZ_{> 0},\, \xi \in \ZZ \setminus \{0\}}}
  \frac{y^{1-k}\, (c \tau + d)^{k - 1}\, e\big( \xi l (\alpha d - \beta c) \big)}
       {|c \tau + d|\, \exp\big( 2 \pi |c \tau + d| \det(Y)^{\frac{1}{2}} |\xi| l \slash y \big)}
\\
  \sum_{n_2 = 0}^{k-2}
  \frac{(1-k)_{n_2} 2^{1-k-n_2}}{n_2! (k-2-n_2)!}
  \Big( \frac{2 \pi \xi}{i} \Big)^{k-2-n_2}\,
  \Big( \frac{\sgn(\xi) l}{i} \frac{|c \tau + d| \det(Y)^{\frac{1}{2}}}{y} \Big)^{1-k-n_2}
\tx{,}
\end{multline}
arising from $\xi \ne 0$.

We next argue that~\eqref{eq:la:kohnen-limit-zero-m-eisenstein-series:transformed-expression-xi-zero} contributes exclusively to~$e^{\sk\,\pi}_{k,0}$ and~\eqref{eq:la:kohnen-limit-zero-m-eisenstein-series:transformed-expression-xi-nonzero} only to~$e^{\sk\,\indef}$. To this end, we inspect their asymptotic behavior as $y' \ra \infty$ (as opposed to $y' \ra 0$, which appears in Definition~\eqref{eq:def:kohnen-limit-zero-m-indefinite} of the Kohnen limit process).

Fourier terms for $\begin{psmatrix} 0 & 0 \\ 0 & 0 \end{psmatrix}_{(0)}$, for $(\begin{psmatrix} 0 & 0 \\ 0 & 0 \end{psmatrix}_{(1)}, 0_{(0)})$, for $(\begin{psmatrix} 0 & 0 \\ 0 & 0 \end{psmatrix}_{(1)}, 0_{(1)})$, and for semi-definite~$T = \begin{psmatrix} n & 0 \\ 0 & 0 \end{psmatrix}$ with $n \ne 0$ have polynomial asymptotic behavior as $y' \ra \infty$. Their contributions cannot cancel, because they either feature different $n$ or because their leading asymptotics are mutually different. For this reason, we infer that $e^{\sk\,\pi}_{k,0}$ does not receive contributions from~\eqref{eq:la:kohnen-limit-zero-m-eisenstein-series:transformed-expression-xi-nonzero}, since it decays faster than any polynomial as $y' \ra \infty$.

On the other hand, both $a^\sk_{k,0}(\begin{psmatrix} 0 & 0 \\ 0 & 0 \end{psmatrix}_{(1)}, \nb; Y)$ for $\nb \ne 0$ and $a^\sk_k(T Y)$ for indefinite~$T$ decay exponentially in~$\sqrt{y'}$. In the former case this follows from the exponential decay of the $K$-Bessel function that appears in~\eqref{eq:def:fourier-coefficients-skew-zero} and the behavior of $\yb$ with respect to~$y'$ in~\eqref{eq:def:Y-coordinates}. The latter one follows from Shimura's estimates in~\cite{shimura-1982}. This exponential decay cannot be canceled by any of the Fourier series coefficients that contribute to~$e^{\sk\,\pi}_{k,0}$. As a consequence, we find that $e^{\sk\,\indef}_{k,0}$ does not receive contributions from~\eqref{eq:la:kohnen-limit-zero-m-eisenstein-series:transformed-expression-xi-zero}.

We are now in position to establish expressions~\eqref{eq:la:kohnen-limit-zero-m-eisenstein-series-semi-definite-0} and~\eqref{eq:la:kohnen-limit-zero-m-eisenstein-series-semi-definite-1}. We have already determined the Fourier expansions of $(E^\sk_k)_{(0)}$ and $(E^\sk_k)_{(1)}$. Recall that $(E^\sk_k)_{(0)}(Z) = 1$, and hence it contributes~$y^{k-\frac{1}{2}}$ to $\Klim[0,0]^\sk_{k,0}(e^\sk_{k,0})$. We split the contribution of~$(E^\sk_k)_{(1)}$ to~$\Klim[0,0]^\sk_{k,0}(e^\sk_{k,0})$ into two parts. The first summand in~\eqref{eq:fourier-coefficients-siegel-eisenstein-series-(1)-type-skew-first-summand} yields a contribution of
\begin{gather*}
  c_1y^{\frac{1}{2}}=\frac{(-1)^{\frac{1-k}{2}} 2^{2-k} \pi^{\frac{1}{2}} \Gamma(k-1)}
       {\Gamma(k - \frac{1}{2})}
  \frac{\zeta(k-1)}{\zeta(k)}
  y^{\frac{1}{2}}
\end{gather*}
via the Fourier series expansion in~\eqref{eq:fourier-coefficients-siegel-eisenstein-series-(1)-type-skew-first-summand} and Lemma~\ref{la:kohnen-limit-zero-m-semi-definite-fourier-coefficients}. Recall that the second and third summands in~\eqref{eq:fourier-coefficients-siegel-eisenstein-series-(1)-type-skew} only contribute to~$e^\sk_{k,0}$ if~$U = \begin{psmatrix} 1 & 0 \\ 0 & 1 \end{psmatrix}$. Lemma~\ref{la:kohnen-limit-zero-m-semi-definite-fourier-coefficients} shows that their contributions are given by
\begin{multline*}
  \frac{(-1)^{\frac{1-k}{2}} 2^{1-k} \pi^{\frac{1}{2}}}
       {\Gamma(\frac{1}{2})\, \zeta(k)}
  \sum_{n = 1}^\infty
  \frac{\sigma_{k-1}(n)}{n^{k-\frac{1}{2}}}\,
  (4 \pi n y)^{\frac{k-1}{2}}
  W_{\frac{1-k}{2}, \frac{k-1}{2}}(4 \pi n y) e(n x)
\\
  +\,
  \frac{(-1)^{\frac{1-k}{2}} 2^{1-k} \pi^{\frac{1}{2}}}
       {\Gamma(k-\frac{1}{2})\, \zeta(k)}
  \sum_{n = 1}^\infty
  \frac{\sigma_{k-1}(n)}{n^{k-\frac{1}{2}}}\,
  (4 \pi n y)^{\frac{k-1}{2}}
  W_{\frac{k-1}{2}, \frac{k-1}{2}}(4 \pi n y) e(-n x)
\tx{.}
\end{multline*}
Lemma~\ref{la:elliptic-eisenstein-series-fourier-expansion} shows that the contributions to $\Klim[0,0]^\sk_{k,0}(e^\sk_{k,0})$ that we have discovered so-far precisely account for $E_{1-k}\big( k-\tfrac{1}{2}, \tau \big)$. This proves~\eqref{eq:la:kohnen-limit-zero-m-eisenstein-series-semi-definite-0}, since neither~\eqref{eq:la:kohnen-limit-zero-m-eisenstein-series:transformed-expression-xi-zero} nor~\eqref{eq:la:kohnen-limit-zero-m-eisenstein-series:transformed-expression-xi-nonzero} contribute to~$\Klim[0,0]^\sk_{k,0}(e^\sk_{k,0})$. 

We now consider~\eqref{eq:la:kohnen-limit-zero-m-eisenstein-series-semi-definite-1}. Observe that $(E^\sk_k)_{(0)}$ does not contribute to $\Klim[0,1]^\sk_{k,0}(e^\sk_{k,0})$ in~\eqref{eq:la:kohnen-limit-zero-m-eisenstein-series-semi-definite-1}. Recall that the contribution of $(E^\sk_k)_{(1)}$ arises solely from the second term in~\eqref{eq:fourier-coefficients-siegel-eisenstein-series-(1)-type-skew-first-summand}, which (using again Lemma~\ref{la:kohnen-limit-zero-m-semi-definite-fourier-coefficients}) is given by
\begin{gather}
\label{eq:la:kohnen-limit-zero-m-eisenstein-series-semi-definite-1:contribution-(0)}
  \frac{(-1)^{\frac{1-k}{2}} 2^{2-k} \pi^{\frac{1}{2}} \Gamma(k-1)}{\Gamma(k - \frac{1}{2})}\,
  \frac{\zeta(k-1)}{\zeta(k)}\;
  \frac{2^{4-2k} \pi \Gamma(2k-3)}{\Gamma(k-1)^2}\,
  \frac{\zeta(2k-3)}{\zeta(2k-2)}
	\text{.}
\end{gather}
We already argued that~\eqref{eq:la:kohnen-limit-zero-m-eisenstein-series:transformed-expression-xi-nonzero} does not contribute to~$\Klim[0,1]^\sk_{k,0}(e^\sk_{k,0})$. On the other hand, the right-hand side of~\eqref{eq:la:kohnen-limit-zero-m-eisenstein-series:transformed-expression-xi-zero} has polynomial asymptotic behavior of order~$y^{\prime\,\frac{3}{2}-k}$ as~$y' \ra \infty$. Hence it contributes to $\Klim[0,1]^\sk_{k,0}(e^\sk_{k,0})$, and this contribution equals
\begin{gather}
\label{eq:la:kohnen-limit-zero-m-eisenstein-series-semi-definite-1:contribution-(1)}
  \frac{(-1)^{\frac{1-k}{2}} 2^{2-k} \pi^{\frac{1}{2}} \Gamma(k-1)}{\Gamma(k-\frac{1}{2})}\,
  \frac{\zeta(k-1)}{\zeta(k)}\;
  \frac{\pi^{\frac{1}{2}} \Gamma(k - \frac{3}{2})}{\Gamma(k-1)}\,
  \frac{\zeta(2k-3)}{\zeta(2k-2)}\,
  \sum_{\begin{psmatrix} 1 & 0 \\ 0 & 1 \end{psmatrix} \ne \gamma = \begin{psmatrix} a & b \\ c & d \end{psmatrix} \in \Gamma_\infty \backslash \SL{2}(\ZZ)}\hspace{-2em}
  1 \big|_{k-1}\,\ga
\tx{.}
\end{gather}
Now, insert the relation of $\Gamma$-functions
\begin{gather*}
  \Gamma(2k - 3)
=
  2^{2k - 4} \frac{\Gamma(k-1) \Gamma(k-\frac{3}{2})}{\Gamma(\frac{1}{2})}
\end{gather*}
in~\eqref{eq:la:kohnen-limit-zero-m-eisenstein-series-semi-definite-1:contribution-(0)} to recognize that the sum of~\eqref{eq:la:kohnen-limit-zero-m-eisenstein-series-semi-definite-1:contribution-(0)} and~\eqref{eq:la:kohnen-limit-zero-m-eisenstein-series-semi-definite-1:contribution-(1)} matches the right-hand side of~\eqref{eq:la:kohnen-limit-zero-m-eisenstein-series-semi-definite-1}. 

It remains to establish the expression for the indefinite Kohnen limit process in~\eqref{eq:la:kohnen-limit-zero-m-eisenstein-series-indefinite}. Let $\alpha, \beta\in \frac{1}{N}\ZZ$ for $N\geq 1$. If either of the following limits on the left or the right-hand side exists, then the limits are equal, and we have
\begin{multline*}
  \lim_{y' \ra 0}
  y^{\prime\,k-\frac{1}{2}}
  \Bigg( \ov{
  e^{\sk\,\indef}_{k,0}
  \big|^\sk_k\, \up\big(\begin{psmatrix} 1 & 0 \\ 0 & 1 \end{psmatrix}, (\alpha,\beta,0) \big)
  \Big)\, (\tau, 0, \tau')
  } \Bigg)
\\
=
  \lim_{y' \ra \alpha^2 y}
  (y' - \alpha^2 y)^{k-\frac{1}{2}}\,
  \ov{
  e^{\sk\,\indef}_{k,0}(\tau, \alpha \tau + \beta, \tau')
  }
\tx{.}
\end{multline*}
Thus, it suffices to examine the leading asymptotic behavior of~\eqref{eq:fourier-coefficients-siegel-eisenstein-series-(1)-type-skew-first-summand:indefinite-contribution} and~\eqref{eq:la:kohnen-limit-zero-m-eisenstein-series:transformed-expression-xi-nonzero} as~$y' \ra \alpha^2 y$.

We first consider~\eqref{eq:fourier-coefficients-siegel-eisenstein-series-(1)-type-skew-first-summand:indefinite-contribution}. If $y' \ra \alpha^2 y$, then $\taub \ra \alpha$. Therefore, we need to determine the asymptotic expansion of~\eqref{eq:fourier-coefficients-siegel-eisenstein-series-(1)-type-skew-first-summand:indefinite-contribution} at the cusp $\alpha \in \QQ$. Let $\breve\ga \in \SL{2}(\ZZ)$ be a transformation that sends~$\infty$ to~$\alpha$, and write $\breve{c}$ and $\breve{d}$ for the bottom entries of~$\breve{\ga}$. Note that $\breve{c}$ is the denominator of~$\alpha$ and hence does not vanish. Then the modular invariance of the elliptic Eisenstein series of weight~$0$ implies that
\begin{multline}
\label{eq:fourier-coefficients-siegel-eisenstein-series-(1)-type-skew-first-summand:indefinite-contribution-transformed}
  \Big(
  E_0(k-1,\taub)
  \,-\,
  \yb^{k-1}
  \,-\,
  \frac{2^{4-2k} \pi \Gamma(2k-3)}
       {\Gamma(k-1)^2}
  \frac{\zeta(2k-3)}{\zeta(2k-2)}\,
  \yb^{2-k}
  \Big) \Big|_0\, \breve\ga \breve\ga^{-1}
\\
=
  E_0(k-1,\taub) \big|_0\, \breve\ga^{-1}
  \,-\,
  \yb^{k-1}
  \,-\,
  \frac{2^{4-2k} \pi \Gamma(2k-3)}
       {\Gamma(k-1)^2}
  \frac{\zeta(2k-3)}{\zeta(2k-2)}\,
  \yb^{2-k}
\tx{.}
\end{multline}
As~$\taub \ra \alpha$ (i.e.,\@ $\yb \ra 0$), the leading asymptotics of the second and third terms in~\eqref{eq:fourier-coefficients-siegel-eisenstein-series-(1)-type-skew-first-summand:indefinite-contribution-transformed} are given by~$\yb^{k-1}$ and~$\yb^{2-k}$, respectively. The first term in~\eqref{eq:fourier-coefficients-siegel-eisenstein-series-(1)-type-skew-first-summand:indefinite-contribution-transformed} has leading asymptotic $(\breve{\ga}^{-1} \yb)^{k-1}$, which dominates the other aymptotics. When combining this fact with the contribution of $\det(Y)^{\frac{1-k}{2}}$, one finds that~\eqref{eq:fourier-coefficients-siegel-eisenstein-series-(1)-type-skew-first-summand:indefinite-contribution} grows polynomially of order $\frac{2 - k}{2} + \frac{1 - k}{2} = \frac{3}{2} - k$ in $y' - \alpha^2 y$ as~$y' \ra \alpha^2 y$. This growth is compensated by the factor $y^{\prime\,k-\frac{1}{2}}$ in the defining equation~\eqref{eq:def:kohnen-limit-zero-m-indefinite} of the indefinite Kohnen limit process. We conclude that~\eqref{eq:fourier-coefficients-siegel-eisenstein-series-(1)-type-skew-first-summand:indefinite-contribution} does not contribute to the Kohnen limit process $\Klim[N]^\sk_{k,0} ( e^\sk_{k,0} )$.

We next consider the contribution of~\eqref{eq:la:kohnen-limit-zero-m-eisenstein-series:transformed-expression-xi-nonzero} to the indefinite Kohnen limit process. To this end, we rearrange~\eqref{eq:la:kohnen-limit-zero-m-eisenstein-series:transformed-expression-xi-nonzero} as follows:
\begin{multline}
\label{eq:la:kohnen-limit-zero-m-eisenstein-series:transformed-expression-xi-nonzero:polylogarithm}
  \sum_{n_2 = 0}^{k-2}
  \frac{c_1\, (-1)^{n_2+1} 2^{-2n_2} \pi^{k-1-n_2} (1-k)_{n_2}}
       {n_2! (k-2-n_2)!\; \zeta(2k-2)}
\\
  \sum_{\substack{\begin{psmatrix} a & b \\ c & d \end{psmatrix} \in \Gamma_\infty \backslash \SL{2}(\ZZ) \\ l \in \ZZ_{> 0}}}
  \frac{y^{\frac{1-k+n_2}{2}}\, (c \tau + d)^{k - 1}}
       {l^{-1+k+n_2}\, |c \tau + d|^{k+n_2}}
  (y' - \alpha^2 y)^{\frac{1-k-n_2}{2}}\, 
\\
  \sum_{\sgn(\xi) \in \{\pm 1\}}
  \sgn(\xi)
  \sum_{|\xi| = 1}^\infty
  \frac{|\xi|^{k-2-n_2}\, e\big( \sgn(\xi) |\xi| l (\alpha d - \beta c) \big)}
       {\exp\big( 2 \pi |c \tau + d| \det(Y)^{\frac{1}{2}} |\xi| l \slash y \big)}
\tx{.}
\end{multline}
Note that the sum over~$|\xi|$ equals the polylogarithm
\begin{gather}
\label{eq:la:kohnen-limit-zero-m-eisenstein-series:polylogarithm}
  \Li{2-k+n_2}\Bigg(
  \frac{e\big( \sgn(\xi) l (\alpha d - \beta c) \big)}
       {\exp\big( 2 \pi |c \tau + d| \det(Y)^{\frac{1}{2}} l \slash y \big)}
  \Bigg)
\tx{.}
\end{gather}
When $y' \ra \alpha^2 y$, then the argument tends to $e( \sgn(\xi) l (\alpha d - \beta c))$. In particular, if $l (\alpha d - \beta c) \not\in \ZZ$, then \eqref{eq:la:kohnen-limit-zero-m-eisenstein-series:polylogarithm} converges. If $l (\alpha d - \beta c) \in \ZZ$, then $e\big( \sgn(\xi) l (\alpha d - \beta c) \big) = 1$ and the asymptotic growth of~\eqref{eq:la:kohnen-limit-zero-m-eisenstein-series:polylogarithm} as $y' \ra \alpha^2 y$ is given by
\begin{align*}
&
  \Big( 1 - \exp\big( - 2 \pi |c \tau + d| \det(Y)^{\frac{1}{2}} l \slash y \big) \Big)^{k-1-n_2}\;
  \Li{2-k+n_2}\Big(
  \exp\big( - 2 \pi |c \tau + d| \det(Y)^{\frac{1}{2}} l \slash y \big)
  \Big)
\\&\hspace{22em}
  \cdot\,
  \Big( 1 - \exp\big( - 2 \pi |c \tau + d| \det(Y)^{\frac{1}{2}} l \slash y \big) \Big)^{1-k+n_2}
\\
={}&
  \big( (1 - h)^{k-1-n_2} \Li{2-k+n_2}(h) \big)_{h = 1}\;
  (-1)^{k-1-n_2}\,
  \Big( \exp\big( - 2 \pi |c \tau + d| \det(Y)^{\frac{1}{2}} l \slash y \big) - 1 \Big)^{1-k+n_2}
\\
\asymp{}&
   (2-k+n_2)!\;
  \big( 2 \pi |c \tau + d| (y' - \alpha^2 y)^{\frac{1}{2}} l \slash y^{\frac{1}{2}} \big)^{1-k+n_2}
\tx{.}
\end{align*}
This is independent of~$\sgn(\xi)$ and hence the sum over $\sgn(\xi)$ in~\eqref{eq:la:kohnen-limit-zero-m-eisenstein-series:transformed-expression-xi-nonzero:polylogarithm} cancels. Consider the case $l (\alpha d - \beta c) \not\in \ZZ$, in which case the polylogarithm in~\eqref{eq:la:kohnen-limit-zero-m-eisenstein-series:polylogarithm} is regular at $y' = \alpha^2 y$, and takes the value $\Li{2-k+n_2}(e( \sgn(\xi) l (\alpha d - \beta c) ))$. When inserting this into~\eqref{eq:la:kohnen-limit-zero-m-eisenstein-series:transformed-expression-xi-nonzero:polylogarithm}, we find that its leading asymptotic behavior as $y' \ra \alpha^2 y$ arises from $n_2 = k-2$, and equals
\begin{multline}
\label{eq:la:kohnen-limit-zero-m-eisenstein-series:transformed-expression-xi-nonzero:asymptotic}
  \big( y' - \alpha^2 y \big)^{\frac{3}{2}-k}\;
  \frac{c_1\, (-1)^{k-1} 2^{4-2k} \pi (1-k)_{k-2}}
       {(k-2)!\, \zeta(2k-2)}
\\
  \sum_{\substack{\begin{psmatrix} a & b \\ c & d \end{psmatrix} \in \Gamma_\infty \backslash \SL{2}(\ZZ) \\ l \in \ZZ_{> 0} \\ l (\alpha d - \beta c) \not\in \ZZ}}
  \frac{y^{-\frac{1}{2}}\, (c \tau + d)^{k-1}}
       {l^{2k-3}\, |c \tau + d|^{2k-2}}
  \sum_{\sgn(\xi) \in \{\pm 1\}}
  \sgn(\xi)\, \Li{0}\big(e\big( \sgn(\xi) l (\alpha d - \beta c) \big)\big)
\tx{.}
\end{multline}
This already implies convergence of the Kohnen limit process in~\eqref{eq:la:kohnen-limit-zero-m-eisenstein-series-indefinite}, since $y'^{k-\frac{3}{2}}$ in~\eqref{eq:def:kohnen-limit-zero-m-indefinite} is compensated by the corresponding power of~$(y' - \alpha^2 y)$ in~\eqref{eq:la:kohnen-limit-zero-m-eisenstein-series:transformed-expression-xi-nonzero:asymptotic}.

Observe that $\Li{0}(h) = h / (1 - h)$ for all $h \in \CC \setminus \{1\}$.  We now use the relation
\begin{gather*}
  \frac{h}{1 - h} - \frac{h^{-1}}{1 - h^{-1}}
=
  \frac{2}{1 - h} - 1
\end{gather*}
to rewrite the sum over $\sgn(\xi)$ in~\eqref{eq:la:kohnen-limit-zero-m-eisenstein-series:transformed-expression-xi-nonzero:asymptotic}, which allows us to replace~\eqref{eq:la:kohnen-limit-zero-m-eisenstein-series:transformed-expression-xi-nonzero:asymptotic} by
\begin{multline}
\label{eq:la:kohnen-limit-zero-m-eisenstein-series:transformed-expression-xi-nonzero:asymptotic-removed-xi}
  (y' - \alpha^2 y)^{\frac{3}{2}-k}\;
  \frac{c_1\, (-1)^{k-1} 2^{4-2k} \pi (1-k)_{k-2}}
       {(k-2)!\, \zeta(2k-2)}
\\
  \sum_{\substack{\begin{psmatrix} a & b \\ c & d \end{psmatrix} \in \Gamma_\infty \backslash \SL{2}(\ZZ) \\ l \in \ZZ_{> 0} \\ l (\alpha d - \beta c) \not\in \ZZ}}
  \Big( \frac{2}{1 - e\big( l (\alpha d - \beta c) \big)} - 1 \Big)
  \frac{y^{-\frac{1}{2}}\, (c \tau + d)^{k-1}}
       {l^{2k-3}\, |c \tau + d|^{2k-2}}
\tx{.}
\end{multline}
We next insert~$c_1$ and express the sum over~$\begin{psmatrix} a & b \\ c & d \end{psmatrix}$ in terms of an elliptic slash action.
\begin{multline}
\label{eq:la:kohnen-limit-zero-m-eisenstein-series:transformed-expression-xi-nonzero:asymptotic-removed-xi-elliptic-slash}
  y^{\frac{1}{2}-k}
  (y' - \alpha^2 y)^{\frac{3}{2}-k}\;
  \frac{(-1)^{\frac{-1-k}{2}} 2^{6-3k} \pi^{\frac{3}{2}}\, \Gamma(k)}{\Gamma(k-\frac{1}{2})}\,
  \frac{\zeta(k-1)}{\zeta(k) \zeta(2k-2)}
\\
  \sum_{\substack{\begin{psmatrix} a & b \\ c & d \end{psmatrix} \in \Gamma_\infty \backslash \SL{2}(\ZZ) \\ l \in \ZZ_{> 0} \\ l (\alpha d - \beta c) \not\in \ZZ}}
  \Big( \frac{2}{1 - e\big( l (\alpha d - \beta c) \big)} - 1 \Big)
  \frac{y^{k-1}}{l^{2k-3}} \big|_{1-k}\,
  \begin{psmatrix} a & b \\ c & d \end{psmatrix}
\end{multline}

Recall that $\alpha, \beta \in \frac{1}{N} \ZZ$. In particular, the condition $l (\alpha d - \beta c) \not\in \ZZ$ depends only on~$c$ and~$d$ mod~$N$. This allows us to express the sum in~\eqref{eq:la:kohnen-limit-zero-m-eisenstein-series:transformed-expression-xi-nonzero:asymptotic-removed-xi-elliptic-slash} in terms of the following Eisenstein series (for $s \in \CC$, $\kappa \in \ZZ$, and $2 \Re(s) + \kappa > 2$)
\begin{gather*}
   E_{\kappa,\Gamma_1(N)}(s,\tau)
:=
  \sum_{\ga \in \Gamma_\infty \backslash \Gamma_1(N)} y^s \big|_\kappa\, \ga
\text{,}
\end{gather*}
where
\begin{gather*}
  \Gamma_1(N)
:=
  \big\{
  \begin{psmatrix} a & b \\ c & d \end{psmatrix} \in \SL{2}(\ZZ) \,:\,
  c \equiv 0 \pmod{N}, a \equiv d \equiv 1 \pmod{N}
  \big\}
\tx{.}
\end{gather*}
Indeed, the value of $c,d \pmod{N}$ for~$\ga = \begin{psmatrix} a & b \\ c & d \end{psmatrix} \in \SL{2}(\ZZ)$ is determined by the associated coset $\Gamma_1(N) \ga$:
\begin{multline*}
  \sum_{\substack{\begin{psmatrix} a & b \\ c & d \end{psmatrix} \in \Gamma_\infty \backslash \SL{2}(\ZZ) \\ l \in \ZZ_{> 0} \\ l (\alpha d - \beta c) \not\in \ZZ}}
  \Big( \frac{2}{1 - e\big( l (\alpha d - \beta c) \big)} - 1 \Big)
  \frac{y^{k-1}}{l^{2k-3}} \big|_{1-k}\,\ga
\\
=
  \sum_{\substack{c,d \in \ZZ \slash N \ZZ \\ \gcd(c,d,N) = 1}}
   E_{1-k,\Gamma_1(N)}(k-1,\tau)
  \big|_{1-k}\,\begin{psmatrix} \ast & \ast \\ c & d \end{psmatrix}\,
  \sum_{\substack{l \in \ZZ_{>0} \\ l (\alpha d - \beta c) \not\in \ZZ}}
  \Big( \frac{2}{1 - e\big( \ov{l} (\alpha d - \beta c) \big)} - 1 \Big)
  \frac{1}{l^{2k-3}}
\tx{,}
\end{multline*}
where $\begin{psmatrix} \ast & \ast \\ c & d \end{psmatrix}$ is any matrix in~$\SL{2}(\ZZ)$ whose bottom row is congruent to~$\begin{psmatrix} c & d \end{psmatrix}$ mod~$N$.
We can express the remaining sum over~$l$ in terms of the Hurwitz~$\zeta$-function as follows:
\begin{gather*}
  \sum_{\substack{l \in \ZZ_{>0} \\ l (\alpha d - \beta c) \not\in \ZZ}}
  \Big( \frac{2}{1 - e\big( \ov{l} (\alpha d - \beta c) \big)} - 1 \Big)
  \frac{1}{l^{2k-3}}
=
  \sum_{\substack{\ov{l} = 1 \\ \ov{l} (\alpha d - \beta c) \not\in \ZZ}}^N
  \Big( \frac{2}{1 - e\big( \ov{l} (\alpha d - \beta c) \big)} - 1 \Big)
  \frac{\zeta\big( 2k-3, \frac{\ov{l}}{N} \big)}{N^{2k - 3}}
	\text{.}
\end{gather*}
Since $E_{1-k, \Gamma_1(N)}(k-1, \tau) |_{1-k} \begin{psmatrix} \ast & \ast \\ c & d \end{psmatrix}$ equals the $(c,d)$\thdash\ component of the vector-valued Eisenstein series $E_{1-k,N}(k-1,\tau) = y^{k-1} \ov{E_{k-1,N}(\tau)}$, this completes the proof of Lemma~\ref{la:kohnen-limit-zero-m-eisenstein-series}.
\end{proof}

We state the next lemma in greater generality than required for the proof of Lemma~\ref{la:kohnen-limit-zero-m-eisenstein-series}, since we will need it also in the third part of this series of papers.
\begin{lemma}
\label{la:kohnen-limit-zero-m-eisenstein-series:fourier-transform-lambda}
Let $a, b \in \ZZ$ such that $a \ge 0$ and $a + 2b \le 0$, and assume that $x, y > 0$ are real. Then
\begin{gather*}
  \lambda^a (x + y \lambda^2)^b
\end{gather*}
is square-integrable with respect to $\lambda \in \RR$, and its Fourier transform $\cF\big[ \lambda^a (x + y \lambda^2)^b \big](\xi)$ equals
\begin{gather*}
  \frac{2 \pi i y^b}{\exp\big( 2 \pi \sqrt{\tfrac{x}{y}}\, |\xi| \big)}
  \sum_{\substack{n_1, n_2, m \ge 0 \\ n_1 + n_2 + m = -1 - b}} \!\!\!
  \frac{a_{n_1} b_{n_2} 2^{b - n_2}}{n_1! n_2! m!}
  \Big( \frac{2 \pi \xi}{i} \Big)^m\,
  \Big( \frac{\sgn(\xi)}{i} \sqrt{\tfrac{x}{y}} \Big)^{a + b - n_1 - n_2}
\tx{,}\quad
  \tx{if $\xi \ne 0$,}
\end{gather*}
where $a_{n_1}$, etc.\ are Pochhammer symbols; and
\begin{gather*}
  \cF\big[
  \lambda^a (x + y \lambda^2)^b
  \big](0)
\;=\;
  \frac{(1 + (-1)^a) \Gamma(\frac{1 + a}{2}) \Gamma(\frac{-1-a-2b}{2})}{2 \Gamma(-b)}
  x^{\frac{1 + a}{2} + b} y^{\frac{-1-a}{2}}
\tx{.}
\end{gather*}
\end{lemma}
\begin{proof}
The case of $\xi = 0$ is a straightforward computation. If $\xi \ne 0$, say $\xi > 0$ write the integral
\begin{gather*}
  \int_\RR \lambda^a \big(x + y \lambda^2 \big)^b\, e(\lambda \xi) \, d\!\lambda
\end{gather*}
as a sum of a two path integrals. The first path consists of lines from $-\infty$ to~$-B$ and from $B$ to $\infty$ and an arc in the lower half plane with baseline from~$-B$ to~$B$. The second one is the corresponding semi-circle with baseline from~$-B$ to~$B$. As $B \ra \infty$, the first path integral can be estimated directly and the second one can be computed using the residue theorem.
\end{proof}

\section{The Maa\ss\ lift of harmonic Maa\ss-Jacobi forms}
\label{sec:maass-lift}

We have determined most Fourier series coefficients of~$E^\sk_k$. In Section~\ref{ssec:skew-eisenstein-series-fourier-expansion}, we use the Kohnen limit processes to examine the remaining parts of its Fourier series expansion.  This naturally leads to a definition of a Hecke-like operator~$V_0$ (defined in Section~\ref{ssec:hecke-operator-V-0}), which maps an abstract Fourier-Jacobi term of index~$1$ to an abstract elliptic function. The image under~$V_0$ essentially captures the $0$\thdash\ Fourier series coefficient of its preimage and its Fourier coefficients of isotropic index.  
Finally, in Section~\ref{ssec:proof-of-skew-maass-lift} we present the proof of Theorem~\ref{mainthm:skew-maass-lift}.

\subsection{The Fourier expansion of skew Eisenstein series}
\label{ssec:skew-eisenstein-series-fourier-expansion}

In Section~\ref{sec:kohnen-limit-nonzero-m}, we explored the Kohnen limit process on Fourier-Jacobi coefficients~$e^\sk_{k,m}$ of~$E^\sk_k$ for~$m \ne 0$. Injectivity of~$\Klim^\sk_{k,m}$ for $m \ne 0$ applied to Fourier coefficients allowed us to recover the Fourier coefficients of $E^\sk_k$ of index~$T = \begin{psmatrix} n & r \slash 2 \\ r \slash 2 & m \end{psmatrix}$, if $m \ne 0$. In this section, we complete the study of Fourier coefficients of~$E^\sk_k$, by inspecting its $0$\thdash\ Fourier-Jacobi coefficient.

Propositions~\ref{prop:fourier-jacobi-coefficient-zero-m-eisenstein-series-semi-definite-horizontal} and~\ref{prop:fourier-jacobi-coefficient-zero-m-eisenstein-series-semi-definite-vertial} use the Kohnen limit process to derive the semi-definite part of the Fourier series expansion of~$E^\sk_k$. The indefinite part is then exhibited in Proposition~\ref{prop:fourier-jacobi-coefficient-zero-m-eisenstein-series-indefinite}.

\begin{proposition}
\label{prop:fourier-jacobi-coefficient-zero-m-eisenstein-series-semi-definite-horizontal}
We have the following Fourier series expansion of\/~$\Klim[0]^{\sk\,\pi}_{k,0}\big( e^\sk_{k,0}(Z) \big)$:
\begin{align*}
&
  a^\sk_k\big(\begin{psmatrix} 0 & 0 \\ 0 & 0 \end{psmatrix}_{(0)}; Z \big)
  \,+\,
  \frac{(-1)^{\frac{1-k}{2}} 2^{2-k} \pi^{\frac{1}{2}} \Gamma(k - 1)}
       {\Gamma(k - \frac{1}{2})}\,
  \frac{\zeta(k - 1)}{\zeta(k)}\,
  a^\sk_k\big(\begin{psmatrix} 0 & 0 \\ 0 & 0 \end{psmatrix}_{(1)}, 0_{(0)}; Z \big)
\\+\;{}&
  \frac{(-1)^{\frac{1-k}{2}} (2 \pi)^k}{\Gamma(\frac{1}{2}) \zeta(k)}\,
  \sum_{n = 1}^\infty
  \sigma_{k-1}(n)
  a^\sk_k\big(\begin{psmatrix} n & 0 \\ 0 & 0 \end{psmatrix} Z \big)
  \,+\,
  \frac{(-1)^{\frac{1-k}{2}} (2 \pi)^k}{\Gamma(k-\frac{1}{2}) \zeta(k)}
  \sum_{n = 1}^\infty
  \sigma_{k-1}(n)\,
  a^\sk_k\big(\begin{psmatrix} -n & 0 \\ 0 & 0 \end{psmatrix}_{(0)} Z \big)
\\[.5\baselineskip]
+\;{}&
  \frac{(-1)^{\frac{k-1}{2}} 2^{2-k} \pi\, \Gamma(k-\frac{3}{2})}
       {\Gamma(k - \frac{1}{2})}\,
  \frac{\zeta(k-1) \zeta(2k-3)}{\zeta(k) \zeta(2k-2)}\,
  a^\sk_k\big(\begin{psmatrix} 0 & 0 \\ 0 & 0 \end{psmatrix}_{(1)}, 0_{(1)}; Z \big)
\\+\;{}&
  \frac{2^{5-2k} \pi^2\, \Gamma(k-\frac{3}{2})}
       {\Gamma(k-1) \Gamma(k - \frac{1}{2})}\,
  \frac{\zeta(2k-3)}{\zeta(k) \zeta(2k-2)}\,
  \sum_{n = 1}^\infty
  \sigma_{2-k}(n)\,
  a^\sk_k\big(\begin{psmatrix} -n & 0 \\ 0 & 0 \end{psmatrix}_{(1)} Z \big)
\tx{.}
\end{align*}
\end{proposition}
\begin{remark}
In light of the technical complications that we have encountered when computing the Kohnen limit process for Eisenstein series, we decided to provide two independent proofs of Proposition~\ref{prop:fourier-jacobi-coefficient-zero-m-eisenstein-series-semi-definite-horizontal} as a cross-check. The first proof employs the Kohnen limit process as determined in Lemma~\ref{la:kohnen-limit-zero-m-eisenstein-series}. The second one avoids the use of Lemma~\ref{la:kohnen-limit-zero-m-eisenstein-series} and instead recurses to the Kohnen limit processes of nonzero index.
\end{remark}

\begin{proof}[First proof of Proposition~\ref{prop:fourier-jacobi-coefficient-zero-m-eisenstein-series-semi-definite-horizontal}]
Consider the Fourier expansion of~$\Klim[0]^{\sk\,\pi}_{k,0} \big( e^\sk_{k,0}(Z) \big)$ in its general form:
\begin{align*}
&
  c \big( E^\sk_k;\begin{psmatrix} 0 & 0 \\ 0 & 0 \end{psmatrix}_{(0)} \big)
  a^\sk_k\big(\begin{psmatrix} 0 & 0 \\ 0 & 0 \end{psmatrix}_{(0)}; Z \big)
  \;+\;
  c \big( E^\sk_k;\begin{psmatrix} 0 & 0 \\ 0 & 0 \end{psmatrix}_{(1)}, 0_{(0)} \big)
  a^\sk_k\big(\begin{psmatrix} 0 & 0 \\ 0 & 0 \end{psmatrix}_{(1)}, 0_{(0)}; Z \big)
\\+\;{}&
  \sum_{n = 1}^\infty
  c \big( E^\sk_k;\begin{psmatrix} n & 0 \\ 0 & 0 \end{psmatrix} \big)
  a^\sk_k\big(\begin{psmatrix} n & 0 \\ 0 & 0 \end{psmatrix} Z \big)
  \;+\;
  \sum_{n = 1}^\infty
  c \big( E^\sk_k;\begin{psmatrix} -n & 0 \\ 0 & 0 \end{psmatrix}_{(0)} \big)
  a^\sk_k\big(\begin{psmatrix} -n & 0 \\ 0 & 0 \end{psmatrix}_{(0)} Z \big)
\\[.5\baselineskip]
+\;{}&
  c \big( E^\sk_k;\begin{psmatrix} 0 & 0 \\ 0 & 0 \end{psmatrix}_{(1)}, 0_{(1)} \big)
  a^\sk_k\big(\begin{psmatrix} 0 & 0 \\ 0 & 0 \end{psmatrix}_{(1)}, 0_{(1)}; Z \big)
  \;+\;
  \sum_{n = 1}^\infty
  c \big( E^\sk_k;\begin{psmatrix} -n & 0 \\ 0 & 0 \end{psmatrix}_{(1)} \big)
  a^\sk_k\big(\begin{psmatrix} -n & 0 \\ 0 & 0 \end{psmatrix}_{(1)} Z \big)
\tx{.}
\end{align*}
Recall from Lemma~\ref{la:kohnen-limit-zero-m-semi-definite-fourier-coefficients} the explicit expressions for the Kohnen limit processes applied to Fourier series coefficients. We obtain
\begin{align*}
&
  \Klim[0,0]^{\sk\,\pi}_{k,0} \big( e^\sk_{k,0}(Z) \big)
=
  \Klim[0,0]^{\sk\,\pi}_{k,0} \Big(
  \Klim[0]^{\sk\,\pi}_{k,0} \big( e^\sk_{k,0}(Z) \big)
  \Big)
=
\\&\qquad
\hphantom{+\;}
  c \big( E^\sk_k;\begin{psmatrix} 0 & 0 \\ 0 & 0 \end{psmatrix}_{(0)} \big)\,
  y^{k-\frac{1}{2}}
  \;+\;
  c \big( E^\sk_k;\begin{psmatrix} 0 & 0 \\ 0 & 0 \end{psmatrix}_{(1)}, 0_{(0)} \big)
  y^{\frac{1}{2}}
\\&\qquad
+\;
  \sum_{n = 1}^\infty
  ( 4 \pi n )^{\frac{1}{2}-k}\,
  c \big( E^\sk_k;\begin{psmatrix} n & 0 \\ 0 & 0 \end{psmatrix} \big)\,
  ( 4 \pi n y )^{\frac{k-1}{2}}
  W_{\frac{1-k}{2}, \frac{k-1}{2}}\big( 4 \pi n y \big) e(n x)
\\&\qquad
+\;
  \sum_{n = 1}^\infty
  ( 4 \pi n )^{\frac{1}{2}-k}\,
  c \big( E^\sk_k;\begin{psmatrix} -n & 0 \\ 0 & 0 \end{psmatrix}_{(0)} \big)\,
  ( 4 \pi n y )^{\frac{k-1}{2}}
  W_{\frac{k-1}{2}, \frac{k-1}{2}}\big( 4 \pi n y \big) e(-n x)
\tx{,}
\\&
  \Klim[0,1]^{\sk\,\pi}_{k,0} \big( e^\sk_{k,0}(Z) \big)
=
  \Klim[0,1]^{\sk\,\pi}_{k,0} \Big(
  \Klim[0]^{\sk\,\pi}_{k,0} \big( e^\sk_{k,0}(Z) \big)
  \Big)
=
\\&\qquad
\hphantom{+\;}
  \ov{c \big( E^\sk_k;\begin{psmatrix} 0 & 0 \\ 0 & 0 \end{psmatrix}_{(1)}, 0_{(1)} \big)}
  \;+\;
  \sum_{n = 1}^\infty
  (4 \pi n)^{k-2}
  \ov{c \big( E^\sk_k;\begin{psmatrix} -n & 0 \\ 0 & 0 \end{psmatrix}_{(1)} \big)}\,
  e(n \tau)
\tx{.}
\end{align*}

Lemma~\ref{la:kohnen-limit-zero-m-eisenstein-series} expresses $\Klim[0,0]^\sk_{k,0} \big( e^\sk_{k,0} \big) (\tau)$ and~$\Klim[0,1]^\sk_{k,0} \big( e^\sk_{k,0} \big) (\tau)$ in terms of elliptic Eisenstein series. Their Fourier series expansions are given by Lemma~\ref{la:elliptic-eisenstein-series-fourier-expansion} with $\kappa \leadsto 1 - k$, $s \leadsto k - \frac{1}{2}$ and $\kappa \leadsto k - 1$, $s \leadsto 0$, respectively. Comparing these Fourier series coefficients with the above ones implies the statement of Proposition~\ref{prop:fourier-jacobi-coefficient-zero-m-eisenstein-series-semi-definite-horizontal}.
\end{proof}
\begin{proof}[Second proof of Proposition~\ref{prop:fourier-jacobi-coefficient-zero-m-eisenstein-series-semi-definite-horizontal}]
We start with the contributions to Fourier series coefficients of positive semi-definite index. If $m>0$, then Lemma~\ref{la:kohnen-limit-positive-m-eisenstein-series} asserts that 
\begin{gather*}
  \Klim^\sk_{k,m}\big( e^\sk_{k,m} \big)
=
  \frac{(-1)^{\frac{1-k}{2}} (2 \pi)^k}{\Gamma(\frac{1}{2}) \zeta(k)}\,
  E^\Jsk_{k,1} \big|^\Jsk_{k,1}\, V_m
\tx{.}
\end{gather*}
Note that the constant coefficient of~$E^\Jsk_{k,1}$ equals~$1$, and Lemma~\ref{la:kohnen-limit-positive-m-semi-definite-fourier-coefficients} implies that
\begin{gather*}
  \Klim^\sk_{k,m}\Big(
  c\big( E^\sk_k; \begin{psmatrix} 0 & 0 \\ 0 & m \end{psmatrix} \big)\,
  a^\sk_k\big(\begin{psmatrix} 0 & 0 \\ 0 & m \end{psmatrix} Z \big)
  \Big)
=
  c\big( E^\sk_k; \begin{psmatrix} 0 & 0 \\ 0 & m \end{psmatrix} \big)
\text{.}
\end{gather*}
Thus, the explicit formula for the action of $V_m$ on Fourier series coefficients in~\eqref{eq:V-hecke-operator-coefficient-formula} yields that
\begin{gather}
  c\big( E^\sk_k; \begin{psmatrix} 0 & 0 \\ 0 & m \end{psmatrix} \big)
=
  \frac{(-1)^{\frac{1-k}{2}} (2 \pi)^k}{\Gamma(\frac{1}{2}) \zeta(k)}\,
  \sigma_{k-1}(m)
\tx{.}
\end{gather}
Observe that $E^\sk_k$ is invariant under the slash action of $\rot(\begin{psmatrix} 0 & 1 \\ 1 & 0 \end{psmatrix})$, which allows us to confirm the contribution to the Fourier series coefficients of index~$\begin{psmatrix} m & 0 \\ 0 & 0 \end{psmatrix}$ with~$m > 0$:
\begin{gather}
\label{eq:prop:fourier-jacobi-coefficient-zero-m-eisenstein-series-semi-definite-horizontal:second-proof:fourier-coefficient-semi-definite-positive}
  c\big( E^\sk_k; \begin{psmatrix} m & 0 \\ 0 & 0 \end{psmatrix} \big)
=
  (-1)^{k-1}
  c\big( E^\sk_k; \begin{psmatrix} 0 & 0 \\ 0 & m \end{psmatrix} \big)
=
  \frac{(-1)^{\frac{k-1}{2}} (2 \pi)^k}{\Gamma(\frac{1}{2}) \zeta(k)}\,
  \sigma_{k-1}(m)
\tx{.}
\end{gather}
Since $k$ is odd, i.e., $(-1)^{k-1} = 1$, this contribution coincides with the corresponding coefficient in Proposition~\ref{prop:fourier-jacobi-coefficient-zero-m-eisenstein-series-semi-definite-horizontal}.

We employ a similar argument to determine Fourier series coefficients of negative semi-de\-fi\-nite, degenerate index, but we must take the multiplicity-two of such coefficients into account. Specifically, if $m<0$, then Lemma~\ref{la:kohnen-limit-negative-m-semi-definite-fourier-coefficients} shows that
\begin{gather}
\label{eq:prop:fourier-jacobi-coefficient-zero-m-eisenstein-series-semi-definite-horizontal:second-proof:klim-negative-m}
\begin{aligned}
  \Klim^\sk_{k,m}\Big(
  c\big( E^\sk_k; \begin{psmatrix} 0 & 0 \\ 0 & m \end{psmatrix}_{(0)} \big)\,
  a^\sk_k\big(\begin{psmatrix} 0 & 0 \\ 0 & m \end{psmatrix}_{(0)} Z \big)
  \Big)
&=
  \ov{c\big( E^\sk_k; \begin{psmatrix} 0 & 0 \\ 0 & m \end{psmatrix}_{(0)} \big)}\,
  y^{k-1}
\tx{,}
\\
  \Klim^\sk_{k,m}\Big(
  c\big( E^\sk_k; \begin{psmatrix} 0 & 0 \\ 0 & m \end{psmatrix}_{(1)} \big)\,
  a^\sk_k\big(\begin{psmatrix} 0 & 0 \\ 0 & m \end{psmatrix}_{(1)} Z \big)
  \Big)
&=
  (4 \pi |m|)^{k-\frac{3}{2}}
  \ov{c\big( E^\sk_k; \begin{psmatrix} 0 & 0 \\ 0 & m \end{psmatrix}_{(1)} \big)}\,
  y^{\frac{1}{2}}
\tx{.}
\end{aligned}
\end{gather}
The Kohnen limit process on Fourier-Jacobi coefficients of~$E^\sk_k$ for negative index~$m < 0$ is given in Lemma~\ref{la:kohnen-limit-negative-m-eisenstein-series}:
\begin{gather*}
  \Klim^\sk_{k,m}\big( e^\sk_{k,m} \big)
=
  |m|^{k-1}\,
  \frac{(-1)^{\frac{k-1}{2}} (2 \pi)^k}{\Gamma(k-\frac{1}{2}) \zeta(k)}\,
  E^\cJsk_{2-k,1} \big|^\Jsk_{2-k, 1}\, V_{|m|}
\tx{.}
\end{gather*}
The constant coefficient of~$E^\cJsk_{2-k,1}$ is given in Lemma~\ref{la:jacobi-eisenstein-series-zeroth-fourier-coefficient}. We combine that coefficient with the formula for the action of~$V_{|m|}$ on powers of~$y$ in Lemma~\ref{eq:V-hecke-operator-real-analytic-coefficient-formula} to find that
\begin{multline*}
  c\big( \Klim^\sk_{k,m}\big( e^\sk_{k,m} \big);\, 0, 0;\, y \big)
\,=\,
  \frac{(-1)^{\frac{k-1}{2}} (2 \pi)^k}{\Gamma(k-\frac{1}{2}) \zeta(k)}\,
  \Big(
  \sigma_{k-1}(|m|)\,
  y^{k-1}
\\
  +\,
  \frac{(-1)^{\frac{k-1}{2}} 2^{2-k} \pi^{\frac{1}{2}}\, \Gamma(k-\frac{3}{2})}{\Gamma(k-1)}\,
  \frac{\zeta(2k-3)}{\zeta(2k-2)}\,
  |m|^{k-\frac{3}{2}} \sigma_{2-k}(|m|)\,
  y^{\frac{1}{2}}
  \Big)
\tx{.}
\end{multline*}
Therefore, using~\eqref{eq:prop:fourier-jacobi-coefficient-zero-m-eisenstein-series-semi-definite-horizontal:second-proof:klim-negative-m}, we have 
\begin{gather}
\label{eq:prop:fourier-jacobi-coefficient-zero-m-eisenstein-series-semi-definite-horizontal:second-proof:fourier-coefficient-semi-definite-negative}
\begin{aligned}
  c\big( E^\sk_k; \begin{psmatrix} 0 & 0 \\ 0 & m \end{psmatrix}_{(0)} \big)\,
&=
  \frac{(-1)^{\frac{1-k}{2}} (2 \pi)^k}{\Gamma(k-\frac{1}{2}) \zeta(k)}\,
  \sigma_{k-1}(|m|)
\tx{,}
\\
  c\big( E^\sk_k; \begin{psmatrix} 0 & 0 \\ 0 & m \end{psmatrix}_{(1)} \big)\,
&=
  \frac{(-1)^{k-1} 2^{5-2k} \pi^2\, \Gamma(k-\frac{3}{2})}
       {\Gamma(k-1) \Gamma(k-\frac{1}{2})}\,
  \frac{\zeta(2k-3)}{\zeta(k) \zeta(2k-2)}\,
  \sigma_{2-k}(|m|)
\tx{.}
\end{aligned}
\end{gather}
Considering the action of $\rot(\begin{psmatrix} 0 & 1 \\ 1 & 0 \end{psmatrix})$ as before, we obtain the Fourier series coefficient of~$e^\sk_{k,0}$ of negative semi-definite, degenerate index.

It remains to determine the Fourier series coefficients of index~$\begin{psmatrix} 0 & 0 \\ 0 & 0 \end{psmatrix}_{(0)}$, $(\begin{psmatrix} 0 & 0 \\ 0 & 0 \end{psmatrix}_{(1)}, 0_{(0)})$, and $(\begin{psmatrix} 0 & 0 \\ 0 & 0 \end{psmatrix}_{(1)}, 0_{(1)})$. To this end, we employ the modularity of the image of the semi-definite Kohnen limit process for index~$m = 0$, as stated in Proposition~\ref{prop:kohnen-limit-zero-m-semi-definite-convergence}.  

The next computation is similar to one we did in the first proof of Proposition~\ref{prop:fourier-jacobi-coefficient-zero-m-eisenstein-series-semi-definite-horizontal}.  However, we will not use the expressions for $\Klim[0,0]^\sk_{k,0}(e^\sk_{k,0})$ and $\Klim[0,1]^\sk_{k,0}(e^\sk_{k,0})$ from Lemma~\ref{la:kohnen-limit-zero-m-eisenstein-series}. Lemma~\ref{la:kohnen-limit-zero-m-semi-definite-fourier-coefficients} provides explicit formulas for the Kohnen limit process applied to Fourier series. Together with formulas~\eqref{eq:prop:fourier-jacobi-coefficient-zero-m-eisenstein-series-semi-definite-horizontal:second-proof:fourier-coefficient-semi-definite-positive} and~\eqref{eq:prop:fourier-jacobi-coefficient-zero-m-eisenstein-series-semi-definite-horizontal:second-proof:fourier-coefficient-semi-definite-negative}, it yields the following expression for the sum of $\Klim[0,0]^\sk_{k,0}(e^\sk_{k,0})$ and $\Klim[0,1]^\sk_{k,0}(e^\sk_{k,0})$:
\begin{align*}
&
  c\big( E^\sk_k; \begin{psmatrix} 0 & 0 \\ 0 & 0 \end{psmatrix}_{(0)} \big)\,
  y^{k-\frac{1}{2}}
  +
  c\big( E^\sk_k; \begin{psmatrix} 0 & 0 \\ 0 & 0 \end{psmatrix}_{(1)}, 0_{(0)} \big)\,
  y^{\frac{1}{2}}
\\+\;{}&
  \frac{(-1)^{\frac{1-k}{2}} 2^{1-k} \pi^{\frac{1}{2}}}
       {\Gamma(\frac{1}{2}) \zeta(k)}\,
  \sum_{n = 1}^\infty
  n^{\frac{1}{2}-k}
  \sigma_{k-1}(n)\,
  (4 \pi n y)^{\frac{k-1}{2}}
  W_{\frac{1-k}{2}, \frac{k-1}{2}}(4 \pi n y) e(n x)
\\+\;{}&
  \frac{(-1)^{\frac{1-k}{2}} 2^{1-k} \pi^{\frac{1}{2}}}
       {\Gamma(k-\frac{1}{2}) \zeta(k)}
  \sum_{n = 1}^\infty
  n^{\frac{1}{2}-k}
  \sigma_{k-1}(n)\,
  (4 \pi n y)^{\frac{k-1}{2}}
  W_{\frac{k-1}{2}, \frac{k-1}{2}}(4 \pi n y) e(n x)
\tx{,}
\\[.5\baselineskip]
&
  \ov{c\big( E^\sk_k; \begin{psmatrix} 0 & 0 \\ 0 & 0 \end{psmatrix}_{(1)}, 0_{(1)} \big)}\,
  +
  \frac{2 \pi^k\, \Gamma(k-\frac{3}{2})}
       {\Gamma(k-1) \Gamma(k - \frac{1}{2})}\,
  \frac{\zeta(2k-3)}{\zeta(k) \zeta(2k-2)}\,
  \sum_{n = 1}^\infty
  \sigma_{k-2}(n)\,
  e(n \tau)
\tx{.}
\end{align*}

Proposition~\ref{prop:kohnen-limit-zero-m-semi-definite-convergence} asserts that $\Klim[0,0]^\sk_{k,0}(e^\sk_{k,0})$ is an elliptic Maa\ss\ form of weight~$1-k$ and eigenvalue $\frac{1}{2} (k - \frac{1}{2})$ with respect to the weight~$1-k$ hyperbolic Laplace operator. In particular, its $0$\thdash\ Fourier series coefficient is determined by nonzero ones. Comparing the first two series above with the Fourier series expansion of the elliptic Eisenstein series in~Lemma~\ref{la:elliptic-eisenstein-series-fourier-expansion} then gives the coefficients of index~$\begin{psmatrix} 0 & 0 \\ 0 & 0 \end{psmatrix}_{(0)}$ and $(\begin{psmatrix} 0 & 0 \\ 0 & 0 \end{psmatrix}_{(1)}, 0_{(0)})$.  Similarly, Proposition~\ref{prop:kohnen-limit-zero-m-semi-definite-convergence} also states that $\Klim[0,1]^\sk_{k,0}(e^\sk_{k,0})$ is a holomorphic elliptic modular form of weight~$k-1$. Comparing the third series above with the Fourier series expansion of the holomorphic Eisenstein series then provides the expression for the Fourier series coefficient of index~$(\begin{psmatrix} 0 & 0 \\ 0 & 0 \end{psmatrix}_{(1)}, 0_{(1)})$. This completes the second proof of Proposition~\ref{prop:fourier-jacobi-coefficient-zero-m-eisenstein-series-semi-definite-horizontal}.
\end{proof}

The next proposition addresses the remaining coefficients of index~$\begin{psmatrix} 0 & 0 \\ 0 & 0 \end{psmatrix}$. The proof exploits the fact that elliptic Maa\ss\ forms of weight~$0$ and spectral parameter~$(k-1)$ (i.e., eigenvalue $(k-1) (2-k)$ with respect to the hyperbolic Laplace operator) are unique up to multiplication by scalars, and are determined by their $0$\thdash\ Fourier series coefficient---see also Proposition~\ref{prop:multiplicity-one-for-skew-harmonic-fourier-coefficients-zero}.
\begin{proposition}
\label{prop:fourier-jacobi-coefficient-zero-m-eisenstein-series-semi-definite-vertial}
For $\nb \ne 0$, we have
\begin{gather}
  c\big( E^\sk_k;\, \begin{psmatrix} 0 & 0 \\ 0 & 0 \end{psmatrix}_{(1)}, \nb \big)
=
  \frac{(-1)^{\frac{1-k}{2}} 2^{2-k} \pi^{k-1}}
       {\Gamma(k - \frac{1}{2})}\,
  \frac{\zeta(k-1)}{\zeta(k) \zeta(2k-2)}\,
  |\nb|^{1-k} \sigma_{2k - 3}(|\nb|)
\tx{.}
\end{gather}
\end{proposition}
\begin{proof}
Observe that $E^\sk_k$ is invariant under the slash action of~$\rot(\GL{2}(\ZZ))$, and Proposition~\ref{prop:multiplicity-one-for-skew-harmonic-fourier-coefficients-zero} implies that $c(E^\sk_k;\, \begin{psmatrix} 0 & 0 \\ 0 & 0\end{psmatrix}_{(1)}, \bullet;\, \taub)$ is an elliptic Maa\ss\ form of weight~$0$ and eigenvalue $(k-1) (2-k)$ with respect to the hyperbolic Laplace operator. In particular, it is determined by its constant coefficient (see~\eqref{eq:prop:multiplicity-one-for-skew-harmonic-fourier-coefficients-zero:eisenstein}), which by Proposition~\ref{prop:fourier-jacobi-coefficient-zero-m-eisenstein-series-semi-definite-horizontal} equals 
\begin{multline*}
  \frac{(-1)^{\frac{1-k}{2}} 2^{2-k} \pi^{\frac{1}{2}} \Gamma(k - 1)}
       {\Gamma(k - \frac{1}{2})}\,
  \frac{\zeta(k - 1)}{\zeta(k)}\,
  a^\sk_k\big(\begin{psmatrix} 0 & 0 \\ 0 & 0 \end{psmatrix}_{(1)}, 0_{(0)}; Z \big)
\\
  +\,
  \frac{(-1)^{\frac{k-1}{2}} 2^{2-k} \pi\, \Gamma(k-\frac{3}{2})}
       {\Gamma(k - \frac{1}{2})}\,
  \frac{\zeta(k-1) \zeta(2k-3)}{\zeta(k) \zeta(2k-2)}\,
  a^\sk_k\big(\begin{psmatrix} 0 & 0 \\ 0 & 0 \end{psmatrix}_{(1)}, 0_{(1)}; Z \big)
\tx{.}
\end{multline*}
In other words, we have (see Lemma~\ref{la:elliptic-eisenstein-series-fourier-expansion})
\begin{gather*}
  c \big( E^\sk_k;\, \begin{psmatrix} 0 & 0 \\ 0 & 0\end{psmatrix}_{(1)}, \bullet;\, \taub \big)
=
  \frac{(-1)^{\frac{1-k}{2}} 2^{2-k} \pi^{\frac{1}{2}} \Gamma(k - 1)}
       {\Gamma(k - \frac{1}{2})}\,
  \frac{\zeta(k - 1)}{\zeta(k)}\,
  E_0(k-1, \taub)
\tx{.}
\end{gather*}
We finish the proof by applying~\eqref{eq:prop:multiplicity-one-for-skew-harmonic-fourier-coefficients-zero:eisenstein-coefficients} of Proposition~\ref{prop:multiplicity-one-for-skew-harmonic-fourier-coefficients-zero}.
\end{proof}

Finally, it remains to examine Fourier coefficients of indefinite index. We recover these coefficients from~$e^\sk_{k,-1}$ via the invariance of~$E^\sk_k$ under the action of $\rot(\GL{2}(\ZZ))$.
\begin{proposition}
\label{prop:fourier-jacobi-coefficient-zero-m-eisenstein-series-indefinite}
Let $T = \begin{psmatrix} n & r \slash 2 \\ r \slash 2 & 0 \end{psmatrix}$ be indefinite (i.e., \@ $r \ne 0$). Then
\begin{gather*}
  c\big( E^\sk_k; T \big)
=
  \sum_{d \isdiv \gcd(n,r)}
  d^{k-1}\, c\big( E^\sk_k;\, \begin{psmatrix} 0 & r \slash 2d \\ r \slash 2d & 1 \end{psmatrix} \big)
=
  \sum_{d \isdiv \gcd(n,r)}
  d^{k-1}\, c\big( E^\sk_k;\, \begin{psmatrix} 0 & r \slash 2d \\ r \slash 2d & -1 \end{psmatrix} \big)
\tx{.}
\end{gather*}
\end{proposition}
\begin{proof}
We prove the first equality, but omit the proof of the second one, since its proof is analogous. For simplicity, if $T = \begin{psmatrix} n & r \slash 2 \\ r \slash 2 & m \end{psmatrix}$ is indefinite, we write $c(n,r,m) := c(E^\sk_k;\, T)$ for the Fourier series coefficients of~$E^\sk_k$. If $n > 0$, then the invariance of~$E^\sk_k$ under the action of $\rot( \begin{psmatrix} 0 & 1 \\ 1 & 0 \end{psmatrix} )$ and the formula $e^\sk_{k,n} = e^\sk_{k,1} | V_n$ in Lemma~\ref{la:kohnen-limit-positive-m-eisenstein-series} in conjunction with~\eqref{eq:V-hecke-operator-coefficient-formula-abstract-fourier-jacobi-coefficient} imply that
\begin{gather*}
  c(n,r,0)
=
  c(0,r,n)
=
  \sum_{d \isdiv\gcd(r, n)}
  d^{k-1}
  c\big( 0, \frac{r}{d}, 1 \big)
\tx{.}
\end{gather*}
For general~$n$, since $r \ne 0$, the invariance of~$E^\sk_k$ under the Heisenberg subgroup of the embedded Jacobi group yields that 
\begin{gather*}
  c(n,r,0)
=
  c(n + 2 h r, r, 0)
=
  \sum_{d \isdiv \gcd(r, n + 2hr)}
  d^{k-1}
  c\big( 0, \frac{r}{d}, 1 \big)
=
  \sum_{d \isdiv \gcd(r, n)}
  d^{k-1}
  c\big( 0, \frac{r}{d}, 1 \big)
\tx{,}
\end{gather*}
where $h \in \ZZ$ is chosen in such a way that $n + 2h r > 0$.
\end{proof}

\subsection{The Hecke operator \texpdf{$V_0$}{V0}}
\label{ssec:hecke-operator-V-0}

In this section, we define an operator $V_0$ that maps $\wtd\phi_{1} \in \rmA\rmJ^\sk_{k,1}$ and $\wtd\phi_{-1} \in \rmA\rmJ^\sk_{k,-1}$ to abstract elliptic functions. It maps the $1$-st and $(-1)$-st Fourier-Jacobi coefficients of~$E^\sk_k$ to~$e^\sk_{k,0}$.

Let $\wtd\phi_1 \in \rmA\rmJ^\sk_{k,1}$ with Fourier series expansion
\begin{gather*}
  \wtd\phi_1(\tau, z, \tau')
\;=\;
  \sum_{n,r \in \ZZ}
  c\big(\wtd\phi_1;\, n, r \big)\,
  a^\sk_k\big( \begin{psmatrix} n & r \slash 2 \\ r \slash 2 & 1 \end{psmatrix} Z \big)
	\text{,}
\end{gather*}
and $\wtd\phi_{-1} \in \rmA\rmJ^\sk_{k,-1}$ with Fourier series expansion
\begin{multline*}
  \wtd\phi_1(\tau, z, \tau')
\;=\;
  \sum_{\substack{n,r \in \ZZ \\ 4n + r^2 \ne 0}}
  c\big(\wtd\phi_{-1};\, n, r \big)\,
  a^\sk_k\big( \begin{psmatrix} n & r \slash 2 \\ r \slash 2 & -1 \end{psmatrix} Z \big)
\\
  +\,
  \sum_{\substack{n,r \in \ZZ \\ 4n + r^2 = 0}}
  c\big(\wtd\phi_{-1};\, n_{(0)}, r \big)\,
  a^\sk_k\big( \begin{psmatrix} n & r \slash 2 \\ r \slash 2 & -1 \end{psmatrix}_{(0)} Z \big)
  \,+\,
  \sum_{\substack{n,r \in \ZZ \\ 4n + r^2 = 0}}
  c\big(\wtd\phi_{-1};\, n_{(1)}, r \big)\,
  a^\sk_k\big( \begin{psmatrix} n & r \slash 2 \\ r \slash 2 & -1 \end{psmatrix}_{(1)} Z \big)
\tx{.}
\end{multline*}
We set
\begin{multline}
\label{eq:def:V-0-skew-harmonic-positive-jacobi-index}
  \big( \wtd\phi_1 \big|^\sk_k\, V_0 \big)(Z)
\;:=\;
  \frac{(-1)^{\frac{k-1}{2}} \Gamma(\frac{1}{2}) \zeta(k)}{(2 \pi)^k}\,
  c\big(\wtd\phi_1;\, 0, 0 \big)\, \wtd{e}^\sk_{k,0}(Z)
\\
  +\;
  \sum_{\substack{n \in \ZZ \\ r \in \ZZ \setminus \{0\}}}
  \Big(
  \sum_{d \isdiv \gcd(n,r)}
  d^{k-1}
  c\big(\wtd\phi_1;\, 0, r \slash d \big)\,
  \Big)\,
  a^\sk_k\big( \begin{psmatrix} n & r \slash 2 \\ r \slash 2 & 0 \end{psmatrix} Z \big)
\end{multline}
and
\begin{multline}
\label{eq:def:V-0-skew-harmonic-negative-jacobi-index}
  \big( \wtd\phi_{-1} \big|^\sk_k\, V_0 \big)(Z)
\;:=\;
  \frac{(-1)^{\frac{1-k}{2}} \Gamma(k-\frac{1}{2}) \zeta(k)}{(2 \pi)^k}\,
  c\big(\wtd\phi_{-1};\, 0_{(0)}, 0 \big)\, \wtd{e}^\sk_{k,0}(Z)
\\
  +\;
  \sum_{\substack{n \in \ZZ \\ r \in \ZZ \setminus \{0\}}}
  \Big(
  \sum_{d \isdiv \gcd(n,r)}
  d^{k-1}
  c\big(\wtd\phi_{-1};\, 0, r \slash d \big)\,
  \Big)\,
  a^\sk_k\big( \begin{psmatrix} n & r \slash 2 \\ r \slash 2 & 0 \end{psmatrix} Z \big)
\tx{,}
\end{multline}
where
\begin{gather}
\label{eq:def:V-0-skew-harmonic-eisenstein-contribution}
  \wtd{e}^\sk_{k,0}(Z)
=
  \sum_{T = \begin{psmatrix} n & 0 \\ 0 & 0 \end{psmatrix}}
  c\big( E^\sk_k;\, T;\, Y \big) e(T X)
\tx{.}
\end{gather}
The series~\eqref{eq:def:V-0-skew-harmonic-positive-jacobi-index} and~\eqref{eq:def:V-0-skew-harmonic-negative-jacobi-index} converge absolutely, and their Fourier series coefficients are of moderate growth in the sense of~\eqref{eq:def:abstract-elliptic-function-skew-moderate-growth}. Observe that the appearance of~$\wtd{e}^\sk_{k,0}$ in~\eqref{eq:def:V-0-skew-harmonic-positive-jacobi-index} and~\eqref{eq:def:V-0-skew-harmonic-negative-jacobi-index} parallels the appearance of Eisenstein series in the definition of the Hecke-like operator~$V_0$ of~\cite{eichler-zagier-1985}.

We close this section with a remark.
\begin{remark}
From Proposition~\ref{prop:fourier-jacobi-coefficient-zero-m-eisenstein-series-indefinite}, we infer that
\begin{gather*}
  e^\sk_{k,0}
=
  e^\sk_{k,1} \big|^\sk_k\, V_0
=
  e^\sk_{k,-1} \big|^\sk_k\, V_0
\tx{.}
\end{gather*}

It is natural to ask whether $\wtd\phi_1 |^\sk_k\, V_0$ and~$\wtd\phi_{-1} |^\sk_k\, V_0$ are invariant under the embedded Jacobi group for general $\wtd\phi_1$ and $\wtd\phi_{-1}$.  In the case of Jacobi index~$-1$, this is easy to answer in the affirmative. Observe that $\cJ^\sk_{2-k,1}$ is spanned by the Jacobi-Eisenstein series $E^{\cJ\sk}_{2-k,1}$ in~\eqref{eq:def:almost-skew-harmonic-jacobi-eisenstein-series}. Note that
\begin{gather}
\label{eq:v0-kohnen-limit-process-negative}
  \Klim^\sk_{k,-1} :\,
  \rmA\rmJ^\sk_{k,-1}
\lra
  \cJ^\sk_{2-k,1}
\end{gather}
is injective by Proposition~\ref{prop:kohnen-limit-negative-m-injective}.  Moreover, Lemma~\ref{la:kohnen-limit-negative-m-eisenstein-series} provides a preimage of~$E^{\cJ\sk}_{2-k,1}$, and hence~\eqref{eq:v0-kohnen-limit-process-negative} is an isomorphism. The invariance of $e^\sk_{k,0}$ under the embedded Jacobi group then settles the case of
\begin{gather*}
  \rmA\rmJ_{k,-1} \big|^\sk_k\, V_0
=
  \lspan\, \CC e^\sk_{k,0}
\tx{.}
\end{gather*}

The case of Jacobi index~$1$ is more subtle and requires a detailed analysis, since $\rmJ^\sk_{k,1}$ is isomorphic to~$\rmA\rmJ^\sk_{k,1}$ by Proposition~\ref{prop:inverse-kohnen-limit-positive-m}, and since (in general) $\rmJ^\sk_{k,1}$ contains cusp forms. We postpone the answer to this question for~$m = 1$ to a sequel.
\end{remark}

\subsection{Proof of Theorem~\ref{mainthm:skew-maass-lift}}
\label{ssec:proof-of-skew-maass-lift}

Let $\phi \in \bbJ_{3-k,1}$, and recall from the statement of Theorem~\ref{mainthm:skew-maass-lift} that
\begin{gather}
\label{eq:def:repeated:skew-maass-lift-maass-jacobi}
  \MAsk_k(\phi)
\;=\;
  \sum_{m = 0}^\infty
  \Klim^{-1}_{k,-1}(\phi) \big|^\sk_k\, V_m
  \,-\,
  4 (k-2)!\,
  \sum_{m = 1}^\infty
  \Klim^{-1}_{k,1}(\xi^\rmJ_{3-k}\,\phi) \big|^\sk_k\, V_m
\tx{.}
\end{gather}

In Remark~\ref{rem:almost-skew-harmonic-jacobi-forms} we pointed out that $\bbJ_{3-k,1}$ is spanned by $E^\bbJ_{3-k,1}$. Hence it suffices to demonstrate that~\eqref{eq:def:repeated:skew-maass-lift-maass-jacobi} defines a skew-harmonic Maa\ss-Siegel form in the case that $\phi=E^\bbJ_{3-k,1}$. Specifically, we first prove that 
\begin{gather}
\label{eq:thm:skew-maass-lift:lift-equality}
  \MAsk_k\big( E^\bbJ_{3-k,1} \big)
\;=\;
  \frac{(-1)^{\frac{1-k}{2}} \Gamma(k-\frac{1}{2})\, \zeta(k)}{(2 \pi)^k}\,
  E^\sk_k
\end{gather}
by comparing Fourier-Jacobi coefficients.  We then complete the proof of Theorem~\ref{mainthm:skew-maass-lift} by showing that every~$F \in \rmM^\sk_k$ is a multiple of~$E^\sk_k$.

Let $m > 0$. Proposition~\ref{prop:kohnen-limit-positive-m-injective} asserts that $\Klim^\sk_{k,m}$ is injective.  Hence it suffices to show that the Kohnen limit processes of the $m$\thdash\ Fourier-Jacobi coefficients in~\eqref{eq:thm:skew-maass-lift:lift-equality} coincide. Lemma~\ref{la:kohnen-limit-positive-m-eisenstein-series} shows that $\Klim^\sk_{k,m}$ applied to the $m$\thdash\ Fourier-Jacobi coefficient on the right-hand side of~\eqref{eq:thm:skew-maass-lift:lift-equality} gives:
\begin{gather*}
  \frac{\Gamma(k-\frac{1}{2})}{\Gamma(\frac{1}{2})}\,
  E^\Jsk_{k,1} \big|^\Jsk_{k,1}\, V_m
\tx{.}
\end{gather*}

Applying $\Klim^\sk_{k,m}$ to the $m$\thdash\ Fourier-Jacobi coefficient of the left-hand side of~\eqref{eq:thm:skew-maass-lift:lift-equality} yields
\begin{multline*}
  \Klim^\sk_{k,m}\Big(
  \Klim^{-1}_{k,1}\Big(
  - 4 \Gamma(k-1)\,
  \xi^\rmJ_{3-k}\,E^\bbJ_{3-k,1}
  \Big) \Big|^\sk_k\, V_m
  \Big)
\\
=
  \Klim^\sk_{k,1}\Big(
  \Klim^{-1}_{k,1}\Big(
  - 4 \Gamma(k-1)\,
  \xi^\rmJ_{3-k}\,E^\bbJ_{3-k,1}
  \Big)
  \Big)
  \Big|^\Jsk_{k,1}\, V_m
=
  - 4 \Gamma(k-1)\,
  \xi^\rmJ_{3-k}\,E^\bbJ_{3-k,1}
  \big|^\Jsk_{k,1}\, V_m
\tx{.}
\end{multline*}
The first equality follows from the intertwining of $\Klim^\sk_{k,m}$ and $V_m$ in~\eqref{eq:kohnen-limit-positive-m-covariance-hecke-operators} of Proposition~\ref{prop:kohnen-limit-positive-m-covariance}, while the second equality is justified by Proposition~\ref{prop:inverse-kohnen-limit-positive-m}. Moreover,
\begin{gather*}
  - 4 \Gamma(k-1)\,
  \xi^\rmJ_{3-k}\,E^\bbJ_{3-k,1}
=
  - 4 \Gamma(k-1)\,
  \frac{-\pi^{-\frac{1}{2}}\,\Gamma(k-\frac{1}{2})}{4\Gamma(k-1)}\,
  E^\Jsk_{k,1}
=
  \frac{\Gamma(k-\frac{1}{2})}{\Gamma(\frac{1}{2})}\,
  E^\Jsk_{k,1}
\end{gather*}
by~\eqref{eq:almost-skew-harmonic-eisenstein-series-xi-image}, and we find that Fourier-Jacobi coefficients of positive index in~\eqref{eq:thm:skew-maass-lift:lift-equality} agree.

Let $m < 0$. It suffices again to compare Kohnen limit processes of Fourier-Jacobi coefficients. Lemma~\ref{la:kohnen-limit-negative-m-eisenstein-series} shows that $\Klim^\sk_{k,m}$ applied to the $m$\thdash\ Fourier-Jacobi coefficient on the right-hand side of~\eqref{eq:thm:skew-maass-lift:lift-equality} produces:
\begin{gather*}
  |m|^{k-1}\,
  E^\cJsk_{2-k,1} \big|_{2-k, 1} V_{|m|}
\tx{.}
\end{gather*}

Applying $\Klim^\sk_{k,m}$ to the $m$\thdash\ Fourier-Jacobi coefficient of the left-hand side of~\eqref{eq:thm:skew-maass-lift:lift-equality} leads to
\begin{multline*}
  \Klim^\sk_{k,m}\Big(
  \Klim^{-1}_{k,-1}\big( E^\bbJ_{3-k,1} \big) \big|^\sk_k\, V_m
  \Big)
=
  \Klim^\sk_{k,m}\Big(
  \Klim^{-1}_{k,-1}\big( E^\cJsk_{2-k,1} \big) \big|^\sk_k\, V_m
  \Big)
\\
=
  |m|^{k-1}
  \Klim^\sk_{k,1}\Big(
  \Klim^{-1}_{k,-1}\big( E^\cJsk_{2-k,1} \big)
  \Big)
  \big|^\Jsk_{2-k,1}\, V_m
=
  |m|^{k-1}
  E^\cJsk_{2-k,1}
  \big|^\Jsk_{2-k,1}\, V_m
	\text{.}
\end{multline*}
The first equality follows from~\eqref{eq:def:inverse-kohnen-limit-negative-m-maass-jacobi} and~\eqref{eq:almost-skew-harmonic-eisenstein-series-raising-image}, the second one from~\eqref{eq:kohnen-limit-negative-m-covariance-hecke-operators} in Proposition~\ref{prop:kohnen-limit-negative-m-covariance}, while the third equality is justified by Proposition~\ref{prop:inverse-kohnen-limit-negative-m}. We conclude that Fourier-Jacobi coefficients of negative index in~\eqref{eq:thm:skew-maass-lift:lift-equality} agree, too.

It remains to examine the $0$\thdash\ Fourier-Jacobi coefficients in~\eqref{eq:thm:skew-maass-lift:lift-equality}. Observe~\eqref{eq:def:inverse-kohnen-limit-negative-m} to find that
\begin{gather*}
  c\Big( \Klim^{-1}_{k,-1}\big( E^\cJsk_{2-k,1} \big);\, 0_{(0)}, 0 \Big)
=
  c\big( E^\cJsk_{2-k,1};\, 0_{(0)}, 0 \big)
=
  1
\tx{.}
\end{gather*}
Hence the contribution of the first summand in 
the defining equation~\eqref{eq:def:V-0-skew-harmonic-negative-jacobi-index} of~$V_0$, which appear in
\begin{gather*}
  \Klim^{-1}_{k,-1}(E^\bbJ_{3-k,1}) \big|^\sk_k\, V_0
=
  \Klim^{-1}_{k,-1}(E^\cJ_{3-k,1}) \big|^\sk_k\, V_0
\end{gather*}
on the left-hand side of~\eqref{eq:thm:skew-maass-lift:lift-equality}, matches the Fourier series coefficients of index~$\begin{psmatrix} n & 0 \\ 0 & 0 \end{psmatrix}$ on the right-hand side of~\eqref{eq:thm:skew-maass-lift:lift-equality}. Invariance of $E^\sk_k$ under~$\rot(\GL{2}(\ZZ))$ shows that the corresponding contribution of the second summand in~\eqref{eq:def:V-0-skew-harmonic-negative-jacobi-index} matches the Fourier series coefficients of index~$\begin{psmatrix} n & r \slash 2 \\ r \slash 2 & 0 \end{psmatrix}$ for~$r \ne 0$ on the right-hand side of~\eqref{eq:thm:skew-maass-lift:lift-equality}. This finishes our proof of~\eqref{eq:thm:skew-maass-lift:lift-equality}.

Finally, we show that every~$F \in \rmM^\sk_k$ is a multiple of~$E^\sk_k$. Consider $F= \sum_m \wtd\phi_m(\tau,z,\tau') \in \rmM^\sk_k$.  Replacing $F$  by
\begin{gather*}
  F
  \,-\,
  c\big(F;\, \begin{psmatrix} 0 & 0 \\ 0 & 0 \end{psmatrix}_{(0)} \big) E^\sk_k
\tx{,}
\end{gather*}
allows us to assume that the Fourier series coefficient of~$F$ of index~$\begin{psmatrix} 0 & 0 \\ 0 & 0 \end{psmatrix}_{(0)}$ vanishes. We then accomplish our goal by showing that all Fourier series coefficients of~$F$ vanish.

From Lemma~\ref{la:kohnen-limit-zero-m-semi-definite-fourier-coefficients}, we infer that the~$0$\thdash\ Fourier series coefficient of~$\Klim[0,0]^\sk_{k,0}(\wtd\phi_0)$ equals
\begin{gather*}
  c\big(F;\, \begin{psmatrix} 0 & 0 \\ 0 & 0 \end{psmatrix}_{(0)} \big)
  y^{k - \frac{1}{2}}
  \,+\,
  c\big(F;\, \begin{psmatrix} 0 & 0 \\ 0 & 0 \end{psmatrix}_{(1)} \big)
  y^{\frac{1}{2}}
=
  c\big(F;\, \begin{psmatrix} 0 & 0 \\ 0 & 0 \end{psmatrix}_{(1)} \big)
  y^{\frac{1}{2}}
\tx{.}
\end{gather*}
Proposition~\ref{prop:kohnen-limit-zero-m-semi-definite-convergence} asserts that~$\Klim[0,0]^\sk_{k,0}(\wtd\phi_0)$ is an elliptic Maa\ss\ form of weight~$1-k$ and eigenvalue $\frac{1}{2} (k-\frac{1}{2})$ with respect to the weight~$1-k$ hyperbolic Laplace operator. Theorem~31 of~\cite{maass-1964} shows that the corresponding space is one-dimensional and spanned by the Eisenstein series~$E_{1-k}(k-\frac{1}{2}, \tau)$ defined in~\eqref{eq:def:elliptic-eisenstein-series}. The $0$-th Fourier series coefficient of~$E_{1-k}(k-\frac{1}{2}, \tau)$ features the nonvanishing term~$y^{k-\frac{1}{2}}$. Hence $\Klim[0,0]^\sk_{k,0}(\wtd\phi_0)$ vanishes identically, and for all $n \in \ZZ_{< 0}$ we have that
\begin{gather*}
  c\big( F;\, \begin{psmatrix} n & 0 \\ 0 & 0 \end{psmatrix}_{(0)} \big)
=
  c\big( \Klim[0,0]^\sk_{k,0}(\wtd\phi_0);\, n \big)
=
  0
\tx{.}
\end{gather*}

Let $T = \begin{psmatrix} n & r \slash 2 \\ r \slash 2 & m \end{psmatrix}$ be degenerate with~$m < 0$. Then $T[U] = \begin{psmatrix} n' & 0 \\ 0 & 0 \end{psmatrix}$ for some $U \in \GL{2}(\ZZ)$. Since $m$ is negative, $T$ is negative semidefinite, and hence $n' < 0$. Moreover, $F$ (as a Siegel modular form) is invariant under the slash action of $\rot(U)$. Therefore, 
\begin{gather*}
 c(F;\, T_{(0)})
=
 \det(U)^{k-1}
 c\big(F;\, \begin{psmatrix} n' & 0 \\ 0 & 0 \end{psmatrix}_{(0)} \big)
=
  0
\tx{.}
\end{gather*}

Proposition~\ref{prop:kohnen-limit-negative-m-convergence} guarantees that $\Klim^\sk_{k,m}(\wtd\phi_m) \in \cJ^\sk_{2-k,|m|}$. To show that $\Klim^\sk_{k,m}(\wtd\phi_m)$ vanishes, it suffices by Lemma~\ref{la:almost-skew-harmonic-jacobi-form-vanishing} to ensure the vanishing of its Fourier series coefficients of index~$(n,r)$ with $4|m|n - r^2 = 0$. 
Any such $n, r$ yields a degenerate $T = \begin{psmatrix} n & r \slash 2 \\ r \slash 2 & m \end{psmatrix}$ as above, so that~\eqref{eq:la:kohnen-limit-negative-m-semi-definite-fourier-coefficients} implies that
\begin{gather*}
  c\big( \Klim^\sk_{k,m}(\wtd\phi_m);\, n_{(0)}, r \big)
=
  c(F;\, T_{(0)})
=
  0
\tx{.}
\end{gather*}
Thus, $\Klim^\sk_{k,m}(\wtd\phi_m)$ vanishes for all $m < 0$.

Let $T$ be an indefinite Fourier index. Then there exists a $U \in \GL{2}(\ZZ)$ such that the bottom-right entry~$m$ of~$T[U] = \begin{psmatrix} n & r \slash 2 \\ r \slash 2 & m \end{psmatrix}$ is negative. Since
\begin{gather*}
  c\big( \Klim^\sk_{k,m}(\wtd\phi_m);\, n, r \big)
=
  0
\tx{,}
\end{gather*}
Proposition~\ref{prop:kohnen-limit-negative-m-injective} implies that
\begin{gather*}
  c(F;\, T)
=
  \det(U)^{k-1}\,
  c(F;\, T[U])
=
  0
\end{gather*}
for all indefinite~$T$.

Consider $T = \begin{psmatrix} n & r \slash 2 \\ r \slash 2 & m \end{psmatrix}$ with~$m > 0$. Proposition~\ref{prop:kohnen-limit-positive-m-convergence} asserts that $\Klim^\sk_{k,m}(\wtd\phi_m)\in\rmJ^\sk_{k,m}$.  We just established that~$c(F;\, T) = 0$ for indefinite~$T$. Fourier series coefficients of~$F$ for positive definite~$T$ vanish by Proposition~\ref{prop:multiplicity-one-for-skew-harmonic-fourier-coefficients-nonzero-nonnegative}. Hence the Fourier series expansion of~$\Klim^\sk_{k,m}(\wtd\phi_m)$ is supported on $(n,r)$ with $4 m n - r^2 = 0$. The theta decomposition for skew-holomorphic Jacobi forms then implies that~$\Klim^\sk_{k,m}(\wtd\phi_m)$ vanishes for positive~$m$.

The vanishing of~$\Klim^\sk_{k,m}(\wtd\phi_m)$ for all $m < 0$ together with the injectivity of the Kohnen limit process in Proposition~\ref{prop:kohnen-limit-negative-m-injective} shows that
\begin{gather*}
  c\big( F;\, \begin{psmatrix} m & 0 \\ 0 & 0 \end{psmatrix}_{(1)} \big)
=
  (-1)^{k-1}\,
  c\big( F;\, \begin{psmatrix} 0 & 0 \\ 0 & m \end{psmatrix}_{(1)} \big)
=
  0
\tx{.}
\end{gather*}
Lemma~\ref{la:kohnen-limit-zero-m-semi-definite-fourier-coefficients} then implies that
\begin{gather*}
  \Klim[0,1]^\sk_{k,0}(\wtd\phi_0)
=
  \ov{c\big( F;\, \begin{psmatrix} 0 & 0 \\ 0 & 0 \end{psmatrix}_{(1)}, 0_{(1)} \big)}
  \,+\,
  \sum_{n = 1}^\infty
  (4 \pi n)^{k-2}
  \ov{c\big( F;\, \begin{psmatrix} -n & 0 \\ 0 & 0 \end{psmatrix}_{(1)} \big)}
  e(n\tau)
=
  \ov{c\big( F;\, \begin{psmatrix} 0 & 0 \\ 0 & 0 \end{psmatrix}_{(1)}, 0_{(1)} \big)}
\tx{.}
\end{gather*}
Recall from Proposition~\ref{prop:kohnen-limit-zero-m-semi-definite-convergence} that $\Klim[0,1]^\sk_{k,0}(\wtd\phi_0)$ is a holomorphic elliptic modular form of weight~$k-1 > 0$, and hence it must vanish. In particular, 
\begin{gather*}
  c\big( F;\, \begin{psmatrix} 0 & 0 \\ 0 & 0 \end{psmatrix}_{(1)}, 0_{(1)} \big)
=
  0
\tx{.}
\end{gather*}

It remains to show that $c(F;\, \begin{psmatrix} 0 & 0 \\ 0 & 0 \end{psmatrix}_{(1)}, \bullet;\, \taub) = 0$. Since it is an elliptic modular form of weight~$0$ and eigenvalue $(k-1) (2-k)$ with respect to the hyperbolic Laplace operator by Proposition~\ref{prop:multiplicity-one-for-skew-harmonic-fourier-coefficients-zero}, it is a multiple of the Eisenstein series~$E_0(k-1,\taub)$ by~\eqref{eq:prop:multiplicity-one-for-skew-harmonic-fourier-coefficients-zero:eisenstein}. Equation~\eqref{eq:prop:multiplicity-one-for-skew-harmonic-fourier-coefficients-zero:sub-fourier-expansion} shows that the $0$\thdash\ Fourier series coefficient of~$c(F;\, \begin{psmatrix} 0 & 0 \\ 0 & 0 \end{psmatrix}_{(1)}, \bullet;\, \taub)$ is given by
\begin{gather*}
  c\big( F;\, \begin{psmatrix} 0 & 0 \\ 0 & 0 \end{psmatrix}_{(1)}, 0_{(0)} \big)\,
  \yb^{k-1}
  \,+\,
  c\big( F;\, \begin{psmatrix} 0 & 0 \\ 0 & 0 \end{psmatrix}_{(1)}, 0_{(1)} \big)\,
  \yb^{2-k}
=
  0
\text{,}
\end{gather*}
and hence $c(F;\, \begin{psmatrix} 0 & 0 \\ 0 & 0 \end{psmatrix}_{(1)}, \bullet ;\, \taub) = 0$.

We conclude that~$F = 0$, which completes the proof of Theorem~\ref{mainthm:skew-maass-lift}.

\renewbibmacro{in:}{}
\renewcommand{\bibfont}{\normalfont\small\raggedright}
\renewcommand{\baselinestretch}{.8}

\Needspace*{4em}
\begin{multicols}{2}
\printbibliography[heading=none]%
\end{multicols}

\Needspace*{4\baselineskip}
\noindent
\rule{\textwidth}{0.15em}

{\noindent\small
Martin Raum\\
Chalmers tekniska högskola och G\"oteborgs Universitet,
Institutionen för Matematiska vetenskaper,
SE-412 96 Göteborg, Sweden\\
E-mail: \url{martin@raum-brothers.eu}\\
}%

\vspace{2ex}

{\noindent\small
Olav K. Richter\\
Department of Mathematics,
University of North Texas,
Denton, TX 76203,
USA\\
E-mail: \url{richter@unt.edu}
}

\end{document}

